\documentclass[11pt,a4paper]{article}
\usepackage{amsmath,amssymb,amsthm,mathrsfs}

\renewcommand{\mathcal}{\mathscr}
\newcommand{\cH}{\mathcal{H}}
\newcommand{\cF}{\mathcal{F}}

\renewcommand{\P}{\mathrm{P}}
\newcommand {\E}{{\mathrm E}}
\newcommand{\1}{{\bf 1}}
\newcommand{\dimh}{\dim_{_{\rm H}}}

\newcommand{\R}{\mathbb{R}}
\newcommand{\e}{\epsilon}

\newtheorem{stat}{Statement}[section]
\newtheorem{prop}[stat]{Proposition}
\newtheorem{cor}[stat]{Corollary}
\newtheorem{thm}[stat]{Theorem}
\newtheorem{lem}[stat]{Lemma}

\theoremstyle{definition} \newtheorem{remark}[stat]{Remark}
\numberwithin{equation}{section}

\begin{document}
\title{ \bf Hitting probabilities for systems of non-linear stochastic heat equations in spatial dimension $k\geq 1$}

\author{Robert C. Dalang$^{1,4}$, Davar Khoshnevisan$^{2,5}$ and Eulalia Nualart$^3$}

\date{}

\maketitle

\footnotetext[1]{Institut de Math\'ematiques, Ecole Polytechnique
F\'ed\'erale de Lausanne, Station 8, CH-1015 Lausanne, Switzerland. \texttt{robert.dalang@epfl.ch}}

\footnotetext[2]{Department of Mathematics, The
University of Utah, 155 S. 1400 E. Salt Lake City,
UT 84112-0090, USA. \texttt{davar@math.utah.edu},
\texttt{http://www.math.utah.edu/$\sim$davar/}}

\footnotetext[3]{Institut Galil\'ee, Universit\'e Paris 13, 93430 Villetaneuse, France. \texttt{eulalia@nualart.es},
\texttt{http://nualart.es}}

\footnotetext[4]{Supported in part by the Swiss National Foundation for Scientific Research.}

\footnotetext[5]{Research supported in part by a grant
from the US National Science Foundation.}

\maketitle
\begin{abstract}
We consider a system of $d$ non-linear stochastic heat equations in spatial
dimension $k \geq 1$, whose solution is an $\R^d$-valued random field $u= \{u(t\,,x),\, (t,x) \in \R_+ \times \R^k\}$. The $d$-dimensional driving noise is white in time and with a spatially homogeneous covariance defined as a Riesz kernel with exponent $\beta$, where $0<\beta<(2 \wedge k)$. The non-linearities appear
both as additive drift terms and as multipliers of the noise. Using techniques of
Malliavin calculus, we establish an upper bound on the two-point density, with respect to Lebesgue measure, of the $\R^{2d}$-valued random vector
$(u(s,y),u(t,x))$, that, in particular, quantifies how this density degenerates as
$(s,y) \to (t,x)$. From this result, we deduce a lower bound on hitting
probabilities of the process $u$, in terms of Newtonian capacity. 
We also establish an upper bound on hitting
probabilities of the process in terms of Hausdorff measure. 
These estimates make it possible to show that points are polar when $d > \frac{4+2k}{2-\beta}$ and are not polar when $d < \frac{4+2k}{2-\beta}$. 
In the first case, we also show that the
Hausdorff dimension of the range of the process is $\frac{4+2k}{2-\beta}$ a.s.
\end{abstract}

\vskip 1cm \noindent {\em Abbreviated title:}  Hitting probabilities

 \vskip 10pt

 \vskip 6cm \pagebreak

\section{Introduction and main results}

Consider the following system of stochastic partial differential equations:
\begin{equation} \label{equa1}
\left\{
\begin{array}{l}{\displaystyle
\frac{\partial}{\partial t} u_i (t,x)= \frac{1}{2}\, \Delta_x u_i(t,x) +\sum_{j=1}^d \sigma_{i,j}(u(t,x)) \, \dot{F}^j(t,x) + b_i(u(t,x)), } \\
u_i(0,x)=0, \qquad\qquad\qquad\qquad\qquad\qquad i \in \{1,...,d\}, 
\end{array}\right.
\end{equation}
$t \geq 0$, $x \in \R^k$, $k \geq 1$, $\sigma_{i,j}, b_i : \R^d \rightarrow \R$ are globally Lipschitz functions, $i,j \in \{1,\dots,d\}$, and the $\Delta_x$ denotes the Laplacian in the spatial variable $x$.

The noise $\dot{F}=(\dot{F}^1,...,\dot{F}^d)$ is a spatially homogeneous centered Gaussian generalized random field with covariance of the form
\begin{equation}\label{noise}
\E \, [\dot{F}^i(t,x) \dot{F}^j(s,y)]= \delta(t-s) \Vert x-y \Vert^{-\beta} \delta_{ij}, \qquad 0<\beta< (2 \wedge k).
\end{equation}
Here, $\delta(\cdot)$ denotes the Dirac delta function, $\delta_{ij}$ the Kronecker symbol and $\Vert \cdot \Vert$ is the Euclidean norm. In particular, the $d$-dimensional driving noise $\dot{F}$ is white in time and with a spatially homogeneous covariance given by the Riesz kernel $f(x)=\Vert x \Vert^{-\beta}$.

The solution $u$ of \eqref{equa1} is known to be a $d$-dimensional random field (see Section \ref{sec2}, where precise definitions and references are given), and the aim of this paper is to develop potential theory for $u$. In particular, given a set $A \subset \R^d$, we want to determine whether or not the process $u$ hits $A$ with positive probability. For systems of linear and/or nonlinear stochastic heat equations in spatial dimension $1$ driven by a $d$-dimensional space-time white noise, this type of question was studied in {\sc Dalang, Khoshnevisan, and Nualart} \cite{Dalang:05} and \cite{Dalang2:05}. For systems of linear and/or nonlinear stochastic wave equations, this was studied first in {\sc Dalang and Nualart} \cite{Dalang:04} for the reduced wave equation in spatial dimension $1$, and in higher spatial dimensions in {\sc Dalang and Sanz-Sol\'e} \cite{Dalang:10,Dalang:11}. The approach of this last paper is used for some of our estimates (see Proposition \ref{large}).

We note that for the Gaussian random fields, and, in particular, for \eqref{equa1} when $b \equiv 0$ and $\sigma = I_d$, the $d \times d$-identity matrix, there is a well-developed potential theory \cite{Bierme:09, Xiao:09}. The main effort here concerns the case where $b$ and/or $\sigma$ are not constant, in which case $u$ is {\em not} Gaussian.

Let us introduce some notation concerning potential theory. For all Borel sets $F\subseteq\R^d$, let $\mathcal{P}(F)$ denote the set of all probability measures with compact support in $F$. For all $\alpha \in \R$ and $\mu\in\mathcal{P}(\R^k)$, we let $I_\alpha(\mu)$ denote the \emph{$\alpha$-dimensional energy} of $\mu$, that is,
\begin{equation*}
I_\alpha(\mu) := \iint {\rm K}_\alpha(\|x-y\|)\, \mu(dx)\,\mu(dy),
\end{equation*}
where
\begin{equation} \label{k}
{\rm K}_\alpha(r) :=
\begin{cases}
r^{-\alpha}&\text{if $\alpha >0$},\\
\log ( N_0/r ) &\text{if $\alpha =0$},\\
1&\text{if $\alpha<0$},
\end{cases}
\end{equation}
where $N_0$ is a constant whose value will be specified later (at the end of the proof of Lemma \ref{prel}).

For all $\alpha\in\R$ and Borel sets $F\subset\R^k$, $\text{Cap}_\alpha(F)$ denotes the \emph{$\alpha$-dimensional capacity of $F$}, that is,
\begin{equation*}
\text{Cap}_\alpha(F) := \left[ \inf_{\mu\in\mathcal{P}(F)}
I_\alpha(\mu) \right]^{-1},
\end{equation*}
where, by definition, $1/\infty:=0$.

Given $\alpha\geq 0$, the $\alpha$-dimensional \emph{Hausdorff measure}
of $F$ is defined by
\begin{equation}\label{eq:HausdorfMeasure}
{\mathcal{H}}_ \alpha (F)= \lim_{\epsilon \rightarrow 0^+} \inf
\left\{ \sum_{i=1}^{\infty} (2r_i)^ \alpha : F \subseteq
\bigcup_{i=1}^{\infty} B(x_i\,, r_i), \ \sup_{i\ge 1} r_i \leq
\epsilon \right\},
\end{equation}
where $B(x\,,r)$ denotes the open (Euclidean) ball of radius
$r>0$ centered at $x\in \R^d$. When $\alpha <0$, we define
$\mathcal{H}_\alpha (F)$ to be infinite.

\vskip 12pt

Consider the following hypotheses on the coefficients of the system of equations (\ref{equa1}), which are common assumptions when using Malliavin calculus:
\begin{itemize}
\item[{\bf P1}] The functions $\sigma_{i,j}$ and $b_i$
are $C^{\infty}$ have bounded
partial derivatives of all positive orders,  and the $\sigma_{i,j}$ are bounded, $i,j \in\{1,\dots,d\}$.
\item[{\bf P2}] The matrix $\sigma=(\sigma_{i,j})_{1 \leq i,j \leq d}$ is {\em strongly elliptic,} that is,
$\Vert \sigma(x) \cdot \xi \Vert^2 \geq \rho^2 >0$ (or,
equivalently, since $\sigma$ is a square matrix, $\Vert \xi^{\sf T}
\cdot \sigma(x) \Vert^2 \geq \rho^2
>0$) for some $\rho>0$, for all $x \in \mathbb{R}^d$, $\xi \in
\mathbb{R}^d$, $\Vert \xi \Vert=1$.
\end{itemize}

\begin{remark}
Note that because $\sigma$ is a square matrix,
\begin{equation*}
\inf_{x \in \mathbb{R}^d} \inf_{\Vert \xi \Vert=1} \Vert \sigma(x)
\cdot \xi \Vert^2= \inf_{x \in \mathbb{R}^d} \inf_{\Vert \xi
\Vert=1} \Vert \xi^{\sf T} \cdot \sigma(x) \Vert^2.
\end{equation*}
However, for non square matrices, this equality is false in general.
\end{remark}

For $T>0$ fixed, we say that $I \times J \subset (0,T] \times \mathbb{R}^k$ is a closed
non-trivial rectangle if $I \subset (0,T]$ is a closed non-trivial
interval and $J$ is of the form $[a_1,b_1]\times \cdots \times
[a_k,b_k]$, where $a_i,b_i \in \mathbb{R}$ and $a_i<b_i$,
$i=1,...,k$. \vskip 12pt The main result of this article is the
following.
\begin{thm} \label{t1}
Let ${u}$ denote the solution of \textnormal{(\ref{equa1})}.
Assume conditions \textnormal{{\bf P1}} and \textnormal{{\bf P2}}. Fix $T>0$ and let $I \times J \subset (0,T] \times \mathbb{R}^k$ be a closed non-trivial rectangle. Fix $M>0$ and $\eta>0$.
\begin{itemize}
\item[\textnormal{(a)}] There exists $C>0$ such that for all compact sets
$A\subseteq[-M\,,M]^d$,
\begin{equation*}
\P\left\{ { u}(I \times J) \cap A
\neq\varnothing \right\} \le C \, \mathcal{H}_{
d-(\frac{4+2k}{2-\beta})-\eta}(A).
\end{equation*}

\item[\textnormal{(b)}] There exists $c>0$ such that for all compact sets $A\subseteq[-M\,,M]^d$,
\begin{equation*}
\P\left\{ { u}(I \times J) \cap A
\neq\varnothing \right\} \ge c \, \textnormal{Cap}_{d-(\frac{4+2k}{2-\beta})+\eta}(A) .
\end{equation*}
\end{itemize}
\end{thm}

As a consequence of Theorem~\ref{t1}, we deduce the following result on the polarity of points. Recall that $A$ is a {\em polar set} for $u$ if $P\{u(I\times J) \cap A \neq \varnothing \}=0$, for any $I\times J$ as in Theorem~\ref{t1}.
\begin{cor} \label{c33}
Let ${u}$ denote the solution of \textnormal{(\ref{equa1})}. Assume \textnormal{{\bf P1}} and \textnormal{{\bf P2}}. Then points are not polar for $u$ when
$d < \frac{4+2k}{2-\beta}$,
and are polar when $d > \frac{4+2k}{2-\beta}$ (if $\frac{4+2k}{2-\beta}$ is an integer, then the case $d=\frac{4+2k}{2-\beta}$ is open).
\end{cor}

Another consequence of Theorem~\ref{t1} is the Hausdorff dimension of the range
of the process $u$.
\begin{cor} \label{c3bis}
Let ${u}$ denote the solution of \textnormal{(\ref{equa1})}. Assume \textnormal{{\bf P1}} and \textnormal{{\bf P2}}. 
If $d > \frac{4+2k}{2-\beta}$, then a.s.,
\begin{equation*}
\dimh (u(\mathbb{R}_+ \times \R^k)) = \frac{4+2k}{2-\beta}.
\end{equation*}
\end{cor}

The result of Theorem \ref{t1} can be compared to the best result available for the Gaussian case, using the result of \cite[Theorem 7.6]{Xiao:09}.

\begin{thm}\label{t1'} Let $v$ denote the solution of \eqref{equa1} when $ b \equiv 0$ and $\sigma \equiv I_d$. Fix $T, M>0$ and let $I \times J \subset (0,T] \times \mathbb{R}^k$ be a closed non-trivial rectangle. There exists $c>0$ such that for all compact sets $A\subseteq[-M\,,M]^d$,
\begin{equation*}
 c^{-1} \, \textnormal{Cap}_{d-(\frac{4+2k}{2-\beta})}(A) 
  \le \P\left\{ { v}(I \times J) \cap A \neq\varnothing \right\}  
  \le c \, \mathcal{H}_{d-(\frac{4+2k}{2-\beta})}(A).
\end{equation*}
\end{thm}

   Theorem \ref{t1'} is proved in Section \ref{sec2}. Comparing Theorems \ref{t1} and \ref{t1'}, we see that Theorem \ref{t1} is nearly optimal. 

In order to prove Theorem~\ref{t1}, we shall use techniques of Malliavin calculus in order to establish first the following result.
Let $p_{t,x}(z)$ denote the probability density function of the $\R^d$-valued random vector
$u(t,x) = (u_1(t,x), \ldots, u_d(t,x))$ and for $(s,y) \not= (t,x)$, let $p_{s,y; \, t,x}(z_1,z_2)$ denote the joint density function of the $\R^{2d}$-valued random vector
\begin{equation*}
(u(s,y), u(t,x)) = (u_1(s,y), \ldots, u_d(s,y), u_1(t,x), \ldots, u_d(t,x)).
\end{equation*}
The existence (and smoothness) of $p_{t,x}(\cdot)$ when $d=1$
follows from \cite[Theorem 2.1]{Marquez:01} and Lemma \ref{scalar}
(see also \cite[Theorem 6.2]{Nualart:07}). The extension of this
fact to $d \geq 1$ is proved in Proposition \ref{unif}. The
existence (and smoothness) of $p_{s,y; \, t,x}(\cdot, \cdot)$ is a
consequence of Theorem \ref{ga} and \cite[Thm.2.1.2 and
Cor.2.1.2]{Nualart:95}. \vskip 12pt The main technical effort in
this paper is the proof of the following theorem.
\begin{thm} \label{t2}
Assume \textnormal{{\bf P1}} and \textnormal{{\bf P2}}. Fix $T>0$
and let $I \times J \subset (0,T] \times \mathbb{R}^k$ be a closed
non-trivial rectangle.
\begin{itemize}
\item[\textnormal{(a)}] The density $p_{t,x}(z)$ is a $C^\infty$ function of $z$ and is uniformly bounded
over $z \in \R^d$ and $(t,x)\in I \times J$.

\item[\textnormal{(b)}] For all $\eta>0$ and $\gamma \in (0,2-\beta)$, there exists $c>0$ such that for any $(s,y), (t,x) \in I \times J$, $(s,y) \neq (t,x)$, $z_1,z_2 \in \mathbb{R}^d$, and $p\geq 1$,
\begin{equation}\label{eq1.9}
p_{s,y;\,t,x}(z_1,z_2) \leq c (|t-s|^{\frac{2-\beta}{2}}+\Vert x-y
\Vert^{2-\beta})^{-(d+\eta)/2}
\biggl[\frac{|t-s|^{\gamma/2}+\Vert x-y \Vert^{\gamma}}{\Vert z_1-z_2
\Vert^2}\, \wedge 1 \biggr]^{p/(2d)}.
\end{equation}
\end{itemize}
\end{thm}

Statement (a) of this theorem is proved at the end of Section \ref{sec4}, and statement (b) is proved in Section \ref{sec53}. 

\begin{remark} \label{remexp}
   (a) Theorem \ref{t2}(a) remains valid under a slightly weaker version of \textnormal{{\bf P1}}, in which the $\sigma_{i,j}$ need not be bounded (but their derivatives of all positive orders are bounded). 

   (b) The last factor on the right-hand side of \eqref{eq1.9} is similar to the one obtained in \cite[Remark 3.1]{Dalang:11}, while in the papers \cite{Dalang:05,Dalang2:05}, which concern spatial dimension $1$, it was replaced by
$$
\exp \biggl( -\frac{\Vert z_1-z_2
\Vert^2}{c (|t-s|^{\gamma/2}+\Vert x-y \Vert^{\gamma})}\biggr).
$$
This exponential factor was obtained by first proving this bound in the case where $b_i \equiv 0$, $i=1,\dots,d$, and then using Girsanov's theorem. In the case of higher spatial dimensions that we consider here, we can obtain this same bound when $b_i \equiv 0$, $i=1,\dots,d$ 
(see Lemma \ref{exp} in Section \ref{sec53}). Since there is no applicable Girsanov's theorem in higher spatial dimensions and for equations on all of $\R^d$, we establish \eqref{eq1.9} and, following \cite{Dalang:11}, show in Section \ref{sec23} that this estimate is sufficient for our purposes.
\end{remark}

One further fact about $p_{t,x}(\cdot)$ that we will need is provided by the following recent result of {\sc E.~Nualart} \cite{Nualart:09}.
\begin{thm}\label{te2}
Assume \textnormal{{\bf P1}} and \textnormal{{\bf P2}}. Fix $T>0$ and let $I \times J \subset (0,T] \times \mathbb{R}^k$ be a closed non-trivial rectangle. Then for all $z \in \R^d$ and $(t,x)\in (0,T] \times \R^k$, the density $p_{t,x}(z)$ is strictly positive.
\end{thm}

\section{Proof of Theorems \ref{t1}, \ref{t1'} and Corollaries \ref{c33}, \ref{c3bis} (assuming Theorem \ref{t2})}\label{sec2}

We first define precisely the driving noise that appears in \eqref{equa1}. Let $\mathcal{D}(\R^{k})$ be the space of $C^\infty$ test-functions with compact support. Then $F=\{ F(\phi) =(F^1(\phi),\dots,F^d(\phi)),\ \phi \in \mathcal{D}(\R^{k+1})\}$
is an $L^2(\Omega, \mathcal{F}, \P)^d$-valued mean zero Gaussian process with covariance
\begin{equation*}
\E [F^i(\phi) F^j(\psi) ]
=\delta_{ij}\ \int_{\R_+} dr \int_{\R^k} dy \int_{\R^k} dz \, \phi(r,y) \Vert y-z \Vert^{-\beta} \psi(r,z).
\end{equation*}
Using elementary properties of the Fourier transform (see {\sc Dalang} \cite{Dalang:99}), this covariance can also be written as
\begin{equation*}
\E [F^i(\phi) F^j(\psi) ]
=\delta_{ij}\ c_{k,\beta} \ \int_{\R_+} dr \int_{\R^k} d\xi \,\Vert \xi \Vert^{\beta-k} \mathcal{F} \phi(r,\cdot)(\xi) \overline{\mathcal{F} \psi(r,\cdot) (\xi)},
\end{equation*}
where $c_{k,\beta}$ is a constant and $\mathcal{F} f(\cdot)(\xi)$ denotes the Fourier transform of $f$, that is,
\begin{equation*}
\mathcal{F} f(\cdot)(\xi)=\int_{\R^k} e^{- 2 \pi i \xi \cdot x} \, f(x) \, dx.
\end{equation*}

Since equation (\ref{equa1}) is formal, we first provide, following {\sc Walsh} \cite[p.289-290]{Walsh:86}, a rigorous formulation of \eqref{equa1} through the notion of {\em mild solution} as follows. Let $M=(M^1,...,M^d)$, $M^i=\{ M^i_t(A), t \geq 0, A \in \mathcal{B}_b(\R^k)\}$ be the
$d$-dimensional worthy martingale measure obtained as an extension of the process $\dot{F}$ as in {\sc Dalang and Frangos} \cite{DF98}. Then a {\em mild solution} of (\ref{equa1}) is a jointly measurable
$\R^d$-valued process $u=\{u(t,x),\ t\geq 0,\, x \in \R^k\}$, adapted to the natural filtration generated by $M$, such that
\begin{equation} \label{equa2} \begin{split}
u_i(t,x)&=\int_0^t \int_{\R^k} S(t-s,x-y) \sum_{j=1}^d \sigma_{i,j}(u(s,y)) M^j(ds, dy) \\
& \qquad\qquad+ \int_0^t ds \int_{\R^k} dy \, S(t-s,x-y) \, b_i(u(s,y)),\qquad i \in \{1,...,d\},
\end{split}
\end{equation}
where $S(t,x)$ is the fundamental solution of the deterministic heat equation in $\R^k$, that is,
\begin{equation*}
S(t,x)=(2\pi t)^{-k/2} \exp \biggl(-\frac{\Vert x \Vert^2}{2t} \biggr),
\end{equation*}
and the stochastic integral is interpreted in the sense of \cite{Walsh:86}. We note that the covariation measure of $M^i$ is
$$
   Q([0,t] \times A \times B) = \langle M^i(A), M^i(B) \rangle_t =
   t \int_{\R^k} dx \int_{\R^k} dy\, 1_A(x)\, \Vert x-y\Vert^{-\beta}\, 1_B(y),
$$
and its dominating measure is $K\equiv Q$. In particular,
\begin{align}\nonumber
   &E\left[\left(\int_0^t  \int_{\R^k} S(t-s,x-y)\, M^i(ds,dy)\right)^2 \right]\\ \nonumber
   &\qquad = \int_0^t ds  \int_{\R^k} dy  \int_{\R^k} dz\, S(t-s,x-y) \, \Vert y-z\Vert^{-\beta}\, S(t-s,x-z) \\
   &\qquad = c_{k,\beta} \int_0^t ds  \int_{\R^k} d\xi\, \Vert\xi\Vert^{\beta-k}\, \vert \cF S(t-s,\cdot)(\xi)\vert^2, 
\label{eqdom}
\end{align}
where we have used elementary properties of the Fourier transform (see also {\sc Dalang} \cite{Dalang:99}, {\sc Nualart and Quer-Sardanyons} \cite{Nualart:07}, and {\sc Dalang and Quer-Sardanyons} \cite{DQ10} for properties of the stochastic integral).
This last formula is convenient since
\begin{equation} \label{ft}
\cF S(r,\cdot) (\xi) = \exp(-2\pi^2\, r \Vert \xi \Vert^2).
\end{equation}

The existence and uniqueness of the solution of (\ref{equa1}) is studied in {\sc Dalang} \cite{Dalang:99} for general space correlation functions $f$
which are non-negative, non-negative definite and continuous on $\R^k \setminus \{0\}$ (in the case where $k=1$; for these properties, the extension to $k>1$ is straightforward). In particular, it is proved that if the {\em spectral measure} of $\dot F$, that is, the non-negative tempered measure $\mu$ on $\R^k$ such that $\cF \mu = f$, satisfies
\begin{equation} \label{equa33}
\int_{\R^k} \frac{\mu(d \xi) }{1+\Vert \xi \Vert^2} < +\infty,
\end{equation}
then there exists a unique solution of (\ref{equa1}) such that $(t,x) \mapsto u(t,x)$ is $L^2$-continuous, and condition \eqref{equa33} is also necessary for existence of a mild solution.

In the case of the noise (\ref{noise}), $f(x) = \Vert x \Vert^{-\beta}$ and $\mu(d \xi)= c_d \Vert \xi \Vert^{\beta-k}\, d \xi$, where $c_d$ is a constant (see {\sc Stein} \cite[Chap.V, Section 1, Lemma 2(b)]{Stein:70}), and the condition \eqref{equa33} is equivalent to
\begin{equation}
0 < \beta < (2 \wedge k).
\end{equation}
Therefore, by {\sc Dalang} \cite{Dalang:99}, there exists a unique $L^2$-continuous solution of (\ref{equa1}), satisfying
\begin{equation*}
\sup_{(t,x) \in [0,T] \times \R^k} \E \, \biggl[\vert u_i(t,x) \vert^p \biggr]< +\infty, \qquad i \in \{1,...,d\},
\end{equation*}
for any $T>0$ and $p \geq 1$.

\subsection{H\"older continuity of the solution}

Let $T>0$ be fixed. In {\sc Sanz-Sol\'e and Sarr\`a} \cite[Theorem 2.1]{Sanz:02} it is proved that for any $\gamma \in (0, 2-\beta)$, $s,t \in [0,T]$, $s \leq t$, $x, y \in \mathbb{R}^k$, $p>1$,
\begin{equation} \label{equa3}
\E \, [\Vert u(t,x)-u(s,y) \Vert^p] \leq C_{\gamma,p,T} (\vert t-s\vert^{\gamma/2}
+\Vert x-y\Vert^{\gamma})^{p/2}.
\end{equation}
In particular, the trajectories of $u$ are a.s.~$\gamma/4$-H\"older continuous in $t$ and $\gamma/2$-H\"older continuous in $x$.

\vskip 12pt
The next result shows that the estimate \eqref{equa3} is nearly optimal (the only possible improvement would be to include the value $\gamma = 2 - \beta$).

\begin{prop} \label{holderoptimal}
Let $v$ denote the solution of \textnormal{(\ref{equa1})} with $\sigma \equiv 1$ and $b \equiv 0$. Then for any $0< t_0 <T$, $p>1$ and $K$ a compact set, there exists $c_1=c_1(p,t_0,K)>0$ such that for any $t_0 \leq s \leq t \leq T$, $x, y \in K$, $i \in \{1,...,d\}$,
\begin{equation} \label{equa4}
\E \, [\vert v_i(t,x)-v_i(s,y) \vert^p] \geq c_1 \left(\vert t-s \vert^{\frac{2-\beta}{4}p}+ \Vert x-y\Vert^{\frac{2-\beta}{2}p}\right).
\end{equation}
\end{prop}

\begin{proof}
Since $v$ is Gaussian, it suffices to check (\ref{equa4}) for $p=2$. Setting $t=s+h$ and $x=y+z$, we observe from \eqref{eqdom} that
$$
   \E \, [\vert v_i(s+h,y+z)-v_i(s,y) \vert^2] = c_{k,\beta}\, (I_1 + I_2),
$$
where
\begin{align*}
  I_1 &=\int_s^{s+h} dr \int_{\R^k} d\xi \,\Vert \xi \Vert^{\beta-k}\, \,\vert \mathcal{F} S(s+h-r,\cdot)(\xi)\vert^2 \\
  I_2 &= \int_0^{s} dr \int_{\R^k} d\xi \,\Vert \xi \Vert^{\beta-k}\, \,\vert \mathcal{F} S(s+h-r,\cdot)(\xi)\, e^{-2\pi i \xi \cdot(y+z)}- \mathcal{F} S(s-r,\cdot)(\xi)\, e^{-2\pi i \xi \cdot y}\vert^2.
\end{align*}
\vskip 12pt

\noindent{\em Case 1.} $h \geq \Vert z \Vert^2$. In this case, we notice from \eqref{ft} that
\begin{align*}
   I_1 + I_2 \geq I_1 &= \int_s^{s+h} dr \int_{\R^k} d \xi \, \Vert \xi \Vert^{\beta-k}
\exp(-4 \pi^2 (s+h-r) \Vert \xi \Vert^2 ) \\ \nonumber
& = \int_{\R^k} d \xi \, \Vert \xi \Vert^{\beta-k} \biggl(\frac{1-\exp(-4 \pi^2 h \Vert \xi \Vert^2)}{4 \pi^2\Vert \xi \Vert^2}\biggr).
\end{align*}
We now use the change of variables $\tilde{\xi}=h^{1/2} \xi$ to see that the last right-hand side is equal to
\begin{equation*}
h^{\frac{2-\beta}{2}} \int_{\R^k} d \tilde{\xi} \, \Vert \tilde{\xi} \Vert^{\beta-k} \biggl(\frac{1-\exp(-4 \pi^2 \Vert \tilde{\xi} \Vert^2)}{4 \pi^2\Vert \tilde{\xi} \Vert^2}\biggr).
\end{equation*}
Note that the last integral is positive and finite. Therefore, when $h \geq \Vert z \Vert^2$,
$$
   \E \, [\vert v_i(t+h,y+z)-v_i(s,y) \vert^2] \geq c \left(\max\left(h,\Vert z \Vert^2 \right)\right)^{\frac{2-\beta}{2}}.
$$

\noindent{\em Case 2.} $\Vert z \Vert^2 \geq h$. In this case, we notice that
\begin{align*}
  I_1 + I_2 &\geq I_2 \\
  &= \int_0^s dr \int_{\R^k} d \xi \, \Vert \xi \Vert^{\beta-k} \exp(-4 \pi^2 (s-r) \Vert \xi \Vert^2 )\, \left\vert 1- \exp(-4\pi^2 h \Vert \xi \Vert^2)\exp(-2 \pi i\, \xi \cdot z )\right\vert^2.
\end{align*}
We use the elementary inequality $\vert 1 - re^{i\theta}\vert \geq \frac{1}{2} \vert 1 - e^{i\theta}\vert$, valid for all $r \in [0,1]$ and $\theta \in \R$, and we calculate the $dr$-integral, to see that
$$ 
   I_2 \geq \int_{\R^k} d \xi \, \Vert \xi \Vert^{\beta-k} \biggl(\frac{1-\exp(-4 \pi^2 s \Vert \xi \Vert^2)}{4 \pi^2 \Vert \xi \Vert^2}\biggr)
\, \vert 1-\exp(-2 \pi i\, \xi \cdot z) \vert^2.
$$
Because $z \in K - K$ and $K$ is compact, fix $M>0$ such that $\Vert z \Vert \leq M$.
When $ z \neq 0$, we use the change of variables $\tilde{\xi}=\Vert z \Vert \xi$ and write $e=z / \Vert z \Vert$ to see that the last right-hand side is equal to
\begin{align*}
& c \Vert z \Vert^{2-\beta} \int_{\R^k} d \tilde{\xi} \, \Vert \tilde{\xi} \Vert^{\beta-k-2} \left(1-\exp\left(-4 \pi^2 t \Vert \tilde{\xi} \Vert^2/\Vert z \Vert^2 \right)\right)\, \vert 1-\exp(-2 \pi i\, \tilde{\xi} \cdot e) \vert^2 \\
&\qquad\geq c \Vert z \Vert^{2-\beta} \int_{\R^k} d \tilde{\xi} \, \Vert \tilde{\xi} \Vert^{\beta-k-2}
\left(1-\exp \left(-4 \pi^2 t_0\Vert \tilde{\xi} \Vert^2/M^2 \right)\right)\,
\vert 1-\exp(-2 \pi i\, \tilde{\xi} \cdot e) \vert^2.
\end{align*}
The last integral is a positive constant. Therefore, when $\Vert z \Vert^2 \geq h$,
$$
 \E \, [\vert v_i(t+h,y+z)-v_i(s,y) \vert^2] \geq c \left(\max\left(h,\Vert z \Vert^2 \right)\right)^{\frac{2-\beta}{2}}.
$$

   Cases 1 and 2 together establish (\ref{equa4}).
\end{proof}

\subsection{Proof of Theorem \ref{t1'}}

Under the hypotheses on $b$ and $\sigma$, the components of $v=(v_1,\dots,v_d)$ are independent, so $v$ is a $(1+k,d)$-Gaussian random field in the sense of \cite{Xiao:09}. We apply Theorem 7.6 in \cite{Xiao:09}. For this, we are going to verify Conditions (C1) and (C2) of \cite[Section 2.4, p.158]{Xiao:09} with $N=k+1$, $H_1=(2-\beta)/4$, $H_j=(2-\beta)/2$, $j=1,...,k$.

   In particular, for (C1), we must check that there are positive constants $c_1,\dots,c_4$ such that for all $(t,x)$ and $(s,y)$ in $I\times J$,
\begin{equation}\label{rdC11}
    c_1 \leq E(v_1(t,x)^2) \leq c_2,
\end{equation}
and
\begin{equation}\label{rdC12}
   c_3\left(\vert t-s\vert^{\frac{2-\beta}{2}} + \Vert x-y\Vert^{2-\beta} \right) \leq E[(v_1(t,x) - v_1(s,y))^2] \leq c_4 \left(\vert t-s\vert^{\frac{2-\beta}{2}} + \Vert x-y\Vert^{2-\beta} \right).
\end{equation}

Condition \eqref{rdC11} is satisfied because $\E[v_1(t,x)^2]=C t^{(2-\beta)/2}$ (see \eqref{eqdom}, \eqref{ft} and Lemma \ref{scalar}).
The lower bound of \eqref{rdC12} follows from Proposition 2.1. The upper bound is a consequence of \cite[Propositions 2.4 and 3.2]{Sanz:00}.

Finally, in order to establish Condition (C2) it suffices to apply the fourth point of Remark 2.2 in \cite{Xiao:09}. Indeed, it is stated there that Condition (C1) implies condition (C2) when $(t,x) \mapsto E[v_1(t,x)^2] = C t^{(2-\beta)/2}$ is continuous in $I\times J$ with continuous partial derivatives, and this is clearly the case.

   This completes the proof of Theorem \ref{t1'}.
\hfill $\Box$
\vskip 16pt

\subsection{Proof of Theorem \ref{t1}(a)}

Fix $T>0$ and let $I \times J \subset (0,T] \times \mathbb{R}^k$ be a closed non-trivial rectangle. Let $\gamma \in (0, 2-\beta)$.
For all positive integers $n$, $i \in \{0,...,n\}$ and $j=(j_1,...,j_k) \in \{0,...,n \}^k$,
set $t_i^n=i 2^{-\frac{4n}{\gamma}}$, $x_j^n=(x_{j_1}^n=j_1 2^{-\frac{2n}{\gamma}},...,x_{j_k}^n=j_k 2^{-\frac{2n}{\gamma}})$, and
\begin{equation*}
I^n_{i,j}=[t_i^n \,,t_{i+1}^n] \times [x_{j_1}^n \,,x_{j_1+1}^n] \times \cdots \times [x_{j_k}^n \,,x_{j_k+1}^n].
\end{equation*}

The proof of the following lemma uses Theorem \ref{t2}(a) and (\ref{equa3}), but follows along the same lines as \cite[Theorem 3.3]{Dalang:05} with ${\bf \Delta}((t,x);(s,y))$ there replaced by $\vert t-s\vert^{\gamma/2} + \Vert x-y \Vert^\gamma$, $\beta$ there replaced by $d - \eta$ and $\epsilon$ in Condition (3.2) there replaced by $2^{-n}$. It is therefore omitted.

\begin{lem} \label{ball}
Fix $\eta>0$ and $M>0$. Then there exists $c>0$ such that for all $z \in [-M, M]^d$,
$n$ large and $I^n_{i,j} \subset I\times J$,
\begin{equation*}
P \{ u(I_{i,j}^n) \cap B(z\,,2^{-n}) \neq \varnothing\} \le c 2^{-n(d-\eta)}.
\end{equation*}
\end{lem}

\begin{proof}[Proof of the upper bound in Theorem \ref{t1}]
Fix $\e \in (0,1)$ and $n \in \mathbb{N}$ such that $2^{-n-1} <\epsilon \leq 2^{-n}$, and write
\begin{equation*}
\P\left\{ u\left( I \times J\right)
\cap B(z\,,\e) \neq\varnothing \right\}
\le \sum_{(i,j): I_{i,j}^n \cap (I\times J) \neq \varnothing} \P \{
u(I_{i,j}^n) \cap B(z\,,\e) \neq \varnothing\}.
\end{equation*}
The number of $(1+k)$-tuples $(i,j)$ involved in the sum is at most
$c \, 2^{n(\frac{4}{\gamma}+\frac{2}{\gamma}k)}$. Lemma \ref{ball} implies therefore that for all $z \in A$, $\eta>0$ and large $n$,
\begin{equation*}
\P\left\{ {u}\left( I \times J \right)
\cap B(z\,,\e) \neq\varnothing \right\}
\le \tilde{C} (2^{-n})^{d-\eta}\, 2^{n \frac{4+2k}{\gamma}}.
\end{equation*}
Let $\eta' = \eta + \left(\frac{1}{\gamma}- \frac{1}{2-\beta} \right) (4+2k)$. Then this is equal to
$$
2^{-n(d-(\frac{4+2k}{2-\beta}+\eta'))} \le C \e^{d-\frac{4+2k}{2-\beta}-\eta'},
$$
because $2^{-n-1}< \epsilon \leq 2^{-n}$. Note that $C$ does not depend on $(n \,,\epsilon)$, and $\eta'$ can be made arbitrarily small by choosing $\gamma$ close to $2-\beta$ and $\eta$ small enough. In particular, for all $\e \in (0,1)$,
\begin{equation}\label{eq:hit:ball:UB}
\P\left\{ {u}\left( I \times J \right) \cap B(z\,,\e) \neq\varnothing \right\}
\le C\, \e^{d-\frac{4+2k}{2-\beta}-\eta'}.
\end{equation}

Now we use a \emph{covering argument}:
Choose $\tilde\e \in (0,1)$ and let $\{B_i\}_{i=1}^{\infty}$ be a sequence of open balls
in $\mathbb{R}^d$ with
respective radii $r_i \in [0, \tilde\e)$ such that
\begin{equation} \label{hit2}
A\subset \cup_{i=1}^\infty B_i \text{ and }
\sum_{i=1}^\infty (2r_i)^{d-\frac{4+2k}{2-\beta}-\eta'} \le
\mathcal{H}_{d-\frac{4+2k}{2-\beta}-\eta'} (A)+\tilde\e.
\end{equation}
Because $\P\left\{ {u}\left( I \times J \right)
\cap A \neq\varnothing \right\}$ is at most
$\sum_{i=1}^\infty \P\left\{
{u}(I \times J )\cap B_i\neq
\varnothing \right\}$, the bounds in (\ref{eq:hit:ball:UB}) and (\ref{hit2}) together imply that
\begin{equation*}
\P\left\{ {u}\left( I \times J\right)
\cap A \neq\varnothing \right\} \le C \left( \mathcal{H}_{d-\frac{4+2k}{2-\beta}-\eta'}
(A)+\tilde\e\right).
\end{equation*}
Let $\tilde\e\to 0$ to conclude.
\end{proof}

\subsection{Proof of Theorem \ref{t1}(b)}\label{sec23}

The following preliminary lemmas are the analogues needed here of \cite[Lemma 2.2]{Dalang:05} and \cite[Lemma 2.3]{Dalang:05}, respectively.

\begin{lem} \label{prel}
Fix $T>0$ and let $I \times J \subset (0,T] \times \mathbb{R}^k$ be a closed non-trivial rectangle. Then for all $N>0$, $b>0$, $\tilde{\gamma}>\gamma>0$ and $p > \frac{2d}{\gamma}(\tilde\gamma b - 2k -4)$, there exists a finite and positive constant $C=C(I, J, N, b, \gamma, \tilde{\gamma},p)$ such that for all $a\in[0\,,N]$,
\begin{align}\nonumber
&\int_{I} dt \int_{I} ds \int_J dx \int_J dy
\ (\vert t-s \vert^{\tilde{\gamma}/2}+\Vert x-y \Vert^{\tilde{\gamma}})^{-b/2}
\left(\frac{\vert t-s \vert^{\gamma/4}+\Vert x-y \Vert^{\gamma/2}}{a} \, \wedge 1\right)^{p/(2d)} \\
&\qquad \le C \, {\rm K}_{\frac{\tilde{\gamma}}{\gamma}b-\frac{4+2k}{\gamma}}(a).
\label{int}
\end{align}
\end{lem}

\begin{proof}
Let $\vert J \vert$ denote the diameter of the set $J$. Using the change of variables $\tilde{u}=t-s$ ($t$ fixed), $\tilde{v}=x-y$ ($x$ fixed), we see that the integral in (\ref{int}) is bounded above by
\begin{equation*}
\vert I \vert \lambda_k(J) \int_0^{\vert I \vert} d\tilde{u}
\int_{B(0, \vert J \vert)} d\tilde{v} \,
(\tilde{u}^{\tilde{\gamma}/2}+\Vert \tilde{v} \Vert^{\tilde{\gamma}})^{-b/2}
\left(\frac{\tilde{u}^{\gamma/4}+\Vert \tilde{v} \Vert^{\gamma/2}}{a} \,\wedge 1\right)^{p/(2d)},
\end{equation*}
where $\lambda_k$ denotes Lebesgue measure in $\R^k$. A change of variables [$\tilde{u}=a^{4/\gamma} u^2$, $\tilde{v}=a^{2/\gamma} v$] implies that this is equal to
\begin{equation*}
C \, a^{\frac{4+2k}{\gamma}-\frac{\tilde{\gamma}}{\gamma} b} \int_0^{a^{-2/\gamma}(\vert I \vert)^{1/2} } u \, du
\int_{B(0, \vert J \vert a^{-2/\gamma})} dv \,
( u^{\tilde{\gamma}}+ \Vert v \Vert^{\tilde{\gamma}})^{-b/2}
\left(\left({u}^{\gamma/4}+\Vert {v} \Vert^{\gamma/2}\right) \,\wedge 1\right)^{p/(2d)}.
\end{equation*}
We pass to polar coordinates in the variable $v$, to see that this is bounded by
\begin{equation*}
C \, a^{\frac{4+2k}{\gamma}-\frac{\tilde{\gamma}}{\gamma}b} \int_0^{a^{-2/\gamma}(\vert I \vert)^{1/2} } du
\int_{0}^{ \vert J \vert a^{-2/\gamma}} dx \,  x^{k-1}\, u\,
( u^{\tilde{\gamma}}+ x^{\tilde{\gamma}})^{-b/2}
\left(\left(u^{\gamma/2}+x^{\gamma/2}\right)^{p/(2d)} \wedge 1\right) .
\end{equation*}
Bounding $x^{k-1}u$ by $(u+x)^k$ and using the fact that all norms in $\R^2$ are equivalent, we bound this above by
\begin{equation*}
C \, a^{\frac{4+2k}{\gamma}-\frac{\tilde{\gamma}}{\gamma} b} \int_0^{a^{-2/\gamma}(2\vert I \vert)^{1/2} } du
\int_{0}^{2 \vert J \vert a^{-2/\gamma}} dx \,
( u+ x)^{k-\frac{\tilde{\gamma} b}{2}}
\left((u+x)^{\gamma p/(4d)} \wedge 1\right).
\end{equation*}
We now pass to polar coordinates of $(u,x)$, to bound this by
\begin{equation} \label{a22}
C \, a^{\frac{4+2k}{\gamma}-\frac{\tilde{\gamma}}{\gamma} b} (I_1+I_2(a)),
\end{equation}
where
\begin{equation*}
\begin{split}
&I_1= \int_0^{K N^{-2/\gamma}} d\rho \,
\rho^{k+1-\frac{\tilde{\gamma} b}{2}}\, (\rho^{\gamma p /(4d)}\wedge 1), \\
&I_2(a)=\int_{KN^{-2/\gamma}}^{K a^{-2/\gamma}} d\rho \,
\rho^{k+1-\frac{\tilde{\gamma} b}{2}},
\end{split}
\end{equation*}
where $K=2(\sqrt{\vert I \vert} \vee  \vert J \vert)$. Clearly, $I_1 \leq C < \infty$ since $k+1 - \frac{\tilde \gamma}{2} b + \frac{\gamma p}{4d} > -1$ by the hypothesis on $p$. Moreover, if $k+2-\frac{\tilde{\gamma} b}{2} \neq 0$, then
\begin{equation*}
I_2(a)=K^{k+2-\frac{\tilde{\gamma} b}{2}} \, \frac{a^{\frac{\tilde{\gamma}}{\gamma}b-\frac{4+2k}{\gamma}} -N^{\frac{\tilde{\gamma}}{\gamma}b-\frac{4+2k}{\gamma}}}{k+2-\frac{\tilde{\gamma} b}{2}}.
\end{equation*}
There are three separate cases to consider.
(i) If $k+2-\frac{\tilde{{\gamma}} b}{2}<0$, then $I_2(a) \leq C$ for all
$a \in [0,N]$.
(ii) If $k+2-\frac{\tilde{\gamma} b}{2}>0$, then $I_2(a) \leq c \,
a^{\frac{\tilde{\gamma}}{\gamma}b-\frac{4+2k}{\gamma}}$.
(iii) If $k+2-\frac{\tilde{\gamma} b}{2}=0$, then
\begin{equation*}
I_2(a) =\frac{2}{\gamma} \biggl[\ln \frac{1}{a}+\ln N \biggr].
\end{equation*}
We combine these observations to conclude that the expression in
\eqref{a22} is bounded by $C \, {\rm
K}_{\frac{\tilde{\gamma}}{\gamma}b-\frac{4+2k}{\gamma}}(a)$,
provided that $N_0$ in (\ref{k}) is sufficiently large. This proves
the lemma.
\end{proof}

For all $a,\nu,\rho>0$, define
\begin{equation} \label{rdpsi}
\Psi_{a,\nu}(\rho) := \int_0^a dx \, \frac{x^{k-1}}{\rho + x^\nu}.
\end{equation}

\begin{lem}\label{lem:Psi}
For all $a,\nu, T>0$, there exists a finite
and positive constant $C=C(a\,,\nu \,,T)$ such that
for all $0<\rho<T$,
\begin{equation*}
\Psi_{a,\nu}(\rho) \le C {\rm K}_{(\nu-k)/\nu}(\rho).
\end{equation*}
\end{lem}

\begin{proof}
If $\nu<k$, then $\lim_{\rho \to 0}
\Psi_{a,\nu}(\rho)=\int_0^a x^{k-1-\nu}\, dx<\infty$.
In addition, $\rho \mapsto \Psi_{a,\nu}(\rho)$ is nonincreasing, so
$\Psi_{a,\nu}$ is bounded on $\R_+$
when $\nu<k$. In this case, ${\rm K}_{(\nu-k)/\nu}(\rho)=1$,
so the result follows in the case that $\nu<k$.

For the case $\nu\ge k$, we change variables $(y=x\rho^{-1/\nu})$
to find that
\begin{equation*}
\Psi_{a,\nu}(\rho) = \rho^{-(\nu-k)/\nu}
\int_0^{a \rho^{-1/\nu}} dy \, \frac{y^{k-1}}{
1+ y^\nu}.
\end{equation*}
When $\nu>k$, this gives the desired result, with $c= \int_0^{+\infty} dy\, y^{k-1}\,(1+ y^\nu)^{-1}$.
When $\nu=k$, we simply evaluate the integral in (\ref{rdpsi}) explicitly: this gives the result for $0< \rho < T$,
given the choice of ${\rm K}_0(r)$ in (\ref{k}).
We note that the constraint ``$0 < \rho < T$'' is needed only in this case.
\end{proof}

\begin{proof}[Proof of the lower bound of Theorem \ref{t1}]

The proof of this result follows along the same lines as the proof
of \cite[Theorem 2.1(1)]{Dalang:05}, therefore we will only sketch
the steps that differ. We need to replace their $\beta-6$ by our $d-\frac{4+2k}{\tilde{\gamma}}+\eta$.

Note that our Theorem \ref{t2}(a) and Theorem \ref{te2} prove that
\begin{equation} \label{strictpos}
\inf_{\Vert z \Vert \leq M} \int_I dt \int_J dx \, p_{t,x}(z) \geq C >0,
\end{equation}
which proves hypothesis {\bf A1'} of \cite[Theorem 2.1(1)]{Dalang:05}, (see \cite[Remark 2.5(a)]{Dalang:05}).

Moreover, Theorem \ref{t2}(b) proves a property that is weaker than hypothesis {\bf A2} of \cite[Theorem 2.1(1)]{Dalang:05} with their $\beta=d+\eta$, $\gamma \in (0, 2-\beta)$ and
$$
{\bf \Delta}((t,x) \, ; (s,y))=\vert t-s \vert^{\gamma/2}+\Vert x-y \Vert^{\gamma},
$$
but which will be sufficient for our purposes.

Let us now follow the proof of \cite[Theorem 2.1(1)]{Dalang:05}.
Define, for all $z \in \mathbb{R}^d$ and $\epsilon>0$,
$\tilde{B}(z \,,\e):=\{y\in\R^d:\ |y-z|<\e\}$, where $|z|:=\max_{1\le j\le d}|z_j|$, and
\begin{equation}
J_\e(z) = \frac{1}{(2\e)^d} \int_{I} dt \, \int_{J } dx \, \1_{\tilde{B}(z,\e)}
(u(t\,,x)),
\end{equation}
as in \cite[(2.28)]{Dalang:05}.

Assume first that $d+\eta<\frac{4+2k}{2-\beta}$. Using Theorem
\ref{t2}(b), we find, instead of \cite[(2.30)]{Dalang:05},
\begin{equation*}
\E \left[ (J_\e(z))^2 \right] \leq c \, \int_{I} dt \int_{I} ds \int_{J} dx \int_{J} dy \,
(\vert t-s \vert^{\frac{2-\beta}{2}}+\Vert x-y \Vert^{2-\beta})^{-(d+\eta)/2}.
\end{equation*}
Use the change of variables
$u=t-s$ ($t$ fixed), $v=x-y$ ($x$ fixed) to see that the above integral is bounded above by
\begin{equation*} \begin{split}
& \tilde{c} \int_0^{ \vert I \vert} d\tilde{u}
\int_{B(0, \vert J \vert)} d \tilde{v} \,
(\tilde{u}^{\frac{2-\beta}{2}}+\Vert \tilde{v} \Vert^{2-\beta})^{-(d+\eta)/2} \\
& \qquad \qquad = c \int_0^{\vert I \vert} du
\int_{0}^{ \vert J \vert} dx \, x^{k-1} ( u^{\frac{2-\beta}{2}}+ x^{2-\beta})^{-(d+\eta)/2} \\
& \qquad \qquad\leq c \int_0^{\vert I \vert} du \, \Psi_{ \vert J
\vert, (2-\beta)(d+\eta)/2} (u^{(2-\beta)(d+\eta)/4}).
\end{split}
\end{equation*}
Hence, Lemma \ref{lem:Psi} implies that for all $\epsilon>0$,
\begin{equation*}
\E \left[ (J_\e(z))^2 \right] \leq C \int_0^{ \vert I \vert}
du \, {\rm K}_{1-\frac{2k}{(2-\beta)(d+\eta)}}(u^{(2-\beta)(d+\eta)/4}).
\end{equation*}
We now consider three
different cases: (i) If $0< (2-\beta)(d+\eta)<2k$, then the integral equals $\vert I \vert$.
(ii) If $2k<(2-\beta)(d+\eta) <4+2k$, then ${\rm K}_{1-\frac{2k}{(2-\beta)(d+\eta)}}(u^{(2-\beta)(d+\eta)/4})=
u^{(k/2)-(2-\beta)(d+\eta)/4}$ and the integral is finite. (iii) If $(2-\beta)(d+\eta)=2k$, then
${\rm K}_{0}(u^{k/2})=\text{log}(N_0/u^{k/2})$ and the integral is also finite.
The remainder of the proof of the lower bound of Theorem \ref{t1}
when $d+\eta<\frac{4+2k}{2-\beta}$ follows exactly as in
\cite[Theorem 2.1(1) {\it Case 1} ]{Dalang:05}.

Assume now that $d+\eta >\frac{4+2k}{2-\beta}$. Define, for all $\mu
\in \mathcal{P}(A)$ and $\e>0$,
\begin{equation*}
J_\e(\mu) = \frac{1}{(2\e)^d} \int_{\R^d} \mu(dz) \, \int_{I} dt \, \int_{J} dx \, \1_{\tilde{B}(z,\e)} (u(t\,,x)),
\end{equation*}
as \cite[(2.35)]{Dalang:05}.

In order to prove the analogue of \cite[(2.41)]{Dalang:05}, we use
Theorem \ref{t2}(b) and Lemma \ref{prel} (instead of \cite[Lemma 2.2(1)]{Dalang:05}), to see that for all $\mu
\in \mathcal{P}(A)$, $\epsilon \in (0,1)$ and $\gamma \in (0,
2-\beta)$,
\begin{equation*}
\E \left[ (J_\e(\mu))^2 \right] \leq c\, \left[\text{Cap}_{\frac{2-\beta}{\gamma}(d+\eta)-\frac{4+2k}{\gamma}}(A)\right]^{-1} =c\, \left[\text{Cap}_{d+\tilde{\eta}-\frac{4+2k}{2-\beta}}(A)\right]^{-1}.
\end{equation*}
The remainder of the proof of the lower bound of Theorem \ref{t1}
when $d+\eta>\frac{4+2k}{2-\beta}$ follows as in \cite[Proof of
Theorem 2.1(1) {\it Case 2}]{Dalang:05}.

The case $d+\eta =\frac{4+2k}{2-\beta}$ is proved exactly along the
same lines as the proof of \cite[Theorem 2.1(1) {\it Case
3}]{Dalang:05}, appealing to (\ref{strictpos}), Theorem \ref{t2}(b)
and Lemma \ref{prel}.
\end{proof}

\subsection{Proof of Corollaries \ref{c33} and \ref{c3bis}}

\begin{proof} [Proof of Corollary \ref{c33}.]
Let $z \in \mathbb{R}^d$. If $d < \frac{4+2k}{2-\beta}$, then there is $\eta>0$ such that $d-\frac{4+2k}{2-\beta}+\eta < 0$, and thus
$$
\textnormal{Cap}_{d-\frac{4+2k}{2-\beta}+\eta}(\{ z\})=1.
$$
Hence, Theorem \ref{t1}(b) implies that
$\{ z\}$ is not polar. On the other hand, if $d>\frac{4+2k}{2-\beta}$, then there is $\eta>0$ such that $d-\frac{4+2k}{2-\beta}-\eta > 0$. Therefore,
$$
\mathcal{H}_{d-\frac{4+2k}{2-\beta}-\eta}(\{ z\})=0
$$
and Theorem \ref{t1}(a) implies that $\{ z\}$ is polar.
\end{proof}

\begin{proof} [Proof of Corollary \ref{c3bis}.]
We first recall, following {\sc Khoshnevisan} \cite[Chap.11, Section 4]{Khoshnevisan:02}, the definition of stochastic codimension of a random set $E$ in $\mathbb{R}^d$, denoted
$\textnormal{codim}(E)$, if it exists: $\textnormal{codim}(E)$ is the real number $\alpha \in [0,d]$ such that for all compact sets $A \subset
\mathbb{R}^d$,
\begin{equation*}
P \{ E \cap A \neq \varnothing \}
\begin{cases}
>0 & \text{whenever $\dimh(A)> \alpha$}, \\
=0 & \text{whenever $\dimh(A)< \alpha$}.
\end{cases}
\end{equation*}
By Theorem \ref{t1},
$\textnormal{codim}(u(\mathbb{R}_+ \times \R^k))=(d-\frac{4+2k}{2-\beta})^+$.
Moreover, in {\sc Khoshnevisan} \cite[Thm.4.7.1, Chap.11]{Khoshnevisan:02}, it is
proved that given a random set $E$ in $\mathbb{R}^d$ whose
codimension is strictly between $0$ and $d$,
\begin{equation*}
\dimh(E)+\textnormal{codim}(E)=d, \; \; \text{a.s}.
\end{equation*}
This implies the desired statement.
\end{proof}


\section{Elements of Malliavin calculus}

Let $\mathcal{S}(\R^{k})$ be the Schwartz space of
$\mathcal{C}^\infty$ functions on $\R^k$ with rapid decrease. Let $\mathcal{H}$ denote the completion of $\mathcal{S}(\R^{k})$
endowed with the inner product
\begin{equation*}
\langle \phi(\cdot), \psi(\cdot) \rangle_{\mathcal{H}}=\int_{\R^k} d x \, \int_{\R^k} dy \,
\phi(x) \Vert x-y \Vert^{-\beta} \psi(y) =\int_{\R^k} d \xi \, \Vert \xi \Vert^{\beta-k} \mathcal{F} \phi(\cdot) (\xi)
\overline{\mathcal{F} \psi(\cdot) (\xi)},
\end{equation*}
$\phi, \psi \in \mathcal{S}(\R^{k})$. Notice that $\mathcal{H}$ may contain Schwartz distributions (see \cite{Dalang:99}).

For $h=(h^1,\dots,h^d) \in \cH^d$ and $\tilde h=(\tilde
h^1,\dots,\tilde h^d) \in \cH^d$, we set $\langle h, \tilde h
\rangle_{\cH^d} = \sum_{i = 1}^d \langle h^i, \tilde h^i
\rangle_\cH$. Let $T>0$ be fixed. We set $\mathcal{H}^d_T=L^2([0,T]; \mathcal{H}^d)$ and
for $0\leq s \leq t \leq T$, we will write $\cH_{s,t}^d=L^2([s,t];
\mathcal{H}^d)$.

The centered Gaussian noise $F$ can be used to construct an isonormal Gaussian process $\{ W(h),\ h \in \mathcal{H}^d_T\}$ (that is, $\E[W(h)W(\tilde h)] = \langle h, \tilde h \rangle_{\mathcal{H}^d_T}$) as follows. Let $\{ e_j,\, j \geq 0\}\subset \mathcal{S}(\R^{k})$ be a complete orthonormal system of the Hilbert space $\mathcal{H}$. Then for any $t \in [0,T]$, $i \in \{1,\dots,d\}$ and $j \geq 0$, set
\begin{equation*}
W^i_j(t)=\int_0^t \int_{\R^k} e_j(x) \cdot F^i(ds,dx),
\end{equation*}
so that $(W^i_j,\ j \geq 1)$ is a sequence of independent standard real-valued Brownian motions such that for any $\phi \in \mathcal{D}([0,T] \times \R^k)$,
\begin{equation*}
F^i(\phi)= \sum_{j=0}^{\infty} \int_0^T \langle \phi(s, \cdot), e_j(\cdot) \rangle_{\mathcal{H}}
\, d W^i_j(s),
\end{equation*}
where the series converges in $L^2(\Omega,\mathcal{F}, \P)$. For $h^i \in \mathcal{H}_T$, we set
\begin{equation*}
W^i(h^i)= \sum_{j=0}^{\infty} \int_0^T \langle h^i(s, \cdot), e_j(\cdot) \rangle_{\mathcal{H}} \, d W^i_j(s),
\end{equation*}
where, again, this series converges in $L^2(\Omega,F,P)$. In particular, for $\phi \in \mathcal{D}([0,T] \times \R^k)$, $F^i(\phi)=W^i(\phi)$. Finally, for $h=(h^1,\dots,h^d) \in \mathcal{H}^d_T$, we set
$$
W(h) = \sum_{i=1}^d W^i(h^i).
$$

With this isonormal Gaussian process, we can use the framework of Malliavin calculus. Let $\mathcal{S}$ denote the class of smooth random variables of the form $G=g(W(h_1),...,W(h_n))$,
where $n \geq 1$, $g \in \mathcal{C}^{\infty}_P(\mathbb{R}^n)$, the set of real-valued
functions $g$ such that $g$ and all its partial derivatives have
at most polynomial growth, $h_i \in \mathcal{H}^d_T$.
Given $G \in \mathcal{S}$, its derivative $(D_r G = (D^{(1)}_r G,\dots,D^{(d)}_r G),\ r \in [0,T])$, is an $\mathcal{H}^d_T$-valued random vector defined by
\begin{equation*}
D_{r} G= \sum_{i=1}^n \frac{\partial g}{\partial x_i}
(W(h_1),...,W(h_n)) h_i(r).
\end{equation*}
For $\phi \in \mathcal{H}^d$ and $r \in [0,T]$, we write $D_{r, \phi} G = \langle D_{r} G, \phi(\cdot) \rangle_{\mathcal{H}^d}$. More generally, the derivative $D^m G = (D^{m}_{(r_1,...,r_m)} G,\, (r_1,\dots,r_m) \in [0,T]^m)$ of order $m \geq 1$ of $G$ is the $(\mathcal{H}^d_T)^{\otimes j}$-valued random vector defined by
\begin{equation*}
D^{m}_{(r_1,...,r_m)}G=\sum_{i_1,...,i_m=1}^n \frac{\partial}{\partial{x_{i_1}}} \cdots
\frac{\partial}{\partial{x_{i_m}}} g(W(h_1),...,W(h_n)) h_{i_1}(r_1) \otimes \cdots \otimes
h_{i_m}(r_m).
\end{equation*}
For $p, m\geq 1$, the space $\mathbb{D}^{m,p}$ is the closure of $\mathcal{S}$ with
respect to the seminorm $\Vert \cdot \Vert_{m,p}$ defined by
\begin{equation*}
\Vert G \Vert ^p _{m,p} = \E[\vert G\vert^p] + \sum_{j=1}^m
\E\left[\Vert D^j G \Vert ^p _{(\mathcal{H}^d_T)^{\otimes j}}\right].
\end{equation*}
We set $\mathbb{D}^{\infty}= \cap_{p \geq 1}\cap_{m \geq 1}\mathbb{D}^{m,p}$.

The derivative operator $D$ on $L^2(\Omega)$ has an
adjoint, termed the Skorohod integral and denoted by $\delta$, which
is an unbounded operator on $L^2(\Omega, \mathcal{H}_T^d)$. Its domain, denoted by Dom~$\delta$, is the
set of elements $u \in L^2(\Omega, \mathcal{H}_T^d)$ for which there exists a constant $c$ such that 
$|\E [\langle DF, u \rangle_{\mathcal{H}_T^d}] | \leq c \Vert F \Vert _{0,2}$, for any $F \in
\mathbb{D}^{1,2}$. If $u \in$ Dom~$\delta$, then $\delta (u)$ is the element of
$L^2(\Omega)$ characterized by the following duality relation:
\begin{equation*}
\E [F \delta (u)]= \E [\langle DF, u \rangle_{\mathcal{H}_T^d}] , \;
\; \text{for all} \; F \in  \mathbb{D}^{1,2}.
\end{equation*}

An important application of Malliavin calculus is the following global criterion for existence and smoothness
of densities of probability laws.
\begin{thm} \label{3t1}\textnormal{~\cite[Thm.2.1.2 and Cor.2.1.2]{Nualart:95} or~\cite[Thm.5.2]{Sanz:05}}
Let $F=(F^1,...,F^d)$ be an $\mathbb{R}^d$-valued random vector satisfying the following
two conditions: 
\begin{enumerate}
\item[\textnormal{(i)}] $F \in (\mathbb{D}^{\infty})^d$; 
\item[\textnormal{(ii)}] the Malliavin matrix of $F$ defined by
$\gamma_F=(\langle DF^i, D F^j \rangle_{\mathcal{H}_T^d})_{1\leq i,j\leq d}$ is
invertible a.s. and $(\textnormal{det} \; \gamma_F)^{-1} \in L^p(\Omega)$ for all $p\geq1$.
\end{enumerate}
Then the probability law of $F$ has an infinitely differentiable density function.
\end{thm}

A random vector $F$ that satisfies conditions (i) and (ii) of Theorem~\ref{3t1} is said to be
nondegenerate. The next result gives a criterion for uniform boundedness of the density of a
nondegenerate random vector.
\begin{prop} \label{norm} \textnormal{\cite[Proposition 3.4]{Dalang2:05}}
For all $p>1$ and $\ell \geq1$, let $c_1=c_1(p)>0$ and $c_2=c_2(\ell,p)\geq 0$ be fixed.
Let $F \in (\mathbb{D}^{\infty})^d$ be a nondegenerate random vector such that
\begin{itemize}
\item[\textnormal{(a)}] $\E [(\textnormal{det}
\,\gamma_{F} )^{-p}] \leq c_1$;
\item[\textnormal{(b)}] $\E [\Vert D^{l}(F^i)
\Vert^p_{(\mathcal{H}_T^d)^{\otimes \ell}}] \leq c_2 ,\; i=1,...,d$.
\end{itemize}
Then the density of $F$ is $C^\infty$ and uniformly bounded, and the bound does not depend on $F$ but only on the constants
$c_1(p)$ and $c_2(\ell,p)$.
\end{prop}

In \cite{Marquez:01}, the Malliavin differentiability and the smoothness of the density of $u(t,x)$ was established when $d=1$, and the extension to $d > 1$ can easily be done by working coordinate by coordinate. These results were extended in \cite[Prop.~5.1]{Nualart:07}. In particular, letting $\cdot$ denote the spatial variable, for $r \in [0,t]$ and $i,l \in \{1,\dots,d\}$, the derivative of $u_i(t,x)$ satisfies the system of equations
\begin{equation} \begin{split} \label{derivative}
D^{(l)}_{r} (u_i(t,x)) &= \sigma_{il}(u(r,\cdot))\, S(t-r,x-\cdot) \\
& \, + \int_r^t \int_{\R^k} S(t-\theta,x-\eta) \sum_{j=1}^d D^{(l)}_{r} (\sigma_{i,j}(u(\theta,\eta))) \, M^j(d\theta, d \eta) \\
& \, + \int_r^t d\theta \int_{\R^k} d\eta \, S(t-\theta,x-\eta) D^{(l)}_{r} (b_i(u(\theta,\eta))), \end{split}
\end{equation}
and $D^{(l)}_{r} (u_i(t,x))=0$ if $r>t$. Moreover, by \cite[Prop.~6.1]{Nualart:07}, for any $p>1$, $m\geq 1$ and $i \in \{1,...,d\}$, the order $m$ derivative satisfies
\begin{equation} \label{supderivative}
\sup_{(t,x) \in [0,T] \times \R^k} \E \, \biggl[\big\Vert D^m(u_i(t,x)) \big\Vert^p_{(\mathcal{H}^d_T)^{\otimes m}} \biggr] < +\infty,
\end{equation}
and $D^m$ also satisfies the system of stochastic partial differential equations given in \cite[(6.29)]{Nualart:07} and obtained by iterating the calculation that leads to (\ref{derivative}). In particular, $u(t,x) \in (\mathbb{D}^{\infty})^d$, for all $(t,x) \in [0,T] \times \R^k$.

\section{Existence, smoothness and uniform boundedness of the density}\label{sec4}

The aim of this section is to prove Theorem \ref{t2}(a). For this,
we will use Proposition \ref{norm}. The following
proposition proves condition (a) of Proposition \ref{norm}.
\begin{prop} \label{det}
Fix $T>0$ and assume hypotheses $\textnormal{{\bf P1}}$ and $\textnormal{{\bf {\bf P2}}}$. Then, for any $p\geq 1$,
$
\E \, \bigl[( \textnormal{det} \, \gamma_{u(t,x)})^{-p} \bigr]
$
is uniformly bounded over $(t,x)$ in any closed non-trivial rectangle $I\times J \subset (0,T] \times \R^k$.
\end{prop}

\begin{proof}
Let $(t,x) \in I \times J$ be fixed, where $I\times J$ is a closed non-trivial rectangle of $(0,T] \times \R^k$. We write
\begin{equation*}
\textnormal{det} \gamma_{u(t,x)} \geq \biggl( \textnormal{inf}_{\xi \in
\mathbb{R}^d: \Vert \xi \Vert =1} (\xi^{T} \gamma_{u(t,x)}
\xi)\biggr)^d.
\end{equation*}

Let $\xi \in \mathbb{R}^d$ with $\Vert \xi \Vert=1$ and fix $\epsilon \in (0,1)$. Using (\ref{derivative}), we see that
\begin{align*}
\xi^{T} \gamma_{u(t,x)} \xi &\geq \sum_{l=1}^d \int_{t-\epsilon}^t dr
\, \bigg\Vert \sum_{i=1}^d D^{(l)}_{r}(u_i(t,x)) \xi_i \bigg\Vert_{\mathcal{H}}^2 \\
&= \sum_{l=1}^d \int_{t-\epsilon}^t dr
\, \bigg\Vert \sum_{i=1}^d \sigma_{i,l}(u(r,\cdot)) S(t-r,x -\cdot) \xi_i + \sum_{i=1}^d a_i(l,r,t,x)\xi_i \bigg\Vert_\cH^2,
\end{align*}
where, for $r < t$,
\begin{equation} \begin{split}\label{eq4.1}
a_i(l,r,t,x)=& \int_r^t \int_{\R^k} S(t-\theta,x-\eta) \sum_{j=1}^d D^{(l)}_{r} (\sigma_{i,j}(u(\theta,\eta))) \, M^j(d\theta, d \eta) \\
& + \int_r^t d\theta \int_{\R^k} d\eta \, S(t-\theta,x-\eta) D^{(l)}_{r} (b_i(u(\theta,\eta))).
\end{split}
\end{equation}
We use the inequality
\begin{equation} \label{sumaqua}
\Vert a+b \Vert_{\mathcal{H}}^2 \geq \frac{2}{3} \Vert a \Vert_{\mathcal{H}}^2-2 \Vert b \Vert_{\mathcal{H}}^2,
\end{equation}
to see that
\begin{equation*}
\xi^{\sf T} \gamma_{u(t,x)} \xi \geq \frac{2}{3} \sum_{l=1}^d \int_{t-\epsilon}^t dr
\, \Vert (\xi^{\sf T} \cdot \sigma(u(r,\cdot)))_l\, S(t-r,x -\cdot)
\Vert_\cH^2 - 2 A_3,
\end{equation*}
where
\begin{equation*}
A_3 = \int_{t-\epsilon}^t dr \, \sum_{l=1}^d
\bigg\Vert \sum_{i=1}^d a_i(l,r,t,x) \,\xi_i \bigg\Vert_{\mathcal{H}}^2.
\end{equation*}
The same inequality \eqref{sumaqua} shows that
\begin{equation*}
\sum_{l=1}^d \int_{t-\epsilon}^t dr
\, \Vert (\xi^{\sf T} \cdot \sigma(u(r,\cdot)))_l\, S(t-r,x -\cdot)
\Vert_\cH^2 \geq \frac{2}{3} A_1 - 2 A_2,
\end{equation*}
where
\begin{equation} \label{add}
\begin{split} 
A_{1}&=\int_{t-\epsilon}^t dr \,\sum_{l=1}^d \Vert ({\xi}^{\sf T} \cdot \sigma(u(r,x)))_l\, S(t-r,x -\cdot)
\Vert_\cH^2, \\ 
A_{2}&=\int_{t-\e}^t dr
\sum_{l=1}^d \Vert ({\xi}^{\sf T} \cdot (\sigma(u(r,\cdot))-\sigma(u(r,x))) )_l\, S(t-r,x -\cdot)
\Vert_\cH^2.
\end{split}
\end{equation}
Note that we have added and subtracted a ``localized" term so as to be able to use the ellipticity property of $\sigma$ (a similar idea is used in
\cite{Millet:99} in dimension 1).

Hypothesis {\bf P2} and Lemma \ref{scalar} together yield $A_1 \geq
C \epsilon^{\frac{2-\beta}{2}}$, where $C$ is uniform over $(t,x)
\in I \times J$.

Now, using the Lipschitz property of $\sigma$ and H\"older's inequality with respect to the measure $\Vert y-z \Vert^{-\beta} S(t-r,x-y) S(t-r,x-z) \, dr dy dz$, we get that for $q \geq 1$,
\begin{align*} \nonumber
\E \, \left[\sup_{\xi\in\R^d: \,\Vert \xi\Vert = 1}\vert A_2 \vert^q \right] &\leq \biggl(\int_{t-\epsilon}^t dr \, \int_{\R^k} dy \,
\int_{\R^k} dz \, \Vert y-z \Vert^{-\beta} S(t-r,x-y) S(t-r,x-z)\biggr)^{q-1} \\ \nonumber
& \qquad \times \biggl(\int_{t-\epsilon}^t dr \, \int_{\R^k} dy \,
\int_{\R^k} dz \, \Vert y-z \Vert^{-\beta} S(t-r,x-y) S(t-r,x-z) \\
& \qquad \qquad \qquad \qquad\times \E \, [\Vert u(r,y)-u(r,x )\Vert^q \Vert u(r,z)-u(r,x )\Vert^q]\biggr).
\end{align*}
Using Lemma \ref{scalar} and (\ref{equa3}) we get that for any $q \geq 1$ and $\gamma \in (0, 2-\beta)$,
\begin{equation*}
\E \, [\vert A_2 \vert^q ] \leq C \epsilon^{(q-1)\frac{2-\beta}{2}} \times \Psi,
\end{equation*}
where
\begin{equation*}
\Psi=\int_0^{\epsilon} dr \, \int_{\R^k} dy \,
\int_{\R^k} dz \, \Vert y-z \Vert^{-\beta} S(r,x-y) S(r,x-z) \Vert y-x \Vert^{\frac{\gamma q}{2}}\Vert z-x \Vert^{\frac{\gamma q}{2}}.
\end{equation*}
Changing variables $[\tilde{y}=\frac{x-y}{\sqrt{r}}, \tilde{z}=\frac{x-z}{\sqrt{r}}]$, this becomes
\begin{align*} \nonumber
\Psi&=\int_0^{\epsilon} dr \, r^{-\frac{\beta}{2}+\frac{\gamma q}{2}}
\int_{\R^k} d\tilde{y} \,
\int_{\R^k} d \tilde{z} \, S(1,\tilde{y}) S(1,\tilde{z })\Vert \tilde{y}-\tilde{z} \Vert^{-\beta} \Vert \tilde{y} \Vert^{\frac{\gamma q}{2}} \Vert \tilde{z} \Vert^{\frac{\gamma q}{2}} \\
&= C \epsilon^{\frac{2-\beta}{2}+\frac{\gamma q}{2}}.
\end{align*}
Therefore, we have proved that for any $q \geq 1$ and $\gamma \in (0, 2-\beta)$,
\begin{equation} \label{equa45}
\E \, \biggl[\sup_{\xi \in \R^d: \Vert \xi \Vert=1} \vert A_2 \vert^q \biggr] \leq C \epsilon^{\frac{2-\beta}{2}q+\frac{\gamma }{2}q},
\end{equation}
where $C$ is uniform over $(t,x) \in I \times J$.

On the other hand, applying Lemma \ref{lem:7} with $s=t$, we find that for any $q \geq 1$,
\begin{equation*}
\E \, \biggl[\sup_{\xi \in \R^d: \Vert \xi \Vert=1} \vert A_3 \vert^q \biggr] \leq C \epsilon^{(2-\beta) q },
\end{equation*}
where $C$ is uniform over $(t,x) \in I \times J$.

Finally, we apply \cite[Proposition 3.5]{Dalang2:05} with $Z:=\inf_{
\Vert \xi \Vert=1}(\xi^{\sf T} \gamma_{u(t,x)} \xi)$,
$Y_{1,\epsilon}=Y_{2,\epsilon}=\sup_{ \Vert \xi \Vert=1}(\vert A_{2}
\vert + \vert A_3 \vert)$, $\epsilon_0=1$,
$\alpha_1=\alpha_2=\frac{2-\beta}{2}$, and
$\beta_1=\beta_2=\frac{2-\beta}{2}+\frac{\gamma}{2}$, for any
$\gamma \in (0, 2-\beta)$, to conclude that for any $p\geq1$,
\begin{equation*}
\E \, [(\textnormal{det} \, \gamma_{u(t,x)})^{-p} \bigr]\leq C(p),
\end{equation*}
where the constant $C(p) < \infty$ does not depend on $(t,x) \in I \times J$.
\end{proof}

In \cite[Theorem 3.2]{Marquez:01} the existence and smoothness of the density of the solution of equation (\ref{equa1}) with one single equation
($d=1$) was proved (see also \cite[Theorem 6.2]{Nualart:07}). The extension of this fact for a system of $d$ equations is given in the next proposition.

\begin{prop} \label{unif}
Fix $t>0$ and $x \in \R^k$. Assume hypotheses \textnormal{{\bf P1}} and \textnormal{\bf {P2}}. Then the law of $u(t,x)$, solution of
\textnormal{(\ref{equa1})}, is absolutely continuous with respect to Lebesgue measure on $\R^d$. Moreover, its density $p_{t,x}(\cdot)$ is
$C^{\infty}$.
\end{prop}

\begin{proof}
This is a consequence of Theorem \ref{3t1} and Proposition \ref{det}.
\end{proof}

\noindent{\em Proof of Theorem \ref{t2}(a).}
This follows directly from Proposition \ref{det} and (\ref{supderivative}), using Proposition 
\ref{norm}.
\hfill $\Box$
\vskip 16pt 

\section{Gaussian upper bound for the bivariate density}\label{sec5}

The aim of this section is to prove Theorem \ref{t2}(b).

\subsection{Upper bound for the derivative of the increment}

\begin{prop} \label{derivat}
Assume hypothesis \textnormal{{\bf P1}}.
Then for any $T>0$ and $p\geq 1$, there exists $C:=C(T,p)>0$ such that
for any $0 \leq s \leq t \leq T$, $x, y \in \R^k$, $m \geq 1$, $i \in \{1,...,d\}$, and $\gamma \in (0, 2-\beta)$,
\begin{equation*}
\Vert D^m(u_i(t,x)-u_i(s,y)) \Vert_{L^p(\Omega; (\mathcal{H}_T^d)^{\otimes m})} \leq C ( \vert t-s \vert^{\gamma/2}
+\Vert x-y\Vert^{\gamma})^{1/2}.
\end{equation*}
\end{prop}

\begin{proof}
Assume $m=1$ and fix $p\geq 2$, since it suffices to prove the statement in this case. Let 
$$
   g_{t,x ; s,y}(r, \cdot):=S(t-r,x-\cdot) 1_{\{ r \leq t\} }-S(s-r,y-\cdot) 1_{\{ r\leq s\} }.
$$
Using (\ref{derivative}), we see that
\begin{equation*}
\Vert D(u_i(t,x)-u_i(s,y)) \Vert^p_{L^p(\Omega; \mathcal{H}_T^d)} \leq c_p(A_1+A_{2,1}+A_{2,2}+A_{3,1}+A_{3,2}),
\end{equation*}
where
\begin{equation*} \begin{split}
A_1&=\E \, \biggl[\biggl(\int_0^T dr \sum_{j=1}^d \big\Vert g_{t,x ; s,y}(r,\cdot) \sigma_{ij}(u(r,\cdot)) \big\Vert^2_{\mathcal{H}} \biggr)^{p/2} \biggr], \\ \nonumber
A_{2,1}&= \E \, \biggl[ \bigg\Vert \int_0^T \int_{\R^k} g_{t,x ; t,y}(\theta,\eta) \sum_{j=1}^d D (\sigma_{i,j}(u(\theta,\eta)))
M^j(d\theta, d \eta) \bigg\Vert^p_{\mathcal{H}^d_T} \biggr], \\
A_{2,2}&= \E \, \biggl[ \bigg\Vert \int_0^T \int_{\R^k} g_{t,y ; s,y}(\theta,\eta) \sum_{j=1}^d D (\sigma_{i,j}(u(\theta,\eta)))
M^j(d\theta, d \eta) \bigg\Vert^p_{\mathcal{H}^d_T} \biggr], \\
A_{3,1}&=\E \, \biggl[ \bigg\Vert\int_0^T d\theta \int_{\R^k} d \eta \,
g_{t,x ; t,y}(\theta,\eta) D (b_i(u(\theta,\eta))) \bigg\Vert^p_{\mathcal{H}^d_T}\biggr], \\
A_{3,2}&=\E \, \biggl[ \bigg\Vert\int_0^T d\theta \int_{\R^k} d \eta \,
g_{t,y ; s,y}(\theta,\eta) D (b_i(u(\theta,\eta))) \bigg\Vert^p_{\mathcal{H}^d_T}\biggr].
\end{split}
\end{equation*}

Using Burkholder's inequality, \eqref{eqdom} and (\ref{equa3}), we see that for any $\gamma \in (0, 2-\beta)$,
\begin{equation}\label{5.1a} 
A_1 \leq c_p
\E \, \biggl[\bigg\vert\int_0^T \int_{\R^k} g_{t,x ; s,y}(\theta,\eta) \sum_{j=1}^d  \sigma_{ij}(u(\theta,\eta)) M^j(d\theta, d \eta) \bigg\vert^{p} \biggr] .
\end{equation}
In order to bound the right-hand side of \eqref{5.1a}, one proceeds as in \cite{Sanz:02}, where the so-called ``factorization method" is used. In fact, the calculation used in \cite{Sanz:02} in order to obtain \cite[(10)]{Sanz:02} and \cite[(19)]{Sanz:02} (see in particular the treatment of the terms $I_2(t,h,x)$, $I_3(t,h,x)$, and $J_2(t,x,z)$ in this reference) show that for any $\gamma \in (0, 2-\beta)$,
\begin{equation*}
  A_1 \leq c_p ( \vert t-s \vert^{\frac{\gamma}{2}}
+\Vert x-y\Vert^{\gamma})^{\frac{p}{2}}.
\end{equation*}
We do not expand on this further since we will be using this method several times below, with details

In order to bound the terms $A_{2,1}$ and $A_{2,2}$, we will also use the factorisation method used in \cite{Sanz:02}. That is, using
the semigroup property of $S$, the Beta function and a stochastic Fubini's theorem (whose assumptions can be seen to be satisfied, see e.g. \cite[Theorem 2.6]{Walsh:86}), we see that, for any $\alpha \in (0,\frac{2-\beta}{4})$,
\begin{equation} \label{fact} \begin{split}
&\int_0^t \int_{\R^k} S(t-\theta,x-\eta) \sum_{j=1}^d D (\sigma_{i,j}(u(\theta,\eta))) M^j(d\theta, d \eta) \\
&\qquad \qquad = \frac{\sin(\pi \alpha)}{\pi} \int_0^t dr \int_{\R^k} dz\, S(t-r,x-z) (t-r)^{\alpha-1} Y^i_{\alpha}(r,z),
\end{split}
\end{equation}
where $Y=(Y^i_{\alpha}(r,z), r \in [0,T], z \in \R^k)$ is the $\mathcal{H}^d_T$-valued process defined by
\begin{equation*}
Y^i_{\alpha}(r,z)=\int_0^r \int_{\R^k} S(r-\theta,z-\eta) (r-\theta)^{-\alpha}
\sum_{j=1}^d D (\sigma_{i,j}(u(\theta,\eta))) M^j(d\theta, d \eta).
\end{equation*}

Let us now bound the $L^p(\Omega; \mathcal{H}^d_T)$-norm of the process $Y$.
Using \cite[(3.13)]{Nualart:07} and the boundedness of the derivatives of the coefficients of $\sigma$, we see that for any $p\geq 2$,
\begin{equation*}
\E \biggl[ \Vert Y^i_{\alpha}(r,z) \Vert_{\mathcal{H}^d_T}^p \biggr] \leq c_p
\sum_{i=1}^d \sup_{ (t,x) \in [0,T] \times \R^k} \E \biggl[ \Vert D (u_i(t,x)) \Vert_{\mathcal{H}_T^d}^p \biggr] (\nu_{r,z})^{p/2},
\end{equation*}
where 
\begin{equation} \label{nu}
\nu_{r,z}:=\Vert S(r-\ast, z -\cdot) (r-\ast)^{-\alpha} \Vert_{\mathcal{H}^d_r}.
\end{equation}
We have that
\begin{equation} \label{eqau}
\begin{split}
\nu_{r,z}
&=\int_0^r ds \int_{\R^k} d \xi \, \Vert \xi \Vert^{\beta-k} (r-s)^{-2 \alpha} \exp(-2\pi^2 (r-s) \Vert \xi \Vert^2) \\
&=\int_0^r ds\, (r-s)^{-2 \alpha-\frac{\beta}{2}} \int_{\R^k} d \tilde \xi \, \Vert \tilde \xi \Vert^{\beta-k}  \exp(-2\pi^2 \Vert \tilde \xi \Vert^2) \\
&=r^{\frac{2-\beta}{2}-2\alpha}.
\end{split}
\end{equation}
Hence, we conclude from (\ref{supderivative}) that
\begin{equation} \label{lemaY}
\sup_{(r,z) \in [0,T] \times \R^k} \E[ \Vert Y^i_{\alpha}(r,z) \Vert_{\mathcal{H}^d_T}^p] <+\infty.
\end{equation}

Now, in order to bound $A_{2,1}$, first note that by (\ref{fact}) we can write
\begin{equation*}
A_{2,1}\leq \E \biggl[\bigg\Vert \int_0^t dr \int_{\R^k} dz\, (\psi_{\alpha}(t-r, x-z)-\psi_{\alpha}(t-r, y-z)) Y_{\alpha}^i(r,z) \bigg\Vert^p_{\mathcal{H}^d_T} \biggr],
\end{equation*}
where $\psi_{\alpha}(t, x)=S(t,x) t^{\alpha-1}$. Then, appealing to Minkowski's inequality, (\ref{lemaY}) and Lemma \ref{lemSS}(a) below,
we find that, for any $\gamma \in (0, 4 \alpha)$,
\begin{equation*} \begin{split}
A_{2,1} &\leq c_p \biggl( \int_0^t dr \int_{\R^k} dz \,\vert \psi_{\alpha}(t-r, x-z)-\psi_{\alpha}(t-r, y-z) \vert \biggr)^p \\
&\qquad \qquad \qquad \times
\sup_{(r,z) \in [0,T] \times \R^k} \E[ \Vert Y^i_{\alpha}(r,z) \Vert_{\mathcal{H}^d_T}^p] \\
&\leq c_p \Vert x-y \Vert^{\frac{\gamma}{2} p}.
\end{split}
\end{equation*}

We next treat $A_{2,2}$. Using (\ref{fact}), we have that $A_{2,2}\leq c_{p,\alpha} (A_{2,2,1}+A_{2,2,2})$, where
\begin{equation*}
\begin{split}
A_{2,2,1}&=\E \biggl[\bigg\Vert \int_0^s dr \int_{\R^k} dz \, (\psi_{\alpha}(t-r, x-z)-\psi_{\alpha}(s-r, x-z)) Y_{\alpha}^i(r,z)\bigg\Vert^p_{\mathcal{H}^d_T} \biggr],\\
A_{2,2,2}&=\E \biggl[\bigg\Vert \int_s^t dr \int_{\R^k} dz \, \psi_{\alpha}(t-r, x-z) Y_{\alpha}^i(r,z) \bigg\Vert^p_{\mathcal{H}^d_T} \biggr].
\end{split}
\end{equation*}
Now, by Minkowski's inequality, (\ref{lemaY}) and Lemma \ref{lemSS}(b) below, we find that, for any $\gamma \in (0, 4 \alpha)$,
\begin{equation*} \begin{split}
A_{2,2,1} &\leq c_p \biggl( \int_0^s dr \int_{\R^k} dz \, \vert \psi_{\alpha}(t-r, x-z)-\psi_{\alpha}(s-r, x-z) \vert \biggr)^p \\
&\qquad \qquad \qquad \times
\sup_{(r,z) \in [0,T] \times \R^k} \E[ \Vert Y^i_{\alpha}(r,z) \Vert_{\mathcal{H}^d_T}^p] \\
&\leq c_p \vert t-s \vert^{\frac{\gamma}{4} p}.
\end{split}
\end{equation*}
In the same way, using Minkowski's inequality, (\ref{lemaY}) and Lemma \ref{lemSS}(c) below, for any $\gamma \in (0, 4\alpha)$, we have that
\begin{equation*} \begin{split}
A_{2,2,2} &\leq c_p \biggl( \int_s^t dr \int_{\R^k} dz \, \psi_{\alpha}(t-r, x-z) \biggr)^p
\sup_{(r,z) \in [0,T] \times \R^k} \E[ \Vert Y^i_{\alpha}(r,z) \Vert_{\mathcal{H}^d_T}^p] \\
&\leq c_p \vert t-s \vert^{\frac{\gamma}{4} p}.
\end{split}
\end{equation*}

   Finally, we bound $A_{3,1}$ and $A_{3,2}$, which can be written
\begin{equation*} \begin{split}
A_{3,1}&=\E \, \biggl[ \bigg\Vert\int_0^t d\theta \int_{\R^k} d \eta \,
(S(t-\theta,x-\eta)-S(t-\theta, y-\eta)) D (b_i(u(\theta,\eta))) \bigg\Vert^p_{\mathcal{H}^d_T}\biggr], \\
A_{3,2}&=\E \, \biggl[ \bigg\Vert\int_0^t d\theta \int_{\R^k} d \eta \,
S(t-\theta,y-\eta) D (b_i(u(\theta,\eta))) \\
&\qquad \qquad \qquad - \int_0^s d\theta \int_{\R^k} d \eta \,
S(s-\theta,y-\eta) D (b_i(u(\theta,\eta))) \bigg\Vert^p_{\mathcal{H}^d_T}\biggr].
\end{split}
\end{equation*}
The factorisation method used above is also needed in this case, that is, using the semigroup property of $S$, the Beta function and Fubini's theorem, we see that for any $\alpha \in (0,1)$,
\begin{equation*} \begin{split}
&\int_0^t d\theta \int_{\R^k} d\eta \, S(t-\theta,x-\eta) D (b_i(u(\theta,\eta))) \\
&\qquad \qquad = \frac{\sin(\pi \alpha)}{\pi} \int_0^t dr \int_{\R^k} dz S(t-r,x-z) (t-r)^{\alpha-1} Z^i_{\alpha}(r,z),
\end{split}
\end{equation*}
where $Z=(Z^i_{\alpha}(r,z), r \in [0,T], z \in \R^k)$ is the $\mathcal{H}^d_T$-valued process defined as
\begin{equation*}
Z^i_{\alpha}(r,z)=\int_0^r d\theta \int_{\R^k} d\eta \, S(r-\theta,z-\eta) (r-\theta)^{-\alpha}D (b_i(u(\theta,\eta))).
\end{equation*}
Hence, we can write
\begin{equation*}
A_{3,1} \leq \E \biggl[\bigg\Vert \int_0^t dr \int_{\R^k} dz\, (\psi_{\alpha}(t-r, x-z)-\psi_{\alpha}(t-r, y-z)) Z_{\alpha}^i(r,z) \bigg\Vert^p_{\mathcal{H}^d_T} \biggr],
\end{equation*}
and $A_{3,2}\leq c_{p,\alpha} (A_{3,2,1}+A_{3,2,2})$, where
\begin{equation*}
\begin{split}
A_{3,2,1}&=\E \biggl[\bigg\Vert \int_0^s dr \int_{\R^k} dz \, (\psi_{\alpha}(t-r, y-z)-\psi_{\alpha}(s-r, y-z)) Z_{\alpha}^i(r,z)\bigg\Vert^p_{\mathcal{H}^d_T} \biggr],\\
A_{3,2,2}&=\E \biggl[\bigg\Vert \int_s^t dr \int_{\R^k} dz \, \psi_{\alpha}(t-r, y-z) Z_{\alpha}^i(r,z) \bigg\Vert^p_{\mathcal{H}^d_T} \biggr].
\end{split}
\end{equation*}

We next compute the $L^p(\Omega; \mathcal{H}^d_T)$-norm for the process $Z$.
Using Minkowski's inequality and the boundedness of the derivatives of the coefficients of $b$, we get that
\begin{equation*}
E[ \Vert Z^i_{\alpha}(r,z) \Vert_{\mathcal{H}^d_T}^p] \leq c_p
\sum_{i=1}^d \sup_{ (t,x) \in [0,T] \times \R^k} \E \biggl[ \Vert D (u_i(t,x)) \Vert_{\mathcal{H}_T^d}^p \biggr] (\gamma_{r,z})^{p/2},
\end{equation*}
where
\begin{equation*} 
\gamma_{r,z}=\int_0^r d\theta \int_{\R^k} d\eta \, S(r-\theta,z-\eta) (r-\theta)^{-\alpha}=r^{1-\alpha}.
\end{equation*}
Hence, using (\ref{supderivative}), we conclude that
\begin{equation} \label{lemZ}
\sup_{(r,z) \in [0,T] \times \R^k} \E[ \Vert Z^i_{\alpha}(r,z) \Vert_{\mathcal{H}^d_T}^p] <+\infty.
\end{equation}

Then, proceeding as above, using Minkowski's inequality, (\ref{lemZ}) and 
Lemma \ref{lemSS}, we conclude that
for any $\gamma \in (0, 4\alpha)$,
\begin{equation*} \begin{split}
A_{3,1} +A_{3,2} \leq c_p (\Vert x-y \Vert^{\frac{\gamma}{2} p} +\Vert t-s \Vert^{\frac{\gamma}{4} p}).
\end{split}
\end{equation*}
This concludes the proof of the proposition for $m=1$. 

   The case $m>1$ follows along the same lines by induction using the stochastic partial differential equation satisfied by the iterated derivatives (cf. \cite[Proposition 6.1]{Nualart:07}).
\end{proof}

   The following lemma was used in the proof of Proposition \ref{derivat}.

\begin{lem} For $\alpha >0$, set $\psi_\alpha(t,x) = S(t,x) t^{\alpha - 1}$, $(t,x) \in \R_+ \times \R^k$.

   \textnormal{(a)} For  $ \alpha \in (0,\frac{2-\beta}{4})$, $\gamma \in (0,4\alpha)$, there is $c>0$ such that for all $t \in [0,T]$, $x,y \in \R^k$, and $\e \in [0,t]$,
$$
\int_{t-\e}^t dr \int_{\R^k} dz \,\vert \psi_{\alpha}(t-r, x-z)-\psi_{\alpha}(t-r, y-z) \vert \leq c \e^{\alpha - \frac{\gamma}{2}}\, \Vert x - y \Vert^{\gamma/2} .
$$

   \textnormal{(b)} For $ \alpha \in (0,\frac{2-\beta}{4})$, $\gamma \in (0,4\alpha)$, there is $c>0$ such that for all $s \leq t \in [0,T]$, $x,y \in \R^k$, and $\e \in [0,s]$,
$$
\int_{s - \e}^s dr \int_{\R^k} dz \, \vert \psi_{\alpha}(t-r, x-z)-\psi_{\alpha}(s-r, x-z) \vert
   \leq c \e^{\alpha - \frac{\gamma}{4}}\, \vert t-s\vert^{\gamma/4}.
$$

   \textnormal{(c)} For $ \alpha \in (0,\frac{2-\beta}{4})$, $\gamma \in (0,4\alpha)$, there is $c>0$ such that for all $s \leq t \in [0,T]$, $x,y \in \R^k$,
$$
   \int_s^t dr \int_{\R^k} dz \, \psi_{\alpha}(t-r, x-z) \leq c\, \vert t-s\vert^{\gamma/4}.
$$
\label{lemSS}
\end{lem}

\begin{proof} (a) This is similar to the proof of \cite[(21)]{Sanz:02}.

   (b) This is similar to the proof of \cite[(14)]{Sanz:02}.
   
   (c) This is a consequence of \cite[(15)]{Sanz:02}.
\end{proof}

\subsection{Study of the Malliavin matrix}

Let $T>0$ be fixed. For $s,t \in [0,T]$, $s \leq t$, and $x,y \in \R^k$
consider the $2d$-dimensional random vector
\begin{equation} \label{Z}
Z:=(u(s,y), u(t,x)-u(s,y)).
\end{equation}
Let $\gamma_Z$ be the Malliavin matrix of $Z$. Note that $\gamma_Z=((\gamma_Z)_{m,l})_{m,l=1,...,2d}$ is a symmetric $2d \times 2d$
random matrix with four $d \times d$ blocks of the form
\begin{equation*}
\gamma_Z =\left(
\begin{array}{ccc}
\gamma_Z^{(1)} & \vdots & \gamma_Z^{(2)}\\
\cdots & \vdots & \cdots \\
\gamma_Z^{(3)}& \vdots & \gamma_Z^{(4)}\\
\end{array}
\right),
\end{equation*}
where
\begin{align*}
\gamma_Z^{(1)} &= (\langle D(u_i(s,y)), D(u_j(s,y)) \rangle_{\mathcal{H}_T^d})_{i,j=1,...,d},\\
\nonumber \gamma_Z^{(2)} &= (\langle D(u_i(s,y)),
D(u_j(t,x)-u_j(s,y)) \rangle_{\mathcal{H}_T^d})_{i,j=1,...,d},\\
\gamma_Z^{(3)} &= (\langle D(u_i(t,x)-u_i(s,y)), D(u_j(s,y)) \rangle_{\mathcal{H}_T^d})_{i,j=1,...,d},\\
\gamma_Z^{(4)} &= (\langle D(u_i(t,x)-u_i(s,y)),
D(u_j(t,x)-u_j(s,y)) \rangle_{\mathcal{H}_T^d})_{i,j=1,...,d}.
\end{align*}
We let {\bf (1)} denote the set of couples $\{1,...,d\} \times \{1,...,d\}$, {\bf (2)} the set $\{1,...,d\} \times \{d+1,...,2d\}$,
{\bf (3)} the set $\{d+1,...,2d\} \times \{1,...,d\}$ and {\bf (4)} the set $\{d+1,...,2d\} \times \{d+1,...,2d\}$.

\vskip 12pt

The following two results follow exactly along the same lines as
\cite[Propositions 6.5 and 6.7]{Dalang2:05} using
(\ref{supderivative}) and Proposition \ref{derivat}, so their proofs
are omitted.

\begin{prop} \label{3p2}
Fix $T>0$ and let $I \times J\subset (0,T] \times \R^k$ be a closed non-trivial rectangle. Let $A_Z$ denotes the cofactor matrix of $\gamma_Z$. Assuming {\bf P1}, for any $p>1$ and $\gamma \in (0, 2-\beta)$, there is a constant $c_{\gamma,p,T}$ such that for any $(s,y), (t,x) \in I \times J$ with $(s,y) \neq (t,x)$, 
\begin{equation*}
\E \, [|(A_Z)_{m,l}|^p]^{1/p}\leq
\begin{cases}
c_{\gamma,p,T} ( \vert t-s\vert^{\gamma/2}
+\Vert x-y\Vert^{\gamma})^{d}&
\text{if \; $(m,l) \in {\bf (1)} $}, \\
c_{\gamma,p,T} ( \vert t-s\vert^{\gamma/2}
+\Vert x-y\Vert^{\gamma})^{d-\frac{1}{2}}&
\text{if \; $(m,l) \in {\bf (2)}$ or ${\bf (3)} $}, \\
c_{\gamma,p,T} ( \vert t-s\vert^{\gamma/2}
+\Vert x-y\Vert^{\gamma})^{d-1}& \text{if \; $(m,l) \in {\bf (4)}$}.
\end{cases}
\end{equation*}
\end{prop}

\begin{prop} \label{dgamma}
Fix $T>0$ and let $I \times J\subset (0,T] \times \R^k$ be a closed non-trivial rectangle. Assuming {\bf P1}, for any $p>1$, $k \geq 1$, and $\gamma \in (0, 2-\beta)$, there is a constant $c_{\gamma,k,p,T}$ such that for any $(s,y), (t,x) \in I \times J$ with $(s,y) \neq (t,x)$,
\begin{equation*}
\E \, \big[\Vert D^k(\gamma_Z)_{m,l}
\Vert_{(\mathcal{H}_T^d)^{\otimes k}} ^{p}\big]^{1/p}\leq
\begin{cases}
c_{\gamma,k,p,T} &
\text{if \; $(m,l) \in {\bf (1)} $}, \\
c_{\gamma,k,p,T} ( \vert t-s\vert^{\gamma/2}
+\Vert x-y\Vert^{\gamma})^{1/2}&
\text{if \; $(m,l) \in {\bf (2)}$ or ${\bf (3)} $}, \\
c_{\gamma,k,p,T} ( \vert t-s\vert^{\gamma/2}
+\Vert x-y \Vert^{\gamma})& \text{if \; $(m,l) \in {\bf (4)}$}.
\end{cases}
\end{equation*}
\end{prop}

The main technical effort in this section is the proof of the following proposition.

\begin{prop} \label{3p3}
Fix $\eta,T>0$. Assume {\bf P1} and {\bf P2}. Let $I \times J
\subset (0,T] \times \R^k$ be a closed non-trivial rectangle.
There exists $C$ depending on $T$ and $\eta$ such that for any
$(s,y), (t,x) \in I \times J$, $(s,y) \neq (t,x)$, and $p>1$,
\begin{equation}\label{eq5.629}
\E \, \big[\big(\textnormal{det} \, \gamma_Z\big)^{-p}\big]^{1/p}
\leq C (|t-s|^{\frac{2-\beta}{2}}+ \Vert x-y
\Vert^{2-\beta})^{-d(1+\eta)}.
\end{equation}
\end{prop}

\begin{proof}
The proof has the same general structure as that of \cite[Proposition 6.6]{Dalang2:05}. We write
\begin{equation} \label{edf}
\text{det} \, \gamma_Z = \prod_{i=1}^{2d} (\xi^i)^{\sf T} \gamma_Z
\xi^i,
\end{equation}
where $\xi=\{\xi^1,...,\xi^{2d}\}$ is an orthonormal basis of
${\R}^{2d}$ consisting of eigenvectors of $\gamma_Z$.

We now carry out the perturbation argument of \cite[Proposition
6.6]{Dalang2:05}. Let ${\bf 0} \in {\R}^d$ and consider the spaces $E_1=\{ (\lambda, {\bf 0}) : \lambda \in {\R}^d\}$ and $E_2=\{ ({\bf 0}, \mu) : \mu \in {\R}^d\}$. Each
$\xi^i$ can be written
\begin{equation} \label{xi}
\xi^i=(\lambda^i, \mu^i)= \alpha_i ({\tilde{\lambda}}^i,{\bf 0}) +
\sqrt{1-\alpha_i^2} \, ({\bf 0}, {\tilde{\mu}}^i),
\end{equation}
where $\lambda^i, \mu^i \in \mathbb{R}^d$,
$({\tilde{\lambda}}^i,{\bf 0}) \in E_1$, $({\bf 0}, {\tilde{\mu}}^i)
\in E_2$, with $\Vert {\tilde{\lambda}}^i \Vert=\Vert
{\tilde{\mu}}^i \Vert=1$ and $0\leq \alpha_i \leq 1$. In particular,
$\Vert \xi^i \Vert^2=\Vert \lambda^i \Vert^2 + \Vert \mu^i
\Vert^2=1$.

   The result of \cite[Lemma 6.8]{Dalang2:05} give us at least $d$ eigenvectors
$\xi^1,...,\xi^d$ that have a ``large projection on $E_1$", and we will show that these will contribute a factor of order $1$ to the product in
(\ref{edf}). Recall that for a fixed small $\alpha_0 > 0$, $\xi^i$ has a ``large projection on $E_1$" if $\alpha_i \geq \alpha_0$.  The at most $d$ other eigenvectors with a ``small projection on $E_1$" will each contribute
a factor of order $(|t-s|^{\frac{2-\beta}{2}}+ \Vert x-y
\Vert^{2-\beta})^{-1-\eta}$, as we will make precise below.

Hence, by \cite[Lemma 6.8]{Dalang2:05} and Cauchy-Schwarz inequality,
one can write
\begin{equation} \begin{split} \label{ak}
\E \big[\big(\textnormal{det} \, \gamma_Z\big)^{-p}\big]^{1/p} &\leq
\sum_{K \subset \{1,...,2d\},\, \vert K \vert = d} \biggl( \E
\biggl[ \1_{A_K} \biggl( \prod_{i \in K} (\xi^i)^{\sf T}
\gamma_Z \xi^i \biggr)^{-2p} \biggr] \biggr)^{1/(2p)} \\
& \qquad \qquad \times \biggl( \E \left[
\left(\inf_{\substack{\xi =(\lambda, \mu) \in \R^{2d} :\\ \| \lambda
\| ^2 + \| \mu \| ^2=1}} \,\xi^{\sf T} \gamma_Z \xi \right)^{-2dp}
\right] \biggr)^{1/(2p)},
\end{split}\end{equation}
where $A_K= \cap_{i \in K} \{\alpha_i \geq \alpha_0\}$.

With this, Propositions \ref{smalle} and \ref{large} below will
conclude the proof of Proposition \ref{3p3}.
\end{proof}

\begin{prop} \label{smalle}
Fix $\eta,T>0$. Assume {\bf P1} and {\bf P2}. There exists $C$
depending on $\eta$ and $T$ such that for all $s,t \in I$, $0 \leq
t-s <1$, $x,y \in J$, $(s,y) \neq (t,x)$, and $p>1$,
\begin{equation}\label{eq5.1029}
\E\left[ \left( \inf_{\substack{\xi =(\lambda, \mu) \in \R^{2d}
:\\ \| \lambda \| ^2 + \| \mu \| ^2=1}} \,\xi^{\sf T} \gamma_Z \xi
\right)^{-2dp} \right] \leq C (|t-s|^{\frac{2-\beta}{2}}+ \Vert x-y
\Vert^{2-\beta})^{-2dp(1+\eta)}.
\end{equation}
\end{prop}

\begin{prop} \label{large}
Assume ${\bf P1}$ and ${\bf P2}$. Fix $T>0$ and $p>1$. Then there
exists $C=C(p,T)$ such that for all $s,t \in I$ with $0 \leq t-s <
\frac{1}{2}$, $x,y \in J$, $(s,y) \neq (t,x)$,
\begin{equation} \label{A}
\E \left[ \1_{A_K} \left( \prod_{i \in K} (\xi^i)^{\sf T} \gamma_Z
\xi^i \right)^{-p} \right] \leq C,
\end{equation}
where $A_K$ is defined just below \textnormal{(\ref{ak})}.
\end{prop}

\begin{proof}[Proof of Proposition \ref{smalle}]
Fix $\gamma \in (0, 2-\beta)$. It suffices to prove this for $\eta$
sufficiently small, in particular, we take $\eta < \gamma /2$. The
proof of this lemma follows lines similar to those of \cite[Proposition 6.9]{Dalang2:05}, with significantly different estimates needed to handle the spatially homogeneous noise.

For $\epsilon \in (0, t-s)$,
\begin{equation*}
\xi^{\sf T} \gamma_Z \xi \geq J_1+J_2,
\end{equation*}
where
\begin{equation} \begin{split} \label{J1}
J_1 &:= \int_{s-\e}^s dr \sum_{l=1}^d \bigg\Vert \sum_{i=1}^d (\lambda_i-\mu_i)
\left(S(s-r,y-\cdot) \sigma_{i,l}(u(r,\cdot))+a_i(l,r,s,y)\right) + W \bigg\Vert_{\mathcal{H}}^2,\\
J_2 &:= \int_{t-\e}^t dr \sum_{l=1}^d \,
\Vert W \Vert_{\mathcal{H}}^2,
\end{split}
\end{equation}
where
\begin{equation} \label{w}
W:=\sum_{i=1}^d \mu_i S(t-r,x-\cdot) \sigma_{i,l}(u(r,\cdot))+\mu_i
a_i(l,r,t,x),
\end{equation}
and $a_i(l,r,t,x)$ is defined in \eqref{eq4.1}. \vskip 12pt We now
consider two different cases. \vskip 12pt
\noindent\emph{Case 1.} Assume $t-s>0$ and
$\Vert x-y \Vert^2\le t-s$. Fix $\e\in\left(0,(t-s) \wedge (\frac14)^{2/\eta}\right)$.
We write
\begin{equation} \label{minimum}
\inf_{\|\xi\|=1} \xi^{\sf T}\gamma_Z\xi \ge
\min\left( \inf_{ \|\xi\|=1\,,\|\mu\|\ge
\e^{\eta/2}} J_2\,, \inf_{\|\xi\|=1\,,\|\mu\|\le\e^{\eta/2}}
J_1\right).
\end{equation}
We will now bound the two terms in the above minimum. We start by
bounding the term containing $J_2$. Using (\ref{sumaqua}) and adding
and subtracting a ``local" term as in (\ref{add}), we find that $J_2
\ge \frac23 J_2^{(1)} -4(J_2^{(2)}+ J_2^{(3)})$, where
\begin{align*} \nonumber
J_2^{(1)}&=\sum_{l=1}^d \int_{t-\e}^t dr \, \int_{\R^k} dv \,
\int_{\R^k} dz \, \Vert v-z \Vert^{-\beta} S(t-r,x-v) S(t-r,x-z)
(\mu^{\sf T} \cdot \sigma(u(r,x)) )_l^2, \\ \nonumber
J_2^{(2)}&=\sum_{l=1}^d \int_{t-\e}^t dr \, \int_{\R^k} dv \,
\int_{\R^k} dz \, \Vert v-z \Vert^{-\beta} S(t-r,x-v) S(t-r,x-z) \\
\nonumber &\qquad \qquad \qquad\times \left(\mu^{\sf T} \cdot
[\sigma(u(r,v)) - \sigma(u(r,x))]\right)_l \left(\mu^{\sf T} \cdot [\sigma(u(r,z))-
\cdot \sigma(u(r,x))]\right)_l,
\\ \nonumber J_2^{(3)}&= \int_{t-\e}^t dr \, \sum_{l=1}^d
\bigg\Vert \sum_{i=1}^d a_i(l,r,t,x) \,\mu_i
\bigg\Vert_{\mathcal{H}}^2,
\end{align*}
Now, hypothesis {\bf P2} and Lemma \ref{scalar} together imply that
$J_2^{(1)} \geq c \, \|\mu\|^2\e^{\frac{2-\beta}{2}}$. Therefore,
\begin{equation} \label{min1}
\inf_{\Vert \xi \Vert=1, \|\mu\|\ge\e^{\eta/2}} J_2
\ge c\e^{\frac{2-\beta}{2}+\eta}-\sup_{\Vert \xi \Vert=1, \|\mu\|\ge\e^{\eta/2}} 2 (\vert J_2^{(2)}\vert + J_2^{(3)}).
\end{equation}
Moreover, (\ref{equa45}) and Lemma \ref{lem:7} imply that for any $q\geq 1$,
\begin{equation} \label{min2}
\E \, \biggl[ \sup_{\Vert \xi \Vert=1, \|\mu\|\ge\e^{\eta/2}} (\vert
J_2^{(2)} \vert +J_2^{(3)})^q \biggr]
\leq c \e^{\frac{2-\beta}{2}q+\frac{\gamma}{2}q}.
\end{equation}
This bounds the first term in (\ref{minimum}) and gives an analogue
of the first inequality in \cite[(6.12)]{Dalang2:05}.

In order to bound the second infimum in (\ref{minimum}), we use
again (\ref{sumaqua}) and we add and subtract a ``local" term as in
(\ref{add}) to see that
\begin{equation*}
J_1 \ge \frac23\, J_1^{(1)} -8(J_1^{(2)} + J_1^{(3)}+J_1^{(4)}+J_1^{(5)}),
\end{equation*}
where
\begin{equation*} \begin{split}
J_1^{(1)}&=\sum_{l=1}^d \int_{s-\e}^s dr \, ((\lambda-\mu)^{\sf T}
\cdot \sigma(u(r,y)))_l^2 \int_{\R^k} d\xi \,
\, \Vert \xi \Vert^{\beta-k} \vert \mathcal{F} S(s-r,y-\cdot) (\xi) \vert^2, \\
J_1^{(2)}&=\sum_{l=1}^d \int_{s-\e}^s dr \, \int_{\R^k} dv \,
\int_{\R^k} dz \, \Vert v-z \Vert^{-\beta} S(s-r,y-v) S(s-r,y-z) \\
& \qquad\times  \left((\lambda-\mu)^{\sf T}
\cdot [\sigma(u(r,v)) - \sigma(u(r,y))]\right)_l \left((\lambda-\mu)^{\sf T} \cdot
[\sigma(u(r,z)) - \sigma(u(r,y))]\right)_l, \\
J_1^{(3)} &:= \int_{s-\e}^s
dr \sum_{l=1}^d \bigg\Vert \sum_{i=1}^d
\mu_i S(t-r,x-\cdot) \sigma_{i,l}(u(r,\cdot)) \bigg\Vert_{\mathcal{H}}^2,\\
J_1^{(4)}&:= \int_{s-\e}^s
dr \sum_{l=1}^d \bigg\Vert \sum_{i=1}^d
(\lambda_i-\mu_i) a_i(l,r,s,y) \bigg\Vert_{\mathcal{H}}^2,\\
J_1^{(5)} &:= \int_{s-\e}^s
dr \sum_{l=1}^d \bigg\Vert \sum_{i=1}^d
\mu_i a_i(l,r,t,x) \bigg\Vert_{\mathcal{H}}^2.
\end{split}
\end{equation*}
Hypothesis {\bf P2} and Lemma \ref{scalar} together imply that $J_1^{(1)} \geq c \, \| \lambda-\mu\|^2\e^{\frac{2-\beta}{2}}$.
Therefore,
\begin{equation} \label{min3}
\inf_{\Vert \xi \Vert=1, \|\mu\|\le\e^{\eta/2}} J_1
\ge \tilde{c} \e^{\frac{2-\beta}{2}}-\sup_{\Vert \xi \Vert=1, \|\mu\|\le\e^{\eta/2}} 8\left(\vert J_1^{(2)}\vert +J_1^{(3)}+J_1^{(4)}+J_1^{(5)}\right).
\end{equation}
Now, (\ref{equa45}) implies that for any $q\geq 1$,
\begin{equation*}
\E \, \biggl[ \sup_{\Vert \xi \Vert=1, \|\mu\|\le\e^{\eta/2}} \vert J_1^{(2)} \vert^q \biggr]
\leq c \e^{\frac{2-\beta}{2}q+\frac{\gamma}{2}q}.
\end{equation*}
Moreover, hypothesis {\bf P1} (in particular, the fact that $\sigma$ is bounded), the Cauchy-Schwarz inequality and
Lemma \ref{scalar} imply that for any $q\geq 1$,
\begin{equation*}
\E \, \biggl[ \sup_{\Vert \xi \Vert=1, \|\mu\|\le\e^{\eta/2}} \vert J_1^{(3)} \vert^q \biggr]
\leq c \e^{\frac{2-\beta}{2}q+\eta q}.
\end{equation*}
Applying Lemma \ref{lem:7} with $t=s$, we get that for any $q\geq 1$,
\begin{equation*}
\E \, \biggl[ \sup_{\Vert \xi \Vert=1, \|\mu\|\le\e^{\eta/2}} \vert J_1^{(4)} \vert^q \biggr]
\leq c \e^{\frac{2-\beta}{2}q+\frac{2-\beta}{2}q}.
\end{equation*}
Again Lemma \ref{lem:7} gives, for any $q\geq 1$,
\begin{equation*}
\E \, \biggl[ \sup_{\Vert \xi \Vert=1, \|\mu\|\le\e^{\eta/2}} \vert J_1^{(5)} \vert^q \biggr]
\leq c \e^{\frac{2-\beta}{2}q+\eta q}.
\end{equation*}
Since we have assumed that $\eta < \frac{\gamma}{4}$, the above
bounds in conjunction prove that for any $q\geq 1$,
\begin{equation} \label{min4}
\E \, \left[ \sup_{\Vert \xi \Vert=1, \|\mu\|\le\e^{\eta/2}} \left(\vert
J_1^{(2)} \vert+J_1^{(3)}+J_1^{(4)}+J_1^{(5)} \right)^q \right]
\leq c \e^{\frac{2-\beta}{2}q+\eta q}.
\end{equation}

We finally use 
(\ref{minimum})--(\ref{min4}) together with \cite[Proposition 3.5]{Dalang2:05} with $\alpha_1=\frac{2-\beta}{2}+\eta$, $\beta_1=\frac{2-\beta}{2}+\frac{\gamma}{4}$, $\alpha_2=\frac{2-\beta}{2}$ and $\beta_2= \frac{2-\beta}{2}+\eta$ to conclude that
\begin{equation*} \begin{split}
\E \left[ \left(\inf_{\|\xi\|=1} \xi^{\sf T}\gamma_Z\xi \right)^{-2pd}\right]
& \le c
\left[ (t-s) \wedge \left(\frac14 \right)^{2/\eta} \right]^{-2pd(\frac{2-\beta}{2}+\eta)}\leq c' (t-s)^{-2pd(\frac{2-\beta}{2}+\eta)}\\
&\le \tilde{c} \left[ (t-s)^{\frac{2-\beta}{2}} + \Vert x-y \Vert^{2-\beta}\right]^{-2pd
(1+\eta^{\prime})},
\end{split}
\end{equation*}
(for the second inequality, we have used the fact that $t-s < 1$, and for the third, that $\Vert x-y\Vert^2 \leq t-s$), whence follows the proposition in the case that $\Vert x-y \Vert^2\le t-s$.

\vskip 12pt

\noindent\emph{Case 2.} Assume that $\Vert x-y \Vert>0$ and $\Vert x-y \Vert^2\ge t-s\ge 0$. Then
\begin{equation*}
\xi^{\sf T} \gamma_Z \xi \ge J_1 + \tilde{J}_2,
\end{equation*}
where $J_1$ is defined in (\ref{J1}),
\begin{equation*}
\tilde{J}_2:=\int_{(t-\e) \vee s}^t dr \sum_{l=1}^d \,\Vert W
\Vert_{\mathcal{H}}^2,
\end{equation*}
and $W$ is defined in (\ref{w}).
Let $\e>0$ be such that
$(1+\alpha)\e^{1/2}<\frac{1}{2} \Vert x-y \Vert$, where
$\alpha>0$ is large but fixed; its specific value will be decided
on later. From here on, Case 2 is divided into two further sub-cases.

\vskip 12pt

\noindent\emph{Sub-Case A.} Suppose that $\e\ge t-s$. Apply
inequality \eqref{sumaqua} and add and subtract a ``local" term as in
(\ref{add}), to find that
\begin{equation*}\begin{split}
J_1 & \geq \frac23\, A_1- 8(A_2+A_3 + A_4+A_5), \\
\tilde J_2 & \geq \frac23\, B_1-4(B_2 + B_3),
\end{split} \end{equation*}
where
\begin{align*}
A_1 &:= \sum_{l=1}^d \int_{s-\e}^s
dr \, \Big\Vert
S(s-r,y-\cdot) \left((\lambda-\mu)^{\sf T}
\cdot \sigma(u(r,y))\right)_l\\
& \qquad \qquad \qquad \qquad + S(t-r,x-\cdot) \left(\mu^{\sf T}
\cdot \sigma(u(r,x))\right)_l \Big\Vert_{\mathcal{H}}^2,\\
A_2 &:= \sum_{l=1}^d \int_{s-\e}^s dr\, \left\Vert S(s-r,y-\cdot)
  \left((\lambda-\mu)^{\sf T} \cdot [\sigma(u(r,\cdot))- \sigma(u(r,y))]\right)_l\right\Vert^2_{\mathcal{H}}\\
A_3 &:= \sum_{l=1}^d \int_{s-\e}^s dr \,  \left\Vert S(t-r,x-\cdot) \left(\mu^{\sf T} \cdot [\sigma(u(r,\cdot))-\sigma(u(r,x))]\right)_l \right\Vert^2_{\mathcal{H}}\\
A_4 &:=\sum_{l=1}^d \int_{s-\e}^s
dr \, \bigg\Vert \sum_{i=1}^d
(\lambda_i-\mu_i) a_i(l,r,s,y) \bigg\Vert_{\mathcal{H}}^2,\\
A_5 &:= \sum_{l=1}^d \int_{s-\e}^s
dr \, \bigg\Vert \sum_{i=1}^d
\mu_i a_i(l,r,t,x) \bigg\Vert_{\mathcal{H}}^2, \\
B_1 &:= \sum_{l=1}^d\int_{s}^t
dr \, \bigg\Vert S(t-r,x-\cdot) (\mu^{\sf T}
\cdot \sigma(u(r,x)))_l\bigg\Vert_{\mathcal{H}}^2, \\
B_2 &:= \sum_{l=1}^d \int_{s}^t dr \, \left\Vert S(t-r,x-\cdot) \left(\mu^{\sf T} \cdot [ \sigma(u(r,\cdot))- \sigma(u(r,x))] \right)_l \right\Vert_{\mathcal{H}}^2, \\
B_3 &:= \sum_{l=1}^d\int_{s}^t
dr \, \bigg\Vert \sum_{i=1}^d
\mu_i a_i(l,r,t,x) \bigg\Vert_{\mathcal{H}}^2.
\end{align*}

Using the inequality $(a+b)^2\geq a^2+b^2-2\vert ab\vert$, we see
that $A_1 \geq \tilde{A_1}+\tilde{A_2}-2\tilde{B_4}$, where
\begin{equation*} \begin{split}
\tilde{A_1}&=\sum_{l=1}^d \int_{s-\e}^s
dr \, \bigg\Vert
S(s-r,y-\cdot) ((\lambda-\mu)^{\sf T}
\cdot \sigma(u(r,y)))_l\bigg\Vert_{\mathcal{H}}^2, \\
\tilde{A_2}&=\sum_{l=1}^d \int_{s-\e}^s
dr \, \bigg\Vert
S(t-r,x-\cdot) (\mu^{\sf T}
\cdot \sigma(u(r,x)))_l\bigg\Vert_{\mathcal{H}}^2, \\
\tilde{B_4}&=\sum_{l=1}^d \int_{s-\e}^s
dr \, \langle S(s-r,y-\cdot) ((\lambda-\mu)^{\sf T}
\cdot \sigma(u(r,y)))_l, S(t-r,x-\cdot) (\mu^{\sf T} \cdot
\sigma(u(r,x)))_l\rangle_{\mathcal{H}}.
\end{split}
\end{equation*}
By hypothesis {\bf P2} and Lemma \ref{scalar}, we see that
\begin{equation*}
\begin{split}
\tilde{A_2}+B_1&=\sum_{l=1}^d \int_{s-\e}^t
dr \, \bigg\Vert
S(t-r,x-\cdot) (\mu^{\sf T}
\cdot \sigma(u(r,x)))_l\bigg\Vert_{\mathcal{H}}^2 \\
&\geq \sum_{l=1}^d \int_{t-\e}^t
dr \, \bigg\Vert
S(t-r,x-\cdot) (\mu^{\sf T}
\cdot \sigma(u(r,x)))_l\bigg\Vert_{\mathcal{H}}^2 \\
&\geq \Vert \mu \Vert^2 \epsilon^{\frac{2-\beta}{2}}.
\end{split}
\end{equation*}
Similarly, $\tilde{A_1}\geq \Vert\lambda- \mu \Vert^2
\epsilon^{\frac{2-\beta}{2}}$, and so
\begin{equation} \label{case21}
\tilde{A_1}+\tilde{A_2}+B_1 \geq (\Vert\lambda- \mu \Vert^2+\Vert
\mu \Vert^2) \epsilon^{\frac{2-\beta}{2}}\geq c
\epsilon^{\frac{2-\beta}{2}}.
\end{equation}

Turning to the terms that are to be bounded above, we see as in \eqref{equa45} that
$$\E [ \vert A_2
\vert^q ]\leq c \epsilon^{\frac{2-\beta}{2}q+\frac{\gamma}{2}q}, \, \text{ and } \, \E [ \vert B_2 \vert^q] \leq c
\epsilon^{\frac{2-\beta}{2}q+\frac{\gamma}{2}q}.$$
Using Lemma
\ref{lem:7} and the fact that $t-s \leq \epsilon$, we see that
$$
\E[\vert B_3 \vert^q] \leq c \epsilon^{(2-\beta)q}, \quad \E[\vert A_4
\vert^q] \leq c \epsilon^{(2-\beta)q}, \; \text{ and } \, \E[\vert A_5 \vert^q]
\leq c \epsilon^{(2-\beta)q}.
$$

In order to bound the $q$-th moment of $A_3$, we proceed as we did
for the random variable $A_2$ in (\ref{add}). It suffices to bound the
$q$-th moment of
\begin{equation*} \begin{split}
&\sum_{l=1}^d \int_{s-\e}^s dr \int_{\mathbb{R}^k} dv
\int_{\mathbb{R}^k} dz \Vert v-z \Vert^{-\beta} S(t-r,x-v) S(t-r,
x-z)\\
&\qquad \qquad \times \left(\mu^{\sf T} \cdot [\sigma(u(r,v))-\sigma(u(r,x))]\right)_l
  \left(\mu^{\sf T} \cdot [\sigma(u(r,z))-\sigma(u(r,x))]\right)_l.
\end{split}
\end{equation*}
Using H\"older's inequality, the Lipschitz property of $\sigma$ and
(\ref{equa3}), this $q$-th moment is bounded by
\begin{equation*} \begin{split}
&\biggl(\int_{s-\e}^s dr \int_{\mathbb{R}^k} dv \int_{\mathbb{R}^k}
dz \Vert v-z \Vert^{-\beta} S(t-r,x-v) S(t-r,
x-z)\biggr)^{q-1}\\
& \qquad \times \int_{s-\e}^s dr \int_{\mathbb{R}^k} dv
\int_{\mathbb{R}^k} dz \Vert v-z \Vert^{-\beta} S(t-r,x-v) S(t-r,
x-z) \Vert v-x \Vert^{\frac{\gamma q}{2}}\, \Vert z-x \Vert^{\frac{\gamma q}{2}} \\
&\quad=:a_1\times a_2.
\end{split}
\end{equation*}
By Lemma \ref{scalar}, $a_1 \leq \epsilon^{\frac{2-\beta}{2}(q-1)}$.
For $a_2$, we use the change of variables
$\tilde{v}=\frac{x-v}{\sqrt{t-r}}$,
$\tilde{z}=\frac{x-z}{\sqrt{t-r}}$, to see that
\begin{equation*} \begin{split}
a_2&=\int_{s-\e}^s dr \int_{\mathbb{R}^k} d\tilde{v}
\int_{\mathbb{R}^k} d\tilde{z}\, \Vert \tilde{v}-\tilde{z}
\Vert^{-\beta} (t-r)^{-\beta/2} S(1,\tilde{v}) S(1, \tilde{z})\, \Vert
\tilde{v} \Vert^{\frac{\gamma q}{2}} \Vert
\tilde{z} \Vert^{\frac{\gamma q}{2}} (t-r)^{\frac{\gamma q}{2}} \\
&=\int_{s-\e}^s dr\, (t-r)^{\frac{\gamma
q}{2}-\frac{\beta}{2}}\int_{\mathbb{R}^k} d\tilde{v}
\int_{\mathbb{R}^k} d\tilde{z} \, S(1,\tilde{v}) S(1, \tilde{z})
\Vert \tilde{v}-\tilde{z} \Vert^{-\beta} \Vert \tilde{v}
\Vert^{\frac{\gamma q}{2}} \Vert \tilde{z}
\Vert^{\frac{\gamma q}{2}} \\
&=c \biggl( (t-s+\epsilon)^{\frac{2-\beta}{2}+\frac{\gamma
q}{2}}-(t-s)^{\frac{2-\beta}{2}+\frac{\gamma q}{2}} \biggr) \\
&\leq c \, \epsilon^{\frac{2-\beta}{2}+\frac{\gamma q}{2}},
\end{split}
\end{equation*}
since $t-s<\epsilon$. Putting together this bounds for $a_1$ and $a_2$
yields $\E[\vert A_3 \vert^q] \leq c
\epsilon^{\frac{2-\beta}{2}+\frac{\gamma q}{2}}$.

We now study the term $\tilde{B_4}$, with the objective of showing
that $\tilde{B_4} \leq \Phi(\alpha) \epsilon^{\frac{2-\beta}{2}}$,
with $\lim_{\alpha \rightarrow +\infty} \Phi(\alpha)=0$. We note
that by hypothesis {\bf P1},
\begin{equation*} \begin{split}
\tilde{B_4} &\leq c \int_{s-\e}^s dr \int_{\mathbb{R}^k} dv
\int_{\mathbb{R}^k} dz \, \Vert v-z \Vert^{-\beta} S(s-r,y-v) S(t-r,
x-z) \\
&= c \int_{s-\e}^s dr \int_{\mathbb{R}^k} dv \,\Vert v \Vert^{-\beta}
(S(s-r,y-\cdot) \ast S(t-r, \cdot-x))(v) \\
&= c \int_{s-\e}^s dr \int_{\mathbb{R}^k} dv \,\Vert v \Vert^{-\beta}
S(t+s-2r,y-x+v),
\end{split}
\end{equation*}
where we have used the semigroup property of $S(t,v)$. Using the change of variables $\bar r = s-r$, it follows
that
\begin{equation*} \begin{split}
\tilde{B_4} &\leq 
  c \int_{0}^{\e} d\bar r \int_{\mathbb{R}^k} dv \, \Vert v \Vert^{-\beta}
(t-s+2\bar r)^{-k/2} \exp \biggl(-\frac{\Vert y-x+v\Vert^2}{2(t-s+2\bar r)}
\biggr) \\
&=:c(I_1+I_2),
\end{split}
\end{equation*}
where
\begin{equation*} \begin{split}
I_1 &=\int_{0}^{\e} dr \int_{\Vert v \Vert < \sqrt{r}(1+\alpha)} dv \,
\Vert v \Vert^{-\beta} (t-s+2r)^{-k/2} \exp \biggl(-\frac{\Vert
y-x+v\Vert^2}{2(t-s+2r)} \biggr), \\
I_2&=\int_{0}^{\e} dr \int_{\Vert v \Vert \geq \sqrt{r}(1+\alpha)}
dv \, \Vert v \Vert^{-\beta} (t-s+2r)^{-k/2} \exp \biggl(-\frac{\Vert
y-x+v\Vert^2}{2(t-s+2r)} \biggr).
\end{split}
\end{equation*}
Concerning $I_1$, observe that when $\Vert v \Vert <
\sqrt{r}(1+\alpha)$, then
\begin{equation*} \begin{split}
\Vert y-x+v \Vert \geq \Vert y-x \Vert-\Vert v \Vert \geq \Vert y-x
\Vert- \sqrt{\e} (1+\alpha) \geq \frac12 \Vert y-x \Vert \geq \alpha
\sqrt{\e},
\end{split}
\end{equation*}
since we have assumed that $(1+\alpha) \sqrt{\e} <\frac12 \Vert
y-x \Vert$. Therefore,
\begin{equation*} \begin{split}
I_1 &\leq \int_{0}^{\e} dr \, (t-s+2r)^{-k/2} \exp
\biggl(-\frac{\alpha^2 \e}{2(t-s+2r)} \biggr) \int_{\Vert v \Vert <
\sqrt{r}(1+\alpha)} dv\, \Vert v \Vert^{-\beta},
\end{split}
\end{equation*}
and the $dv$-integral is equal to $(1+\alpha)^{k-\beta}
r^{\frac{k-\beta}{2}}$, so
\begin{equation*} \begin{split}
I_1 &\leq (1+\alpha)^{k-\beta} \int_{0}^{\e} dr \, (t-s+2r)^{-k/2}
r^{\frac{k-\beta}{2}} \exp \biggl(-\frac{\alpha^2 \e}{2(t-s+2r)}
\biggr) \\
&\leq (1+\alpha)^{k-\beta} \int_{0}^{\e} dr \, (t-s+2r)^{-\beta/2}
\exp \biggl(-\frac{\alpha^2 \e}{2(t-s+2r)}
\biggr),
\end{split}
\end{equation*}
where the second inequality uses the fact that $k-\beta>0$.
Use the change of variables $\rho=\frac{t-s+2r}{\alpha^2 \e}$ and the inequality $t-s \leq \e$ to see that
\begin{equation*} \begin{split}
I_1 &\leq (1+\alpha)^{k-\beta} \int_{\frac{t-s}{\alpha^2
\e}}^{\frac{t-s+2\e}{\alpha^2 \e}} d\rho \, \alpha^2 \e \,(\alpha^2
\e \rho)^{-\beta/2} \exp \biggl(-\frac{1}{2\rho}
\biggr) \\
&\leq \e^{\frac{2-\beta}{2}} (1+\alpha)^{k-\beta} \alpha^{2-\beta}
\int_0^{3/\alpha^2} d\rho \, \rho^{-\beta/2} \exp
\biggl(-\frac{1}{2\rho} \biggr)=:\e^{\frac{2-\beta}{2}}
\Phi_1(\alpha).
\end{split}
\end{equation*}
We note that $\lim_{\alpha \rightarrow +\infty} \Phi_1(\alpha)=0$.

Concerning $I_2$, note that
\begin{equation*} \begin{split}
I_2& \leq \int_{0}^{\e} dr \int_{\Vert v \Vert > \sqrt{r}(1+\alpha)}
dv\, r^{-\beta/2} (1+\alpha)^{-\beta} (t-s+2r)^{-k/2} \exp
\biggl(-\frac{\Vert y-x+v\Vert^2}{2(t-s+2r)} \biggr) \\
&\leq (1+\alpha)^{-\beta} \int_{0}^{\e} dr \, r^{-\beta/2}
\int_{\mathbb{R}^k} dv \, (t-s+2r)^{-k/2} \exp
\biggl(-\frac{\Vert y-x+v\Vert^2}{2(t-s+2r)} \biggr) \\
&= c (1+\alpha)^{-\beta} \e^{\frac{2-\beta}{2}}.
\end{split}
\end{equation*}
We note that $\lim_{\alpha \rightarrow +\infty}
(1+\alpha)^{-\beta}=0$, and so we have shown that $\tilde{B_4} \leq
\Phi(\alpha) \e^{\frac{2-\beta}{2}}$, with $\lim_{\alpha \rightarrow
+\infty} \Phi(\alpha)=0$.

Using (\ref{case21}), we have shown that
\begin{equation*} \begin{split}
\inf_{\Vert \xi \Vert=1} \xi^{\sf T} \gamma_Z \xi &\geq \frac23\, A_1 - 8(A_2+A_3 + A_4+A_5)+\frac23\, B_1 - 4(B_2 + B_3) \\
&\geq \frac23 (\tilde{A_1}+\tilde{A_2} + B_1)-\frac43\, \tilde{B_4} -8( 
A_2 + A_3  + A_4 + A_5) -4( B_2+ B_3) \\
&\geq \frac23\, c\, \e^{\frac{2-\beta}{2}}-4 \Phi(\alpha)
\e^{\frac{2-\beta}{2}}-Z_{1, \e},
\end{split}
\end{equation*}
where $\E[\vert Z_{1, \e} \vert^q] \leq
\e^{\frac{2-\beta}{2}q+\frac{\gamma}{2}q}$. We choose $\alpha$ large
enough so that $\Phi(\alpha) < \frac{1}{12} c$, to get
\begin{equation*}
\inf_{\Vert \xi \Vert=1} \xi^{\sf T} \gamma_Z \xi \geq \frac13 c
\e^{\frac{2-\beta}{2}}-Z_{1,\e}.
\end{equation*}

\vskip 12pt \noindent\emph{Sub-Case B.}
Suppose that $\e\le t-s\le |x-y|^2.$ As in (\ref{minimum}), we have
\begin{equation*}
\inf_{\|\xi\|=1} \xi^{\sf T}\gamma_Z\xi \ge
\min\left( c \e^{\frac{2-\beta}{2}+\eta}-Y_{1, \e}, c \e^{\frac{2-\beta}{2}} -Y_{2, \e}
\right),
\end{equation*}
where $\E[ \vert Y_{1, \e} \vert^q] \leq c \,
\e^{\frac{2-\beta}{2}q+\frac{\gamma}{2}q}$ and $\E[ \vert Y_{2, \e}
\vert^q] \leq c \, \e^{\frac{2-\beta}{2}q+\eta q}$. This suffices
for {\it Sub-Case B}. 
\vskip 12pt 

   Now, we combine {\it Sub-Cases A} and {\it B} to see that for $0< \e < \frac14
(1+\alpha)^{-2} \Vert x-y \Vert^2$,
\begin{equation*}
\inf_{\|\xi\|=1} \xi^{\sf T}\gamma_Z\xi \ge
\min\left( c \e^{\frac{2-\beta}{2}+\eta}-Y_{1, \e}, c \e^{\frac{2-\beta}{2}} -Y_{2, \e} 1_{\{\e \leq t-s\}} -Z_{1, \e} 1_{\{ t-s<\e\}}\right).
\end{equation*}
By \cite[Proposition 3.5]{Dalang2:05}, we see that
\begin{equation*} \begin{split}
\E \biggl[ \biggl(\inf_{\Vert \xi \Vert=1} \xi^{\sf T} \gamma_Z \xi
\biggr)^{-2dp} \biggr] &\leq c \Vert x-y
\Vert^{2(-2dp)(\frac{2-\beta}{2}+\eta)} \\
&\leq c (\vert t-s \vert +\Vert x-y
\Vert^2)^{-2dp(\frac{2-\beta}{2}+\eta)}\\
&\leq c (\vert t-s \vert^{\frac{2-\beta}{2}} +\Vert x-y
\Vert^{2-\beta})^{-2dp(1+\tilde{\eta})}
\end{split}
\end{equation*}
(in the second inequality, we have used the fact that $\Vert x-y\Vert^2 \geq t-s$). This concludes the proof of Proposition \ref{smalle}. 
\end{proof}

\begin{proof}[Proof of Proposition \ref{large}]

Let $0< \e < s\leq t$. Fix $i_0 \in \{1,\dots,2d\}$ and write $\tilde{\lambda}^{i_0}=(\tilde{\lambda}_1^{i_0},...,\tilde{\lambda}_d^{i_0})$ and
$\tilde{\mu}^{i_0}=(\tilde{\mu}_1^{i_0},...,\tilde{\mu}_d^{i_0})$. We look at $(\xi^{i_0})^{\sf T} \gamma_Z \xi^{i_0}$ on the event $\{\alpha_{i_0} \geq \alpha_0\}$. As in the proof of Proposition \ref{smalle} and using the notation from (\ref{xi}), this is bounded below by
\begin{equation}\label{large1}
\begin{split}
& \int_{s-\epsilon}^s dr \sum_{l=1}^d \bigg\Vert \sum_{i=1}^d
\biggl[ \bigg(\alpha_{i_0}
\tilde{\lambda}^{i_0}_i S(s-r, y-\cdot)\\
&\qquad +\tilde{\mu}_i^{i_0} \sqrt{1-\alpha_{i_0}^2}
\left( S(t-r, x-\cdot)-S(s-r, y-\cdot) \right)\bigg)
\sigma_{i, l}(u(r,\cdot))\\
&\qquad + \alpha_{i_0} \tilde{\lambda}_i^{i_0}
a_i(l,r,s,y)\\
&\qquad +\tilde{\mu}_i^{i_0} \sqrt{1-\alpha_{i_0}^2}
\left( a_i(l,r,t,x)-a_i(l,r,s,y) \right) \biggr] \bigg\Vert_{\mathcal{H}}^2 \\
&\qquad + \int_{s \vee (t-\epsilon)}^t
dr \sum_{l=1}^d \bigg\Vert \sum_{i=1}^d \biggl[ \tilde{\mu}_i^{i_0}
\sqrt{1-\alpha_{i_0}^2} S(t-r, x-\cdot) \sigma_{i,l}(u(r,\cdot)) \\
& \qquad + \tilde{\mu}_i^{i_0} \sqrt{1-\alpha_{i_0}^2}\
a_i(l,r,t,x) \biggr] \bigg\Vert_{\mathcal{H}}^2.
\end{split}\end{equation}
We seek lower bounds for this expression for $0 < \varepsilon < \varepsilon_0$, where $\varepsilon_0 \in (0,\frac{1}{2})$. In the remainder of this proof, we will use the generic notation $\alpha, \tilde \lambda$ and $\tilde \mu$ for the realizations $\alpha_{i_0}(\omega)$, $\tilde{\lambda}^{i_0}(\omega)$, and $\tilde{\mu}^{i_0}(\omega)$. Our proof follows the structure of \cite[Theorem 3.4]{Dalang:11}, rather than \cite[Proposition 6.13]{Dalang2:05}.
\vskip 12pt
\noindent{\bf Case 1.} $t-s > \e$. Fix $\gamma \in (0, 2-\beta)$ and let $\eta$ be
such that $\eta < \gamma /2$. We note that
\begin{equation*}
\inf_{1 \geq \alpha \geq \alpha_0}
\left(\xi^{i_0} \right)^{\sf T}
\gamma_Z \xi^{i_0} :=\min(E_{1,\e},E_{2,\e}),
\end{equation*}
where
\begin{equation*}
E_{1,\e} := \inf_{\alpha_0 \leq \alpha \leq \sqrt{1-\e^\eta}}
\left(\xi^{i_0} \right)^{\sf T} \gamma_Z \xi^{i_0},\qquad
E_{2,\e} := \inf_{\sqrt{1-\e^\eta} \leq \alpha \leq 1}
\left(\xi^{i_0} \right)^{\sf T} \gamma_Z \xi^{i_0}.
\end{equation*}

Using (\ref{sumaqua}) and (\ref{large1}), we see that
\begin{equation*}
E_{1,\e} \geq \inf_{\alpha_0 \leq \alpha \leq \sqrt{1-\e^\eta}} \left(\frac23 G_{1,\epsilon}-2 \bar G_{1,\epsilon} \right),
\end{equation*}
where
\begin{equation*} \begin{split}
G_{1,\epsilon} & := (1-\alpha^2) \int_{s \vee (t-\epsilon)}^t dr
\, \sum_{l=1}^d \Vert (\tilde{\mu}^{\sf T} \cdot \sigma(u(r,\cdot)))_l\, S(t-r,x -\cdot)
\Vert_\cH^2, \\
\bar G_{1,\e} &:= \int_{t-\e}^t dr \sum_{l=1}^d \bigg\Vert \sum_{i=1}^d
\tilde\mu_i \sqrt{1 - \alpha^2}\ a_i(l,r,t,x) \bigg\Vert_{\mathcal{H}}^2.
\end{split}
\end{equation*}

Using the same ``localisation argument'' as in the proof
of Proposition \ref{det} (see \eqref{equa45}), we have that there exists a random variable $W_{\epsilon}$ such that
\begin{equation} \label{f2}
G_{1,\epsilon} \geq \rho^2 c (1-\alpha^2) ( (t-s) \wedge \epsilon)^{\frac{2-\beta}{2}}-2 W_{\epsilon},
\end{equation}
where, for any $q \geq 1$,
\begin{equation*}
\E [\vert W_{\epsilon} \vert^q]\leq c_q \epsilon^{\frac{2-\beta}{2}q+\frac{\gamma}{2}q}.
\end{equation*}
Hence, using the fact that $1-\alpha^2 \geq \e^{\eta}$ and $t-s >\e$, we deduce that
\begin{equation*}
E_{1,\e} \geq c \e^{\frac{2-\beta}{2}+\eta}-2 W_{\epsilon} - 2 \bar G_{1,\e},
\end{equation*}
where, from Lemma \ref{lem:7}, $\E \, [| \bar G_{1,\e}|^q] \leq c_q \e^{(2-\beta) q}$, for any $q\geq 1$,

We now estimate $E_{2,\e}$. Using (\ref{sumaqua}) and (\ref{large1}), we see that
\begin{equation*}
E_{2,\e} \geq \frac{2}{3} G_{2,\e} - 8 (\bar G_{2,1,\e} + \bar G_{2,2,\e} +
\bar G_{2,3,\e} + \bar G_{2,4,\e}),
\end{equation*}
where
\begin{equation*}\begin{split}
G_{2,\e} &:= \alpha^2 \int_{s-\epsilon}^s dr
\sum_{l=1}^d \Vert (\tilde{\lambda}^{\sf T} \cdot \sigma(u(r,\cdot)))_l\, S(s-r,y -\cdot)
\Vert_\cH^2,\\
\bar G_{2,1,\e} &:= (1-\alpha^2) \int_{s-\epsilon}^s dr
\sum_{l=1}^d \Vert (\tilde{\mu}^{\sf T} \cdot \sigma(u(r,\cdot)))_l\, S(t-r,x -\cdot)
\Vert_\cH^2,\\
\bar G_{2,2,\e} &:= (1-\alpha^2) \int_{s-\epsilon}^s dr
\sum_{l=1}^d \Vert (\tilde{\mu}^{\sf T} \cdot \sigma(u(r,\cdot)))_l\, S(s-r,y -\cdot)
\Vert_\cH^2, \\
\bar G_{2,3,\e} &:= \int_{s-\e}^s dr \sum_{l=1}^d \bigg\Vert \sum_{i=1}^d
\left(\alpha \tilde\lambda_i - \tilde\mu_i \sqrt{1 - \alpha^2}
\right) a_i(l,r,s,y)
\bigg\Vert_{\mathcal{H}}^2, \\
\bar G_{2,4,\e} &:= (1-\alpha^2) \int_{s-\e}^s dr \sum_{l=1}^d \bigg\Vert \sum_{i=1}^d
\tilde\mu_i a_i(l,r,t,x) \bigg\Vert_{\mathcal{H}}^2.
\end{split}\end{equation*}

As for the term $G_{1,\e}$ in (\ref{f2}) and using the fact that $\alpha^2 \geq 1 - \e^\eta$, we get that
$$G_{2,\e} \geq c \e^{\frac{2-\beta}{2}}-2 W_{\epsilon},$$
where, for any $q \geq 1$,
$
\E [\vert W_{\epsilon} \vert^q]\leq c_q \epsilon^{\frac{2-\beta}{2}q+\frac{\gamma}{2}q}.
$
On the other hand, since $1-\alpha^2 \leq \e^\eta$, we can use hypothesis {\bf P1} and Lemma \ref{scalar} to see that
$$\E \, [| \bar G_{2,1,\e}|^q] \leq
c_q \e^{(\frac{2-\beta}{2}+\eta)q},$$
and similarly, using Lemma \ref{scalar},
$$\E \, [| \bar G_{2,2,\e}|^q] \leq
c_q \e^{(\frac{2-\beta}{2}+\eta)q}.$$
Finally, using Lemma \ref{lem:7}, we have that
$$
\E \, [|\bar G_{2,3,\e}|^q] \leq
c_q \e^{(2-\beta) q}, \text{ and } \E \, [| \bar G_{2,4,\e}|^q] \leq
c_q \e^{\eta q} (t-s+\e)^{\frac{2-\beta}{2}q} \e^{\frac{2-\beta}{2}q} \leq c_q \e^{(\frac{2-\beta}{2}+\eta)q}.
$$
We conclude that $E_{2,\e} \geq c \e^\frac{2-\beta}{2} - J_\e$,
where $\E[|J_\e |^q] \leq c_q \e^{(\frac{2-\beta}{2}+\eta)q}$. Therefore,
when $t-s > \e$,
\begin{equation*}
1_{\{\alpha_{i_0} \geq \alpha_0 \}}\
\left(\xi^{i_0}\right)^{\sf T}
\gamma_Z \xi^{i_0}
\geq 1_{\{\alpha_{i_0} \geq \alpha_0 \}}\
\min\left(c \e^{\frac{2-\beta}{2}+\eta} - V_{\e} ~,~ c \e^\frac{2-\beta}{2}
- J_\e\right),
\end{equation*}
where $\E [\vert V_{\epsilon} \vert^q]\leq c_q \epsilon^{\frac{2-\beta}{2}q+\frac{\gamma}{2}q}$.
\vskip 12pt

\noindent{\bf Case 2.} $t-s \leq \e$, $\frac{\vert x-y \vert^2}{\delta_0}\leq \epsilon$. The constant $\delta_0$ will be chosen sufficiently large (see \eqref{eqdelta0}). Fix $\theta \in (0,\frac{1}{2})$ and $\gamma \in (0, 2-\beta)$. From (\ref{sumaqua}) and (\ref{large1}), we have that
\begin{equation*}
1_{\{\alpha_{i_0} \geq \alpha_0 \}}\ \left(\xi^{i_0}\right)^{\sf T} \gamma_Z \xi^{i_0}
\geq \frac23 G_{3, \e^{\theta}}-8 (\bar G_{3,1,\e^{\theta}} - \bar G_{3,2,\e^{\theta}} -
\bar G_{3,3,\e^{\theta}} -\bar G_{3,4,\e^{\theta}}),
\end{equation*}
where
\begin{equation*}\begin{split}
G_{3,\e^{\theta}} &:= \alpha^2 \int_{s-\e^{\theta}}^s dr
\sum_{l=1}^d \Vert (\tilde{\lambda}^{\sf T} \cdot \sigma(u(r,y)))_l\, S(s-r,y -\cdot)
\Vert_\cH^2,\\
\bar G_{3,1,\e^{\theta}} &:= \alpha^2 \int_{s-\e^{\theta}}^s dr
\sum_{l=1}^d \Vert (\tilde{\lambda}^{\sf T} \cdot (\sigma(u(r,\cdot))-\sigma(u(r,y))) )_l\, S(s-r,y -\cdot)
\Vert_\cH^2,\\
\bar G_{3,2,\e^{\theta}}&:= \int_{s-\e^{\theta}}^s dr \sum_{l=1}^d \bigg\Vert \sum_{i=1}^d
\left(\alpha \tilde\lambda_i - \tilde\mu_i \sqrt{1 - \alpha^2}
\right) a_i(l,r,s,y)
\bigg\Vert_{\mathcal{H}}^2, \\
\bar G_{3,3,\e^{\theta}} &:=(1 - \alpha^2) \int_{s-\e^{\theta}}^s dr \sum_{l=1}^d \bigg\Vert
(\tilde{\mu}^{\sf T} \cdot \sigma(u(r,\cdot)))_l\, (S(t-r,x -\cdot)-S(s-r,y -\cdot))
\bigg\Vert_{\mathcal{H}}^2, \\
\bar G_{3,4,\e^{\theta}} &:=(1 - \alpha^2) \int_{s-\e^{\theta}}^s dr \sum_{l=1}^d \bigg\Vert \sum_{i=1}^d
\tilde\mu_i \ a_i(l,r,t,x) \bigg\Vert_{\mathcal{H}}^2.
\end{split}\end{equation*}
By hypothesis {\bf P2} and Lemma \ref{scalar}, since $t-s \leq \e$ and $\alpha \geq \alpha_0$, we have that
\begin{equation*}
G_{3,\e^{\theta}} \geq \alpha^2_0 c \e^{\theta \frac{2-\beta}{2}}.
\end{equation*}

   As in the proof of Proposition \ref{det} (see in particular \eqref{add} to \eqref{equa45}), 
we get that for any $q \geq 1$,
\begin{equation*}
\E \, [\vert \bar G_{3,1,\e^{\theta}} \vert^q ]\leq C \e^{\theta (\frac{2-\beta}{2}+\frac{\gamma}{2})q}.
\end{equation*}

Appealing to Lemma \ref{lem:7} and using the fact that $t-s \leq \e$, we see that
$$
  \E \, [|\bar G_{3,2,\e^{\theta}}|^q] \leq c_q \e^{\theta (2-\beta) q}
$$
and
$$
    \E \, [| \bar G_{3,4,\e^{\theta}}|^q] \leq c_q (t-s+\e^{\theta})^{\frac{2-\beta}{2}q} \e^{\theta \frac{2-\beta}{2}q} \leq c_q \e^{\theta (2-\beta)q}.
$$

It remains to find an upper bound for $\bar G_{3,3,\e^{\theta}}$. From Burkholder's inequality, for any $q \geq 1$,
\begin{equation} \label{wo}
\E \, [\vert \bar G_{3,3,\e^{\theta}} \vert^q ]\leq c_q (W_{1,\e^{\theta}}+W_{2, \e^{\theta}}),
\end{equation}
where
\begin{equation*}
\begin{split}
W_{1,\e^{\theta}}&= \E \left[ \bigg\vert \int_{s-\e^{\theta}}^s \int_{\R^k} (S(t-r, x-z)-S(s-r,x-z)) \sum_{l=1}^d 
(\tilde{\mu}^{\sf T} \cdot \sigma(u(r,z)))_l M^l(dr, dz) \bigg\vert^{2q} \right], \\
W_{2,\e^{\theta}}&= \E \left[ \bigg\vert \int_{s-\e^{\theta}}^s \int_{\R^k} (S(s-r, x-z)-S(s-r,y-z)) \sum_{l=1}^d (\tilde{\mu}^{\sf T} \cdot \sigma(u(r,z)))_l M^l(dr, dz) \bigg\vert^{2q} \right].
\end{split}
\end{equation*}

   As in the proof of Proposition \ref{derivat}, using the semigroup property of $S$, the Beta function and a stochastic Fubini's theorem (whose assumptions can be seen to be satisfied, see e.g. \cite[Theorem 2.6]{Walsh:86}), we see that for any $\alpha \in (0,\frac{2-\beta}{4})$,
\begin{equation} \label{fub2}
\begin{split}
&\int_{s-\e^{\theta}}^s \int_{\R^k} S(s-v,y-\eta)\sum_{l=1}^d 
(\tilde{\mu}^{\sf T} \cdot \sigma(u(v,\eta)))_l M^l (dv, d \eta) \\
&\qquad \qquad = \frac{\sin(\pi \alpha)}{\pi} \int_{s-\e^{\theta}}^s dr \int_{\R^k} dz\, S(s-r,y-z) (s-r)^{\alpha-1} Y_{\alpha}(r,z)
\end{split}
\end{equation}
where $Y=(Y_{\alpha}(r,z), r \in [0,T], z \in \R^k)$ is the real valued process defined as
\begin{equation*}
Y_{\alpha}(r,z)=\int_{s-\e^\theta}^r \int_{\R^k} S(r-v,z-\eta) (r-v)^{-\alpha} \sum_{l=1}^d 
(\tilde{\mu}^{\sf T} \cdot \sigma(u(v,\eta)))_l M^l(dv, d \eta).
\end{equation*}

We next estimate the $L^p(\Omega)$-norm of the process $Y$.
Using Burkholder's inequality, the boundedness of the coefficients of $\sigma$, and the change variables $\tilde \xi=\sqrt{r-v}\, \xi$, we see that
\begin{equation*}
\begin{split}
\E[ \vert Y_{\alpha}(r,z) \vert^p] &\leq c_p \left(\int_{s-\e^{\theta}}^r dv \int_{\R^k} d\xi\,
\Vert \xi \Vert^{\beta-k} \vert \mathcal{F} S(r-v,z-\cdot) (r-v)^{-\alpha} (\xi) \vert^2 \right)^{\frac{p}{2}} \\
&=c_p \left(\int_{s-\e^{\theta}}^r dv \int_{\R^k} d\xi\,
\Vert \xi \Vert^{\beta-k} (r-v)^{-2\alpha} e^{-4 \pi^2 (r-v) \Vert \xi \Vert^2} \right)^{\frac{p}{2}} \\
&=c_p \left(\int_{s-\e^{\theta}}^r dv\, (r-v)^{-2\alpha-\frac{\beta}{2}} \int_{\R^k} d\tilde \xi\,
\Vert \tilde \xi \Vert^{\beta-k} e^{-4 \pi^2 \Vert \tilde \xi \Vert^2} \right)^{\frac{p}{2}} \\
&\leq c_p (r-s+\e^{\theta})^{(\frac{2-\beta}{4}-\alpha)p}.
\end{split}
\end{equation*}
Hence,  we conclude that
\begin{equation} \label{lemaY2}
\sup_{(r,z) \in [s-\e^{\theta},s] \times \R^k} \E[ \vert Y_{\alpha}(r,z) \vert^p] \leq c_p \e^{\theta (\frac{2-\beta}{4}-\alpha)p}.
\end{equation}

Let us now bound $W_{1,\e^{\theta}}$. Using (\ref{fub2}) and Minskowski's inequality, we have that
\begin{equation*}
\begin{split}
W_{1,\e^{\theta}}&\leq \left( \int_{s-\e^{\theta}}^s dr \int_{\R^k} dz (\psi_{\alpha}(t-r, x-z)-\psi_{\alpha}(s-r,x-z)) \right)^{2q} \\
&\qquad \qquad \qquad \times \sup_{(r,z) \in [s-\e^{\theta},s] \times \R^k} \E[ \vert Y_{\alpha}(r,z) \vert^{2q}],
\end{split}
\end{equation*}
where $\psi_{\alpha}(t,x)=S(t,x) t^{-\alpha}$.
Then by (\ref{lemaY2}) and Lemma \ref{lemSS}(b), we obtain that for any $\gamma <4\alpha$,
\begin{equation*}
W_{1,\e^{\theta}}\leq c_q  \e^{\theta q (2\alpha - \frac{\gamma}{2})}\, \vert t-s \vert^{\frac{\gamma}{2} q} \,
 \e^{\theta (\frac{2-\beta}{2}-2\alpha)q}
= c_q \e^{\theta (\frac{2-\beta}{2}-\frac{\gamma}{2}) q} \vert t-s \vert^{\frac{\gamma}{2} q}.
\end{equation*}
Thus, using the fact that $t-s \leq \e$, we conclude that
\begin{equation} \label{w1}
W_{1,\e^{\theta}} \leq c_q  \e^{\theta (\frac{2-\beta}{2}-\frac{\gamma}{2}) q} \e^{\frac{\gamma}{2}q} =c_q
\e^{\theta \frac{2-\beta}{2} q} \e^{\frac{\gamma}{2}(1-\theta) q}.
\end{equation}

We finally treat $W_{2,\e^{\theta}}$. Using (\ref{fub2}) and Minskowski's inequality, we have that
\begin{equation*}
\begin{split}
W_{2,\e^{\theta}}&\leq \left( \int_{s-\e^{\theta}}^s dr \int_{\R^k} dz (\psi_{\alpha}(s-r, x-z)-\psi_{\alpha}(s-r,y-z)) \right)^{2q} \\
&\qquad \qquad \qquad \times \sup_{(r,z) \in [s-\e^{\theta},s] \times \R^k} \E[ \vert Y_{\alpha}(r,z) \vert^{2q}].
\end{split}
\end{equation*}
Then by (\ref{lemaY2}) and Lemma \ref{lemSS}(a), we obtain that for any $\gamma <4\alpha$,
\begin{equation*}
W_{2,\e^{\theta}}\leq c_q \, \e^{\theta q (2 \alpha - \gamma)}\, \vert x-y \vert^{\gamma q}\, \e^{\theta (\frac{2-\beta}{2}-2\alpha)q} 
= c_q \e^{\theta (\frac{2-\beta}{2}-\gamma )q} \vert x-y \vert^{\gamma q}.
\end{equation*}
Thus, using the fact that $\vert x-y \vert \leq \sqrt{\delta_0 \e}$, we conclude that
\begin{equation} \label{w2}
W_{2,\e^{\theta}} \leq c_q  \e^{\theta (\frac{2-\beta}{2}-\gamma) q} \delta_0^{\frac{\gamma}{2}q} \e^{\frac{\gamma}{2} q} =c_q \delta_0^{\frac{\gamma}{2}q}
\e^{\theta \frac{2-\beta}{2} q} \e^{\frac{\gamma}{2} (1-2\theta) q}.
\end{equation}

Finally, substituting (\ref{w1}) and (\ref{w2}) into (\ref{wo}) we conclude that
for any $q \geq 1$,
\begin{equation*}
\E \, [\vert \bar G_{3,3,\e^{\theta}} \vert^q ]\leq c_q \e^{\theta \frac{2-\beta}{2} q} \e^{\frac{\gamma}{2}(1-2\theta) q}.
\end{equation*}
Therefore, we have proved that in the Case 2,
\begin{equation*}
1_{\{\alpha_{i_0} \geq \alpha_0 \}}\
\left(\xi^{i_0}\right)^{\sf T}
\gamma_Z \xi^{i_0}
\geq 1_{\{\alpha_{i_0} \geq \alpha_0 \}} (c \e^{\theta \frac{2-\beta}{2}} - W_{\e}),
\end{equation*}
where $\E [\vert W_{\epsilon} \vert^q]\leq c_q \epsilon^{\theta \frac{2-\beta}{2}q+\frac{\gamma}{2}q \min(\theta,1-2\theta)}$.

\vskip 12pt
\noindent{\bf Case 3.} $t-s \leq \e$, $0<\e<\frac{\vert x-y \vert^2}{\delta_0}$. From (\ref{sumaqua}) and (\ref{large1}), we have that
\begin{equation*}
1_{\{\alpha_{i_0} \geq \alpha_0 \}}\ \left(\xi^{i_0}\right)^{\sf T} \gamma_Z \xi^{i_0}
\geq \frac23 G_{4, \e}-8 (\bar G_{4,1,\e} - \bar G_{4,2,\e} -
\bar G_{4,3,\e} -\bar G_{4,4,\e}),
\end{equation*}
where
\begin{equation*}\begin{split}
G_{4,\e} &:=\int_{s-\epsilon}^s dr
\sum_{l=1}^d \bigg\Vert \sum_{i=1}^d \bigg\{ \left( \alpha \tilde\lambda_i - \tilde\mu_i \sqrt{1 - \alpha^2}
\right) \sigma_{i,l}(u(r,y)) \, S(s-r,y -\cdot) \\
& \qquad \qquad \qquad \qquad +\sqrt{1 - \alpha^2} \tilde\mu_i \sigma_{i,l}(u(r,x)) S(t-r,x -\cdot)\bigg\} \bigg\Vert_\cH^2,\\
\bar G_{4,1,\e} &:=\int_{s-\epsilon}^s dr
\sum_{l=1}^d \left\Vert \sum_{i=1}^d \left(\alpha \tilde\lambda_i - \tilde\mu_i \sqrt{1 - \alpha^2}
\right) [\sigma_{i, l}(u(r,\cdot))-\sigma_{i,l}(u(r,y))] \, S(s-r,y -\cdot)
\right\Vert_\cH^2,\\
\bar G_{4,2,\e}&:= (1 - \alpha^2) \int_{s-\epsilon}^s dr
\sum_{l=1}^d \left\Vert \left(\tilde{\mu}^{\sf T} \cdot [\sigma(u(r,\cdot))-\sigma(u(r,x))] \right)_l\, S(t-r,x -\cdot)
\right\Vert_\cH^2, \\
\bar G_{4,3,\e} &:= \int_{s-\e}^s dr \sum_{l=1}^d \bigg\Vert \sum_{i=1}^d
(\alpha \tilde\lambda_i - \tilde\mu_i \sqrt{1 - \alpha^2}) \ a_i(l,r,s,y) \bigg\Vert_{\mathcal{H}}^2, \\
\bar G_{4,4,\e} &:=(1 - \alpha^2) \int_{s-\e}^s dr \sum_{l=1}^d \bigg\Vert \sum_{i=1}^d
\tilde\mu_i \ a_i(l,r,t,x) \bigg\Vert_{\mathcal{H}}^2.
\end{split}\end{equation*}
We start with a lower bound for $G_{4,\e}$. Observe that this term is similiar to the term $A_1$ in the Sub-Case A of the proof
of Proposition \ref{smalle}. Using the inequality $(a+b)^2\geq a^2+b^2-2\vert ab\vert$, we see
that $G_{4,\e} \geq G_{4,1,\e}+G_{4,2,\e}-2G_{4,3,\e}$, where
\begin{equation*} \begin{split}
G_{4,1,\e}&=\sum_{l=1}^d \int_{s-\e}^s
dr \, \bigg\Vert
S(s-r,y-\cdot) ((\alpha \lambda-\sqrt{1-\alpha^2}\mu)^{\sf T}
\cdot \sigma(u(r,y)))_l\bigg\Vert_{\mathcal{H}}^2, \\
G_{4,2,\e}&=\sum_{l=1}^d \int_{s-\e}^s
dr \, \bigg\Vert
S(t-r,x-\cdot) (\sqrt{1-\alpha^2}\mu^{\sf T}
\cdot \sigma(u(r,x)))_l\bigg\Vert_{\mathcal{H}}^2, \\
G_{4,3,\e}&=\sum_{l=1}^d \int_{s-\e}^s
dr \, \big\langle S(s-r,y-\cdot) ((\alpha\lambda-\sqrt{1-\alpha^2} \mu)^{\sf T}
\cdot \sigma(u(r,y)))_l, \\
&\qquad \qquad \qquad \qquad S(t-r,x-\cdot) (\alpha \mu^{\sf T} \cdot
\sigma(u(r,x)))_l\big\rangle_{\mathcal{H}}.
\end{split}
\end{equation*}
Hypothesis {\bf P2}, Lemma \ref{scalar}, and the fact that $t-s \leq \e$ imply that
\begin{equation*}
G_{4,1,\e}+G_{4,2,\e} \geq c (\Vert \alpha \lambda-\sqrt{1-\alpha^2}\mu \Vert^2 +\Vert \sqrt{1-\alpha^2}\, \mu \Vert^2) \e^{\frac{2-\beta}{2}}
\geq c_0 \e^{\frac{2-\beta}{2}}.
\end{equation*}
On the other hand, using the same computation as the one done for the term $\tilde{B_4}$ in the Sub-Case A of the proof
of Proposition \ref{smalle}, we conclude that $G_{4,3,\e} \leq
\Phi(\frac12\,\sqrt{\delta_0}-1) \e^{\frac{2-\beta}{2}}$, with $\lim_{\alpha \rightarrow
+\infty} \Phi(\alpha)=0$. Choose $\delta_0$ sufficiently large so that
\begin{equation}\label{eqdelta0}
\Phi(\frac12\,\sqrt{\delta_0}-1) \leq \frac{c_0}{2},
\end{equation}
so that $G_{4,\epsilon} \geq \frac{c_0}{2} \e^{\frac{2-\beta}{2}}$.

We next treat the terms $\bar G_{4,i,\e}$, $i=1,...,4$. Using the same argument as for the term $\bar G_{3,1,\e^\theta}$, we see that for any $q \geq 1$,
\begin{equation*}
\E \, [\vert \bar G_{4,1,\e} \vert^q ]\leq C \e^{(\frac{2-\beta}{2}+\frac{\gamma}{2})q}.
\end{equation*}
Appealing to Lemma \ref{lem:7} and using the fact that $t-s \leq \e$, we find that
$$
\E \, [|\bar G_{4,3,\e}|^q] \leq c_q \e^{(2-\beta) q}, \text{ and } \E \, [| \bar G_{4,4,\e}|^q] \leq c_q \e^{\theta (2-\beta)q}.
$$
Finally, we treat $\bar G_{4,2,\e}$. As in the proof of Proposition \ref{det},
using H\"older's inequality, the Lipschitz property of $\sigma$, Lemma \ref{scalar} and (\ref{equa3}),
we get that for any $q \geq 1$
\begin{equation*}
\E \, [\vert \bar G_{4,2,\e} \vert^q ]\leq C \e^{(\frac{2-\beta}{2})(q-1)} \times \Psi,
\end{equation*}
where
\begin{equation*}
\Psi=\int_{s-\epsilon}^{s} dr \, \int_{\R^k} dv \,
\int_{\R^k} dz \, \Vert z-v \Vert^{-\beta} S(t-r,x-v) S(t-r,x-z) \Vert x-v \Vert^{\frac{\gamma}{2} q}
\Vert x-z \Vert^{\frac{\gamma}{2} q}.
\end{equation*}
Changing variables $[\tilde{v}=\frac{x-v}{\sqrt{t-r}}, \tilde{z}=\frac{x-z}{\sqrt{t-r}}]$, this becomes
\begin{equation*}
\begin{split}
\Psi&=\int_{s-\e}^{s} dr \, (t-r)^{-\frac{\beta}{2}+\frac{\gamma q}{2}}
\int_{\R^k} d\tilde{v} \,
\int_{\R^k} d \tilde{z} \, S(1,\tilde{v}) S(1,\tilde{z })\Vert \tilde{v}-\tilde{z} \Vert^{-\beta} \Vert \tilde{z}\Vert^{\gamma q/2} \Vert\tilde{v}\Vert^{\gamma q/2} \\
&= C ((t-s+\e)^{\frac{2-\beta}{2}+\frac{\gamma q}{2}}-(t-s)^{\frac{2-\beta}{2}+\frac{\gamma q}{2}}) \\
&\leq C \e^{\frac{2-\beta}{2}+\frac{\gamma q}{2}}.
\end{split}
\end{equation*}
Hence, we obtain that for any $q \geq 1$,
\begin{equation*}
\E \, [\vert \bar G_{4,2,\e} \vert^q ]\leq C \e^{(\frac{2-\beta}{2}+\frac{\gamma}{2})q}.
\end{equation*}

Therefore, we have proved that in the Case 3,
\begin{equation*}
1_{\{\alpha_{i_0} \geq \alpha_0 \}}\
\left(\xi^{i_0}\right)^{\sf T}
\gamma_Z \xi^{i_0}
\geq 1_{\{\alpha_{i_0} \geq \alpha_0 \}} (c \e^{\frac{2-\beta}{2}} - G_{\e}),
\end{equation*}
where $\E [\vert G_{\epsilon} \vert^q]\leq c_q \epsilon^{(\frac{2-\beta}{2}+\frac{\gamma}{2})q}$. This completes Case 3.
\vskip 12pt

Putting together the results of the Cases 1, 2 and 3, we see that for $0 < \e \leq \e_0$,
\begin{equation*}
1_{\{\alpha_{i_0} \geq \alpha_0 \}}\
\left(\xi^{i_0}\right)^{\sf T} \gamma_Z \xi^{i_0}
\geq 1_{\{\alpha_{i_0} \geq \alpha_0 \}}\ Z,
\end{equation*}
where
\begin{equation*}
\begin{split}
&Z = \min\left(c \e^{\frac{2-\beta}{2}+\eta} - V_{\e} ,~c\e^\frac{2-\beta}{2} - J_\e \right)\1_{\{t-s>\e\}} +(c \e^{\theta \frac{2-\beta}{2}} - W_{\e})\1_{\{t-s \leq \e,\, \e\geq \frac{\vert x-y \vert^2}{\delta_0}\}} \\
& \qquad \qquad + 
(c\e^\frac{2-\beta}{2} - G_\e)\1_{\{t-s \leq \e <\frac{\vert x-y \vert^2}{\delta_0}\}},
\end{split}
\end{equation*}
where for any $q \geq 1$,
\begin{equation*}
\begin{split}
&\E \, [\vert V_{\e} \vert^q ]\leq C \epsilon^{(\frac{2-\beta}{2}+\frac{\gamma}{2})q}, \qquad 
\E \, [\vert J_{\e} \vert^q ]\leq C \epsilon^{(\frac{2-\beta}{2}+\eta)q},\\
&\E \, [\vert W_{\e} \vert^q ]\leq C \epsilon^{\theta \frac{2-\beta}{2}q+\frac{\gamma}{2}q \min(\theta,1-2\theta)}, \qquad
\E \, [\vert G_{\e} \vert^q ]\leq C \epsilon^{(\frac{2-\beta}{2}+\frac{\gamma}{2})q}.
\end{split}
\end{equation*}
Therefore,
\begin{equation*}
\begin{split}
Z & \geq  \min \Big( c \e^{\frac{2-\beta}{2}+\eta} - V_{\e}, \, 
c\e^{\frac{2-\beta}{2}} - J_\e \1_{\{t-s>\e\}} - G_\e \1_{\{t-s \leq \e <\frac{\vert x-y \vert^2}{\delta_0}\}}, \\
&\qquad \qquad \qquad \qquad 
\, c \e^{\theta\frac{2-\beta}{2}} 
- W_{\e} \1_{\{t-s \leq \e,\, \e\geq \frac{\vert x-y \vert^2}{\delta_0}\}}
\Big). \\
\end{split}
\end{equation*}
Note that all the constants are independent of $i_0$.
Then using \cite[Proposition 3.5]{Dalang2:05} (extended to the minimum of three terms instead of two), we deduce that for all $p \geq 1$, there is $C>0$ such that
$$
E\left[\left(1_{\{\alpha_{i_0} \geq \alpha_0 \}}\ \left(\xi^{i_0}\right)^{\sf T}
\gamma_Z \xi^{i_0} \right)^{-p}\right]
\leq E\left[1_{\{\alpha_{i_0} \geq \alpha_0 \}}\ Z^{-p}\right]
\leq E\left[Z^{-p}\right] \leq C.
$$
Since this applies to any $p \geq 1$, we can use H\"older's inequality to deduce \eqref{A}. This proves Proposition \ref{large}.
\end{proof}

   The following result is analogous to \cite[Theorem 6.3]{Dalang2:05}.
   
\begin{thm} \label{ga}
Fix $\eta, T>0$. Assume {\bf P1} and {\bf P2}. Let $I \times J \subset (0,T] \times \R^k$ be a closed non-trivial rectangle. For any $(s,y), (t,x) \in I \times J$, $s \leq t$, $(s,y) \neq (t,x)$, $k \geq 0$, and $p>1$, 
\begin{equation*}
\Vert (\gamma_Z^{-1})_{m,l} \Vert_{k,p}\leq
\begin{cases}
c_{k,p,\eta,T}(|t-s|^{\frac{2-\beta}{2}}+ \Vert x-y \Vert^{2-\beta})^{- \eta}&
\text{if \; $(m,l) \in {\bf (1)}$}, \\
c_{k,p,\eta,T} (|t-s|^{\frac{2-\beta}{2}}+\Vert x-y \Vert^{2-\beta})^{-1/2- \eta} &
\text{if \; $(m,l) \in {\bf (2)}$ or ${\bf (3)}$}, \\
c_{k,p,\eta,T} (|t-s|^{\frac{2-\beta}{2}}+\Vert x-y \Vert^{2-\beta})^{-1- \eta}& \text{if \; $(m,l) \in {\bf (4)}$}.
\end{cases}
\end{equation*}
\end{thm}

\noindent{\em Proof.} As in the proof of  \cite[Theorem 6.3]{Dalang2:05}, we shall use Propositions \ref{3p2}--\ref{3p3}. Set $\Delta = \vert t-s \vert ^{1/2} + \Vert x-y\Vert$.

   Suppose first that $k=0$. Since the inverse of a matrix is the inverse of its determinant multiplied by its cofactor matrix, we use Proposition \ref{3p3} with $\eta$ replaced by $\tilde \eta = \frac{\eta}{2d(2-\beta)}$ and Proposition \ref{3p2} with $\gamma \in (0,2-\beta)$ such that $2 - \beta - \gamma = \frac{\eta}{2(d-\frac12)}$ to see that for $(m,l) \in{\bf (2)}$ or ${\bf (3)}$, 
\begin{align*}
   \Vert (\gamma_Z^{-1})_{m,l} \Vert_{0,p} &\leq c_{p,\eta,T}\, \Delta^{-d(2-\beta)(1+\tilde \eta)}\, \Delta^{\gamma(d-\frac12)}\\
   &= c_{p,\eta,T}\, \Delta^{-\frac{2-\beta}{2}}\,  \Delta^{(2-\beta - \gamma)\frac12}\, \Delta^{-d(2-\beta - \gamma) - \tilde \eta d (2-\beta) }\\
   &= c_{p,\eta,T}\, \Delta^{-\frac{2-\beta}{2}}\,  \Delta^{-(d-\frac12)(2-\beta - \gamma) - \tilde \eta d (2-\beta)} \\
   &= c_{p,\eta,T}\, \Delta^{-\frac{2-\beta}{2}}\,  \Delta^{-\eta}.
\end{align*}
This proves the statement for  $(m,l) \in{\bf (2)}$ or ${\bf (3)}$. The other two cases are handled in a similar way.

   For $k \geq 1$, we proceed recursively as in the proof of \cite[Theorem 6.3]{Dalang2:05}, using Proposition \ref{dgamma} instead of \ref{3p2}.
\hfill $\Box$
\vskip 16pt

\begin{remark} In \cite[Theorem 6.3]{Dalang2:05}, in the case where $d=1$ and $s=t$, a slightly stronger result, without the exponent $\eta$, is obtained. Here, when $s=t$, the right-hand sides of \eqref{eq5.629} and \eqref{eq5.1029} can be improved respectively to $C \Vert x-y \Vert^{-(2-\beta)d}$ and $C \Vert x-y \Vert^{-(2-\beta)2dp}$. Indeed, when $s=t$, Case 1 in the proof of Proposition \ref{smalle} does not arise, and this yields the improvement of \eqref{eq5.1029}, and, in turn, the improvement of \eqref{eq5.629}. However, this does not lead to an improvement of the result of Theorem \ref{ga} when $s=t$, because the exponent $\eta$ there is also due to the fact that $\gamma < 2-\beta$ in Proposition \ref{3p2}.
\end{remark}

In the next subsection, we will establish the estimate of Theorem \ref{t2}(b). For this,
we will use the following expression for the density of a
nondegenerate random vector that is a consequence of the integration by parts formula of Malliavin calculus.

\begin{cor}\label{no} \textnormal{~\cite[Corollary 3.2.1]{Nualart:98}}
Let $F=(F^1,...,F^d) \in (\mathbb{D}^{\infty})^d$ be a nondegenerate
random vector and let $p_F(z)$ denote the density of $F$ (see Theorem \ref{3t1}). Then for every subset
$\sigma$ of the set of indices $\{1,...,d\}$,
\begin{equation*}
p_F(z)=(-1)^{d-|\sigma|} \E [1_{\{F^i>z^i, i \in \sigma, \,
F^i < z^i, i \not\in \sigma \}} H_{(1,...,d)}(F,1)],
\end{equation*}
where $|\sigma|$ is the cardinality of $\sigma$, and 
\begin{equation*}
H_{(1,...,d)}(F,1)=\delta((\gamma_{F}^{-1} DF)^d
\delta((\gamma_{F}^{-1} DF)^{d-1} \delta( \cdots
\delta((\gamma_{F}^{-1} DF)^{1})\cdots))).
\end{equation*}
\end{cor}

The following result is similar to \cite[(6.3)]{Dalang2:05}. 
\begin{prop} \label{H}
Fix $\eta, T>0$. Assume {\bf P1} and {\bf P2}. Let $I \times J \subset (0,T] \times \R^k$ be a closed non-trivial rectangle. For any $(s,y), (t,x) \in I \times J$, $s \leq t$, $(s,y) \neq (t,x)$, and $k \geq 0$,
\begin{equation*}
\Vert H_{(1,...,2d)}(Z,1) \Vert_{0,2} \leq C_T
(|t-s|^{\frac{2-\beta}{2}}+\Vert x-y \Vert^{2- \beta})^{-(d+
\eta)/2},
\end{equation*}
where $Z$ is the random vector defined in \textnormal{(\ref{Z})}.
\end{prop}

\noindent{\sc Proof.} The proof is similar to that of \cite[(6.3)]{Dalang2:05} using the continuity of the Skorohod integral $\delta$ 
(see \cite[Proposition 3.2.1]{Nualart:95} and \cite[(1.11) and p.131]{Nualart:98}) and H\"older's inequality for Malliavin norms 
(see \cite[Proposition 1.10, p.50]{Watanabe:84}); the only change is that $\gamma$ in Proposition \ref{derivat} must be chosen sufficiently close to $2-\beta$. 
\hfill $\Box$
\vskip 16pt

\subsection{Proof of Theorem \ref{t2}(b)}\label{sec53}

Fix $T>0$ and let $I \times J \subset (0,T] \times \R^k$ be a closed non-trivial rectangle. Let $(s,y),(t,x) \in I \times J$, $s \leq t$, $(s,y) \neq (t,x)$, and $z_1,z_2 \in \R^d$. Let $p_Z$ be the density of the random vector $Z$ defined in (\ref{Z}). Then
\begin{equation*}
p_{s,y; \, t,x}(z_1,z_2)=p_Z(z_1, z_2-z_1).
\end{equation*}
Apply Corollary \ref{no} with
$\sigma=\{ i \in \{1,...,d\} : z_2^i-z_1^i \geq 0 \}$ and H\"older's
inequality to see that
\begin{equation}\label{rd**1}
p_Z(z_1, z_1-z_2) \leq \prod_{i=1}^d \biggl( \P
\biggl\{|u_i(t,x)-u_i(s,y) | > |z_1^i -z_2^i | \biggr\}
\biggr)^{\frac{1}{2d}} \times \Vert H_{(1,...,2d)}(Z,1) \Vert_{0,2}.
\end{equation}
  
   When $\Vert z_1 - z_2 \Vert = 0$,
$$
   \frac{|t-s|^{\gamma/2}+\Vert x-y \Vert^{\gamma}}{\Vert z_1 - z_2 \Vert} \wedge 1 = 1,
$$
since the numerator is positive because $(s,y) \neq (t,x)$. Therefore, \eqref{eq1.9} follows from Proposition \ref{H} in this case. 

   Assume now that $\Vert z_1 - z_2 \Vert \neq 0$. Then there is $i \in \{1,\dots,d\}$, and we may as well assume that $i=1$, such that $0 < \vert z^1_1 - z^1_2\vert = \max_{i=1,\dots,d} \vert z^i_1 - z^i_2\vert$. Then
$$
   \prod_{i=1}^d \left( \P
\left\{|u_i(t,x)-u_i(s,y)| > |z_1^i-z_2^i|
\right\}\right)^{\frac{1}{2d}} \leq \left(\P
\left\{|u_1(t,x)-u_1(s,y)| > |z_1^1-z_2^1|
\right\} \right)^{\frac{1}{2d}}.
$$
Using Chebyshev's inequality and (\ref{equa3}), we see that this is bounded above by
\begin{equation}\label{rd**2}
  c\left[\frac{|t-s|^{\gamma/2}+\Vert x-y \Vert^{\gamma}}{\vert z_1^1-z_2^1
\vert^2}\, \wedge 1 \right]^{\frac{p}{2d}}  \leq \tilde c\left[\frac{|t-s|^{\gamma/2}+\Vert x-y \Vert^{\gamma}}{\Vert z_1-z_2
\Vert^2}\, \wedge 1 \right]^{\frac{p}{2d}}.
\end{equation}
The two inequalities \eqref{rd**1} and \eqref{rd**2}, together with Proposition \ref{H}, prove Theorem \ref{t2}(b).
\hfill $\Box$
\vskip 12pt

As mentioned in Remark \ref{remexp}, in the case where $b \equiv 0$, one can establish the following exponential upper bound.

\begin{lem} \label{exp} 
Let $\tilde{u}$ the solution of \textnormal{(\ref{equa1})} with $b \equiv 0$.
Fix $T>0$ and $\gamma \in (0, 2-\beta)$. Assume {\bf P1}. Let $I \times J \subset (0,T] \times \R^k$ be a closed non-trivial rectangle. 
Then there exist constants $c, c_T>0$ such that for any $(s,y), (t,x) \in I \times J$, $s \leq t$, $(s,y) \neq (t,x)$, $z_1,z_2 \in \R^d$,
\begin{equation*}
\prod_{i=1}^d \biggl( \P
\biggl\{|\tilde{u}_i(t,x)-\tilde{u}_i(s,y)| > |z_1^i-z_2^i|
\biggr\}\biggr)^{\frac{1}{2d}} \leq c\exp \biggl( -\frac{\Vert
z_1-z_2 \Vert^2}{c_T (|t-s|^{\gamma/2}+\Vert x-y \Vert^{\gamma})}\biggr).
\end{equation*}
\end{lem}

\begin{proof} 
Consider the continuous one-parameter martingale $(M_a=(M_a^1,...,M_a^d), \, 0 \leq a \leq t)$ defined by
\begin{equation*}
M^i_a=
\begin{cases}
&\int_0^a \int_{\R^k} (S(t-r,x-v)-S(s-r,y-v)) \sum_{j=1}^d \sigma_{ij}(\tilde{u}(r,v)) \, M^j(dr, dv) \\
& \qquad \qquad \qquad \qquad \qquad \qquad \qquad \qquad\qquad \qquad\qquad \qquad\text{if $0 \leq a \leq s $}, \\
& \\
&\int_0^s \int_{\R^k}
(S(t-r,x-v)-S(s-r,y-v)) \sum_{j=1}^d \sigma_{ij}(\tilde{u}(r,v)) \, M^j(dr, dv) \\
&\qquad \qquad + \int_s^a \int_{\R^k} S(t-r,x-v) \sum_{j=1}^d
\sigma_{ij}(\tilde{u}(r,v)) \, M^j(dr, dv) \\
& \qquad \qquad \qquad \qquad\qquad \qquad \qquad \qquad\qquad
\qquad\qquad \qquad\text{if $s \leq a \leq t$},
\end{cases}
\end{equation*}
for all $i=1,...,d$, with respect to the filtration
($\mathcal{F}_a$, $0 \leq a \leq t$). Notice that
\begin{equation*}
M^i_0=0, \; \; M^i_t=\tilde{u}^i(t,x)-\tilde{u}^i(s,y).
\end{equation*}
Moreover, because the $M^i$ are independent and white in time, $\langle M^i \rangle_t =\mathcal{M}^i_1+\mathcal{M}^i_2$, where
\begin{equation*} \begin{split}
\mathcal{M}^i_1&=\sum_{j=1}^d \int_0^s dr \, \bigg\Vert (S(t-r,x-\cdot)-S(s-r,y-\cdot)) \sigma_{ij}(\tilde{u}(r,\cdot)) \bigg\Vert^2_{\mathcal{H}}, \\
\mathcal{M}^i_2&= \sum_{j=1}^d \int_s^t dr \, \bigg\Vert S(t-r,x-\cdot) \sigma_{ij}(\tilde{u}(r,\cdot)) \bigg\Vert^2_{\mathcal{H}}.
\end{split}
\end{equation*}
Using the fact that the coefficients of $\sigma$ are bounded and Lemma \ref{scalar}, we get that
\begin{equation*}
\mathcal{M}^i_2 \leq c \vert t-s \vert^{\frac{2-\beta}{2}}.
\end{equation*}
On the other hand, we write $\mathcal{M}^i_1 \leq 2(\mathcal{M}^i_{1,1}+\mathcal{M}^i_{1,2})$, where
\begin{equation*} \begin{split}
\mathcal{M}^i_{1,1}&=\sum_{j=1}^d \int_0^s dr \, \bigg\Vert (S(t-r,x-\cdot)-S(t-r,y-\cdot)) \sigma_{ij}(\tilde{u}(r,\cdot)) \bigg\Vert^2_{\mathcal{H}}, \\
\mathcal{M}^i_{1,2}&= \sum_{j=1}^d \int_0^s dr \, \bigg\Vert (S(t-r,y-\cdot)-S(s-r,y-\cdot)) \sigma_{ij}(\tilde{u}(r,\cdot)) \bigg\Vert^2_{\mathcal{H}}.
\end{split}
\end{equation*}
In order to bound these two terms, we will use the factorisation method. Using the semigroup property of $S$ and the Beta function, it yields that, for any $\alpha \in (0,1)$,
\begin{equation*} \begin{split}
S(t-r,x-z)= \frac{\sin(\pi \alpha)}{\pi} \int_{r}^t d\theta \int_{\R^k} d\eta \,
\psi_{\alpha}(t-\theta,x-\eta) S(\theta-r, \eta -z) (\theta-r)^{-\alpha},
\end{split}
\end{equation*}
where $\psi_{\alpha}(t,x)=S(t,x)t^{\alpha-1}$.
Hence, using the boundedness of the coefficients of $\sigma$, we can write
\begin{equation*} \begin{split}
\mathcal{M}^i_{1,1}&\leq c \sum_{j=1}^d \int_0^s dr \, \bigg\Vert \int_{r}^t d\theta \int_{\R^k} d\eta \, \vert \psi_{\alpha}(t-\theta,x-\eta)-\psi_{\alpha}(t-\theta,y-\eta) \vert \\
&\qquad \qquad \qquad \times S(\theta-r, \eta -\cdot) 
(\theta-r)^{-\alpha} \bigg\Vert^2_{\mathcal{H}},
\end{split}
\end{equation*}
and ${M}^i_{1,2} \leq c ({M}^i_{1,2,1}+{M}^i_{1,2,2})$, where
\begin{equation*} \begin{split}
\mathcal{M}^i_{1,2,1}& = \sum_{j=1}^d \int_0^s dr \bigg\Vert \int_{r}^s d\theta \int_{\R^k} d\eta \, \vert \psi_{\alpha}(t-\theta,y-\eta)-\psi_{\alpha}(s-\theta,y-\eta) \vert \\
& \qquad \qquad \qquad \times  S(\theta-r, \eta-\cdot) (\theta-r)^{-\alpha} \bigg\Vert^2_{\mathcal{H}}, \\
\mathcal{M}^i_{1,2,2}&= \sum_{j=1}^d \int_0^s dr \bigg\Vert \int_{s}^t d\theta \int_{\R^k} d\eta \, \psi_{\alpha}(t-\theta,y-\eta) S(\theta-r, \eta -\cdot) (\theta-r)^{-\alpha} \bigg\Vert^2_{\mathcal{H}},
\end{split}
\end{equation*}
Using H\"older's inequality, \eqref{nu}, (\ref{eqau}) and Lemma \ref{lemSS}, we get that for any $\alpha \in (0, \frac{2-\beta}{4})$ and $\gamma \in (0, 4 \alpha)$,
\begin{equation*} \begin{split}
\mathcal{M}^i_{1,1}&\leq c \sup_{(r,z) \in [0,T] \times \R^k} \Vert S(r-\ast, z -\cdot) (r-\ast)^{-\alpha} \Vert_{\mathcal{H}^d_r}^2 \\
& \qquad \qquad \times \biggl( \int_{0}^t dr \int_{\R^k} dz \, \vert \psi_{\alpha}(t-r,x-z)-\psi_{\alpha}(t-r,y-z) \vert \biggr)^2 \\
&\leq c_T(\alpha) \Vert x-y \Vert^{\gamma}, \\
\mathcal{M}^i_{1,2,1}&\leq c \sup_{(r,z) \in [0,T] \times \R^k} \Vert S(r-\ast, z -\cdot) (r-\ast)^{-\alpha} \Vert_{\mathcal{H}^d_r}^2 \\
& \qquad \qquad \times \biggl( \int_{0}^t dr \int_{\R^k} dz \, \vert \psi_{\alpha}(t-r,y-z)-\psi_{\alpha}(s-r,y-z) \vert \biggr)^2 \\
&\leq c_T(\alpha) \Vert t-s \Vert^{\gamma/2}, \\
\mathcal{M}^i_{1,2,2}&\leq c \sup_{(r,z) \in [0,T] \times \R^k} \Vert S(r-\ast, z -\cdot) (r-\ast)^{-\alpha} \Vert_{\mathcal{H}^d_r}^2
\biggl( \int_{s}^t dr \int_{\R^k} dz \, \psi_{\alpha}(t-r,y-z) \biggr)^2 \\
&\leq c_T(\alpha) \Vert t-s \Vert^{\gamma/2}.
\end{split}
\end{equation*}
Thus, we have proved that for any $\gamma \in (0, 2-\beta)$,
\begin{equation*}
\langle M^i \rangle_t \leq c_T (|t-s|^{\gamma/2}+\Vert x-y \Vert^{\gamma}).
\end{equation*}
By the exponential martingale inequality \cite[A.5]{Nualart:95},
\begin{equation*}
\P \biggl\{|\tilde{u}^i(t,x)-\tilde{u}^i(s,y) | > |z_1^i
-z_2^i | \biggr\} \leq 2 \exp \biggl( -\frac{| z^i_1-z^i_2 |^2}{c_T
(|t-s|^{\gamma/2}+\Vert x-y \Vert^{\gamma})}\biggr),
\end{equation*}
which implies the desired result.
\end{proof} 

\vskip 16pt

\appendix

\section{Appendix}

\begin{lem} \label{scalar}
There is $C >0$ such that for any $0 <\epsilon \leq s \leq t$ and $x \in \R^k$,
\begin{equation*}
\int_{s-\epsilon}^s dr \int_{\R^k} d \xi \, \Vert \xi \Vert^{\beta-k}\, \vert \mathcal{F} S(t-r, x-\cdot) (\xi) \vert^2
= C ((t-s+\epsilon)^{\frac{2-\beta}{2}}-(t-s)^{\frac{2-\beta}{2}}).
\end{equation*}
Moreover, there exists $\tilde{C}>0$ such that the above integral is bounded above by $\tilde{C} \epsilon^{\frac{2-\beta}{2}}$,
and if $t-s \leq \epsilon$, then there exists $\bar{C}>0$ such that the above integral is bounded below by $\bar{C} \epsilon^{\frac{2-\beta}{2}}$.
\end{lem}

\begin{proof}
Using (\ref{ft}) and changing variables $[\tilde{r}=t-r,\
\tilde{\xi}=\xi {\sqrt{r}}]$ yields
\begin{equation*} \begin{split}
&\int^{s}_{s-\e} dr \int_{\R^k} d \xi \, \Vert \xi \Vert^{\beta-k}\, \vert \mathcal{F} S(t-r, x-\cdot) (\xi) \vert^2 \\
& \qquad \qquad = \int^{t-s+\e}_{t-s} dr \, r^{-\beta/2}
\int_{\R^k} d\xi \, \Vert \xi \Vert ^{\beta - k}\, e^{-\Vert \xi \Vert^2} \\
& \qquad \qquad = C \int^{t-s+\e}_{t-s} dr \, r^{-\beta/2} \\
& \qquad \qquad = C ((t-s+\epsilon)^{\frac{2-\beta}{2}}-(t-s)^{\frac{2-\beta}{2}}).
\end{split}
\end{equation*}
If $\e<t-s$, then the last integral is bounded above by $C (t-s)^{-\beta/2}\, \e \leq C \e^{\frac{2-\beta}{2}}$. On the other hand, if $t-s\leq \e$, then the last integral is bounded above by
$$
\int^{2\e}_{0} dr \, r^{-\beta/2} \leq C \e^{\frac{2-\beta}{2}}.
$$
Finally, if $t-s\leq \e$, then
$$
   \int^{t-s+\e}_{t-s} dr \, r^{-\beta/2} \geq \e (t-s+\e)^{-\beta/2} \geq \e (2\e)^{-\beta/2} = c \e^{\frac{2-\beta}{2}}.
$$
\end{proof}

\begin{lem} \label{lem:7}
Assume ${\bf P1}$.
For all $T>0$ and $q\ge1$, there exists a constant
$c=c(q,T)\in(0,\infty)$ such that for every
$0<\e\le s\le t\le T$, $x\in \R^k$, and $a>0$,
\begin{equation*}
W:=\E\left[ \sup_{\xi\in\R^d: \Vert\xi\Vert \leq a }\biggl(\int_{s-\e}^s
dr \, \sum_{l=1}^d \bigg\Vert \sum_{i=1}^d
a_i(l,r,t,x) \xi_i \bigg\Vert_{\mathcal{H}}^{2} \biggr)^q\right]
\le c\, a^{2q}\, (t-s+\e)^{\frac{2-\beta}{2}q} \epsilon^{\frac{2-\beta}{2} q },
\end{equation*}
where $a_i(l,r,t,x)$ is defined in \textnormal{(\ref{eq4.1})}.
\end{lem}

\begin{proof}
Use (\ref{eq4.1}) and the Cauchy-Schwarz inequality to get
\begin{equation}\label{eqrdA.0}
W \leq c \, a^{2q} \biggl( \E \, \biggl[ \biggl( \int_{s-\epsilon}^s dr \,
\Vert W_1 \Vert_{\mathcal{H}^d}^2 \biggr)^q \biggr]+\E \, \biggl[ \biggl( \int_{s-\epsilon}^s dr \,
\Vert W_2 \Vert_{\mathcal{H}^d}^2 \biggr)^q \biggr] \biggr),
\end{equation}
where
\begin{align*}
W_{1}&=\sum_{i,j=1}^d \int_r^t \int_{\R^k} S(t-\theta,x-\eta) D_{r} (\sigma_{i,j}(u(\theta,\eta))) \, M^j(d\theta, d \eta), \\
W_{2}&=\sum_{i=1}^d \int_r^t d\theta \int_{\R^k} d\eta \, S(t-\theta,x-\eta) D_{r} (b_i(u(\theta,\eta))).
\end{align*}
Then
\begin{equation*}
\E \, \biggl[ \biggl( \int_{s-\epsilon}^s dr \,
\Vert W_1 \Vert_{\mathcal{H}^d}^2 \biggr)^q \biggr]=\E \left[ \Vert W_1 \Vert^{2q}_{L^2([s-\epsilon, s],\mathcal{H}^d)}\right].
\end{equation*}
We then apply \cite[(6.8) in Theorem 6.1]{Sanz:05} (see also \cite[(3.13)]{Nualart:07}) to see that this is
\begin{equation} \label{sanz1}
\begin{split}
&\leq \left( \int_{s-\epsilon}^t dr \int_{\R^k} \mu(d \xi)\, \vert \mathcal{F} S(r) (\xi) \vert^2 \right)^{q-1} \\
& \qquad \times \int_{s-\epsilon}^t d\rho \int_{\R^k} \mu(d \xi)\, \vert \mathcal{F} S(t-\rho) (\xi) \vert^2
\sup_{\eta \in \R^k} \E \left[ \Vert D_{\cdot, \ast} u(\rho, \eta) \Vert^{2q}_{L^2([s-\epsilon, s],\mathcal{H}^d)}\right].
\end{split}
\end{equation}
According to \cite[Lemma 8.2]{Sanz:05},
\begin{equation*}
\sup_{\eta \in \R^k} \E \left[ \Vert D_{\cdot, \ast} u(\rho, \eta) \Vert^{2q}_{L^2([s-\epsilon, s],\mathcal{H}^d)}\right]
\leq C \left( \int_{s-\epsilon}^{s \wedge \rho} dr \int_{\R^k} \mu(d \xi) \vert \mathcal{F} S(\rho-r) (\xi) \vert^2 \right)^q,
\end{equation*}
and we have
\begin{equation}\label{rdnA.3}
\int_{\R^k} \mu(d \xi)\, \vert \mathcal{F} S(r) (\xi) \vert^2 =\int_{\R^k} \frac{d \xi}{\Vert \xi \Vert^{k-\beta}} e^{-r \Vert \xi \Vert^2}=r^{-\frac{\beta}{2}} \int_{\R^k} \frac{d v}{\Vert v \Vert^{k-\beta}} e^{- \Vert v \Vert^2}=c_0 r^{-\frac{\beta}{2}}.
\end{equation}
For $\rho \leq s$,
\begin{equation*}
\int_{s-\epsilon}^{s \wedge \rho} dr \int_{\R^k} \mu(d \xi)\, \vert \mathcal{F} S(\rho-r) (\xi) \vert^2 = c_0 \int_0^{\rho-s+\epsilon} dr \, r^{-\frac{\beta}{2}}= c (\rho-s+\epsilon)^{\frac{2-\beta}{2}} \leq c \epsilon^{\frac{2-\beta}{2}},
\end{equation*}
and for $s \leq \rho$,
\begin{equation*}
\begin{split}
\int_{s-\epsilon}^{s \wedge \rho} dr \int_{\R^k} \mu(d \xi)\, \vert \mathcal{F} S(\rho-r) (\xi) \vert^2 &= c_0 \int_{\rho-s}^{\rho-s+\epsilon} dr \, r^{-\frac{\beta}{2}}= c \left((\rho-s+\epsilon)^{\frac{2-\beta}{2}}- (\rho-s)^{\frac{2-\beta}{2}}\right) \\
& = c \epsilon \int_0^1 (\rho-s+\epsilon \nu)^{-\frac{\beta}{2}} d \nu \leq c \epsilon \int_0^1 (\epsilon \nu)^{-\frac{\beta}{2}}d \nu \\
& =c\epsilon^{\frac{2-\beta}{2}}.
\end{split}
\end{equation*}
Therefore, from (\ref{sanz1}) and \eqref{rdnA.3} above,
\begin{equation} \label{rdA.3}
\E \left[ \Vert W_1 \Vert^{2q}_{L^2([s-\epsilon, s],\mathcal{H}^d)}\right] \leq c (t-s+\epsilon)^{\frac{2-\beta}{2}q} \epsilon^{\frac{2-\beta}{2}q}.
\end{equation}

We now examine the second term in (\ref{eqrdA.0}). Notice that
\begin{align*}
\int_{s-\e}^s dr \,
\Vert W_2 \Vert_{\mathcal{H}^d}^2 & \leq C \sum_{i=1}^d \int_{s-\e}^s dr \, \Big\langle \int_{s-\e}^t d\theta \int_{\R^k} d\eta\, 1_{\{\theta > r\}} S(t-\theta,x-\eta) D_{r} (b_i(u(\theta,\eta))), \\
& \qquad\qquad \int_{s-\e}^t d\tilde \theta \int_{\R^k} d\tilde\eta\, 1_{\{\tilde \theta > r\}} S(t-\tilde\theta,x-\tilde\eta) D_{r} (b_i(u(\tilde\theta,\tilde\eta)))\Big\rangle_{\mathcal{H}^d} \\
& = C \sum_{i=1}^d \int_{s-\e}^t d\theta \int_{\R^k} d\eta\, \int_{s-\e}^t d\tilde \theta \int_{\R^k} d\tilde\eta\, S(t-\theta,x-\eta) S(t-\tilde\theta,x-\tilde\eta)\\
&\qquad\qquad \times \int_{s-\e}^s dr \, \langle D_{r} (b_i(u(\theta,\eta))), D_{r} (b_i(u(\tilde\theta,\tilde\eta))) \rangle_{\mathcal{H}^d}.
\end{align*}
The $dr$-integral is equal to
$$
\langle D (b_i(u(\theta,\eta))), D (b_i(u(\tilde\theta,\tilde\eta))) \rangle_{\mathcal{H}_{s-\e,s}^d}.
$$
Therefore, we can apply H\"older's inequality to see that
\begin{align*}
&\E \, \biggl[ \biggl( \int_{s-\e}^s dr \,
\Vert W_2 \Vert_{\mathcal{H}^d}^2 \biggr)^q \biggr] \\
&\qquad \leq C \sum_{i=1}^d \left(\int_{s-\e}^t d\theta \int_{\R^k} d\eta\, \int_{s-\e}^t d\tilde \theta \int_{\R^k} d\tilde\eta\, S(t-\theta,x-\eta) S(t-\tilde\theta,x-\tilde\eta) \right)^{q-1}\\
& \qquad \qquad\qquad\times \int_{s-\e}^t d\theta \int_{\R^k} d\eta\, \int_{s-\e}^t d\tilde \theta \int_{\R^k} d\tilde\eta\, S(t-\theta,x-\eta) S(t-\tilde\theta,x-\tilde\eta)\\
&\qquad\qquad\qquad\qquad\times E\left[\langle D (b_i(u(\theta,\eta))), D (b_i(u(\tilde\theta,\tilde\eta))) \rangle^q_{\mathcal{H}_{s-\e,s}^d}\right].
\end{align*}
Using the Cauchy-Schwarz inequality, we see that the expectation above is bounded by
\begin{equation} \label{rdA.2}
\E\left[\Big\Vert D (b_i(u(\theta,\eta))) \Big\Vert_{\cH_{s-\e,s}^d}^{q} \Big\Vert D (b_i(u(\theta,\tilde{\eta})))
\Big\Vert_{\cH_{s-\e,s}^d}^{q} \right] \leq \sup_{\theta,\, \eta} \E\left[\Big\Vert D (b_i(u(\theta,\eta))) \Big\Vert_{\cH_{s-\e,s}^d}^{2q}\right].
\end{equation}
Arguing as for the term $W_1$ and using {\bf P1}, we bound the expectation by $c \e^{\frac{2-\beta}{2} q }$, and the remaining integrals are bounded by $(t-s+\e)^{q}$, so that
$$
\E \, \biggl[ \biggl( \int_{s-\e}^s dr \,
\Vert W_2 \Vert_{\mathcal{H}^d}^2 \biggr)^q \biggr] \leq C (t-s+\e)^{q} \e^{\frac{2-\beta}{2} q }.
$$
Together with \eqref{eqrdA.0} and \eqref{rdA.3}, this completes the proof.
\end{proof}
\vskip 16pt

\noindent {\bf Acknowledgement}.
The authors would like to thank Marta Sanz-Sol\'e for several useful discussions.  The first author also thanks the Isaac Newton Institute for Mathematical Sciences in Cambridge, England, for hospitality during the Spring 2010 program Stochastic Partial Differential Equations, where some of the research reported here was carried out.


\begin{thebibliography}{AAA99}

\bibitem[BP98]{Bally:98} {\sc{Bally, V. and Pardoux, E.}} (1998),
Malliavin calculus for white noise driven parabolic {S}{P}{D}{E}s,
{\it Potential Analysis}, {{\bf 9}}, 27-64.

\bibitem[BLX09]{Bierme:09} {\sc{Bierm\'e, H., Lacaux, C., and Xiao, Y.}} (2009), Hitting probabilities and the Hausdorff dimension of the inverse images of anisotropic Gaussian random fields, {\it Bull. London Math. Soc.}, {\bf 41}, 253-273.

\bibitem[D99]{Dalang:99} {\sc{Dalang, R.C.}} (1999), Extending martingale measure stochastic integral with applications to spatially homogeneous s.p.d.e's, {\it Electronic Journal of Probability}, {\bf{4}}, 1-29.

\bibitem[DF98]{DF98} {\sc{Dalang, R.C., Frangos, N.}} (1998), The stochastic wave equation in two spatial dimensions, {\em Annals of Probability,} {\bf 26}, 187-212.

\bibitem[DKN07]{Dalang:05} {\sc{Dalang, R.C., Khoshnevisan, D., and Nualart, E.}} (2007), Hitting probabilities for systems of non-linear stochastic heat equations with additive noise, {\em ALEA}, {\bf 3}, 231-271.

\bibitem[DKN09]{Dalang2:05} {\sc{Dalang, R.C., Khoshnevisan, D., and Nualart, E.}} (2009), Hitting probabilities for systems of non-linear stochastic heat equations with multiplicative noise, {\em Probab. Th. Rel. Fields}, {\bf 144}, 371-427.

\bibitem[DN04]{Dalang:04} {\sc{Dalang, R.C. and Nualart, E.}} (2004), Potential theory for hyperbolic SPDEs,
{\it The Annals of Probability}, {\bf{32}}, 2099-2148.

\bibitem[DQ10]{DQ10} {\sc{Dalang, R.C. and Quer-Sardanyons, L.}} (2011), Stochastic integrals for spde's: a comparison,
{\it{Expositiones Mathematicae}}, {\bf {29}}, 67-109.

\bibitem[DSS10]{Dalang:10} {\sc{Dalang, R.C. and Sanz-Sol\'e, M.}} (2010), Criteria for hitting probabilities with applications to systems of stochastic wave equations, {\em Bernoulli}, {\bf 16}, 1343-1368.

\bibitem[DSS11]{Dalang:11} {\sc{Dalang, R.C. and Sanz-Sol\'e, M.}} (2011), Hitting probabilities for systems of stochastic waves, {\it Preprint}.

\bibitem[F99]{Fournier:99} {\sc{Fournier, N.}} (1999), Strict positivity of the density for a Poisson
driven S.D.E., {\it Stochastic and Stochastic Reports}, {\bf{68}},
1-43.

\bibitem[K02]{Khoshnevisan:02}
{\sc{Khoshnevisan, D.}} (2002), {\it Multiparameter processes. An introduction to random fields}, Springer-Verlag.

\bibitem[MMS01]{Marquez:01} {\sc{M\'arquez-Carreras, D., Mellouk, M. and Sarr\`a, M.}} (2001), On stochastic partial differential equations
with spatially correlated noise: smoothness of the law, {\it
Stochastic Processes and their Applications}, {\bf{93}}, 269-284.

\bibitem[M82]{Metivier:82}
{\sc{M\'etivier, M.}} (1982), {\it Semimartingales}, de Gruyter.

\bibitem[MS99]{Millet:99} {\sc{Millet, A. and Sanz-Sol\'e, M.}} (1999), A stochastic wave equation in two
space dimension: smoothness of the law,
{\it The Annals of Probability}, {\bf{27}}, 803-844.

\bibitem[N06]{Nualart:95}
{\sc{Nualart, D.}} (2006), {\it The Malliavin calculus and related
topics}, Second Edition, Springer-Verlag.

\bibitem[N98]{Nualart:98}
{\sc{Nualart, D.}} (1998), Analysis on Wiener space and
anticipating stochastic calculus, {\it Ecole d'Et\'e de
Probabilit\'es de Saint-Flour XXV, Lect. Notes in Math.}
{\bf{1690}}, Springer-Verlag, 123-227.

\bibitem[NQ07]{Nualart:07}
{\sc{Nualart, D. and Quer-Sardanyons, L.}} (2007), Existence and smoothness of the density for spatially homogeneous SPDEs,
\textit{Potential Analysis}, {\bf 27}, 281-299.

\bibitem[EN10]{Nualart:09} {\sc{Nualart, E.}} (2010), On the density of systems of non-linear spatially homogeneous SPDEs,
{\it Stochastics and Stoch.~Rep.} (to appear), http://arxiv.org/abs/0908.4587.

\bibitem[R91]{Revuz:91}
{\sc{Revuz, D. and Yor, M.}} (1991), {\it Continuous Martingales
and Brownian Motion}, Springer-Verlag.

\bibitem[S05]{Sanz:05} {\sc{Sanz-Sol\'e, M.}} (2005), {\it Malliavin calculus with applications to stochastic partial
differential equations}, EPFL Press.

\bibitem[SS00]{Sanz:00} {\sc{Sanz-Sol\'e, M. and Sarr\`a, M.}} (2000), Path Properties of a Class of Gaussian Processes with Applications
to SPDE's, {\it Canadian mathematical Society Conference Proceedings}, {\bf{28}}, 303-316.

\bibitem[SS02]{Sanz:02} {\sc{Sanz-Sol\'e, M. and Sarr\`a, M.}} (2002), H\"older continuity for the stochastic
heat equation with spatially correlated noise,
{\it Seminar on Stochastic Analysis, Random Fields ans Applications, III (Ascona, 1999), Progr. Prob.}, {\bf{52}}, 259-268.

\bibitem[S70]{Stein:70}{\sc{Stein, E.M.}} (1970), {\it Singular integrals and differential properties of functions},
Princeton University Press, Princeton.

\bibitem[W86]{Walsh:86} {\sc{Walsh, J.B.}} (1986), An Introduction to Stochastic Partial Differential
Equations, {\it Ecole d'Et\'e de Probabilit\'es de Saint-Flour
XIV, Lect. Notes in Math.}, {\bf{1180}}, Springer-Verlag, 266-437.

\bibitem[W84]{Watanabe:84} {\sc{Watanabe, S.}} (1984), Lectures on Stochastic Differential Equations and Malliavin Calculus, {\it Tata Institute of Fundamental Research Lectures on Math. and Physics}, {\bf 73}, Springer-Verlag, Berlin.

\bibitem[X09]{Xiao:09} {\sc{Xiao, Y.}} (2009), Sample path properties of anisotropic Gaussian random fields, {\it In: A Minicourse on Stochastic Partial Differential Equations, (D. Khoshnevisan and F. Rassoul-Agha, editors), Lecture Notes in Math.},  {\bf 1962}, Springer, New York, 145--212.
\end{thebibliography}
\end{document}